\documentclass[11pt, a4paper]{article}
\usepackage{amsmath,amsthm,amsfonts,amssymb,bbm,enumerate,fullpage,graphicx,longtable,subcaption,xcolor}
\usepackage{hyperref}
\usepackage[capitalise]{cleveref}
\usepackage[sorting=nyt,style=alphabetic,backend=bibtex]{biblatex}
\newcommand{\pr}[1]{\mathbb{P}\!\left(#1\right)}
\newcommand{\E}[1]{\mathbb{E}\!\left[#1\right]}
\newcommand{\estart}[2]{\mathbb{E}_{#2}\!\left[#1\right]}
\newcommand{\prstart}[2]{\mathbb{P}_{#2}\!\left(#1\right)}

\newcommand{\prcond}[3]{\mathbb{P}_{#3}\!\left(#1\;\middle\vert\;#2\right)}
\newcommand{\econd}[2]{\mathbb{E}\!\left[#1\;\middle\vert\;#2\right]}

\newcommand{\Pb}{\mathbb{P}}
\newcommand{\Var}[1]{\mathrm{Var}\!\left(#1\right)}

\newcommand{\dGH}[1]{d_{GH}\!\left(#1\right)}
\newcommand{\dGHP}[1]{d_{GHP}\!\left(#1\right)}

\newtheorem{theorem}{Theorem}[section]
\newtheorem{lem}[theorem]{Lemma}
\newtheorem{prop}[theorem]{Proposition}
\newtheorem{cor}[theorem]{Corollary}
\newtheorem{defn}[theorem]{Definition}
\newtheorem{assn}[theorem]{Assumption}

\newtheorem{rmk}[theorem]{Remark}

\def\N{\mathbb{N}}
\def\Z{\mathbb{Z}}

\def\R{\mathbb{R}}

\def\F{\mathcal{F}}

\def\CC{\mathfrak{C}}

\def\B{B^{\text{exc}}}
\def\X{X^{\text{exc}}}
\def\Bb{B^{\text{br}}}

\renewcommand{\phi}{\varphi}
\renewcommand{\epsilon}{\varepsilon}

\def\T{\mathcal{T}}

\def\cadlag{c\`{a}dl\`{a}g }
\def\Loop'{\textsf{Loop'}}
\def\Loop{\textsf{Loop}}
\def\Dec{\textsf{Dec}}
\def\C{C}

\def\diam{\textsf{Diam}}

\allowdisplaybreaks
\def\Ha{\mathbb{H}_{\alpha}}
\def\tT{\tilde{T}}
\def\tnu{\tilde{m}}
\def\td{\tilde{d}}
\def\trho{\tilde{\rho}}
\def\nut{{\nu}_{\T}}

\def\Zmq{Z_{m,q}}
\newcommand{\Zr}{\overset{\leftarrow}{Z}}
\newcommand{\mur}{\overset{\leftarrow}{\mu}}

\begin{document}

\title{Scaling limit of critical percolation clusters\\on hyperbolic random half-planar triangulations\\and the associated random walks}
\author{Eleanor Archer\thanks{Modal’X, UMR CNRS 9023, UPL, Universit\'{e} Paris Nanterre, 200 avenue de la R\'{e}publique, 92000 Nanterre, France, eleanor.archer@parisnanterre.fr}
and David A.\ Croydon\thanks{Research Institute for Mathematical Sciences, Kyoto University, Japan, croydon@kurims.kyoto-u.ac.jp}}
\maketitle

\begin{abstract}
We show that the Gromov-Hausdorff-Prohorov scaling limit of a critical percolation cluster on a random hyperbolic triangulation of the half-plane is the Brownian continuum random tree. As a corollary, we obtain that a simple random walk on the critical cluster rescales to Brownian motion on the continuum random tree.\\
\textbf{AMS 2020 MSC:} 60K37 (primary), 05C10, 05C81, 60K35, 82B41, 82B43.\\
\textbf{Keywords and phrases:} random planar map, critical percolation cluster, continuum random tree, random walk in random environment.
\end{abstract}

\section{Introduction}

In this paper, we study the scaling limits of critical Bernoulli percolation clusters on hyperbolic random planar maps. Specifically, we work on the half-planar triangulations of \autocite{angelray2015}, which are loopless triangulations of the upper half-plane satisfying translation invariance and a domain Markov property. These triangulations have a boundary consisting of a simple doubly-infinite path, and can be embedded into the upper half-plane so that the boundary aligns with $\R \times \{0\}$. In \autocite{angelray2015}, the authors characterised the laws of all such triangulations and showed that they can be indexed by a parameter $\alpha \in [0,1)$. In the regime $\alpha \in (2/3, 1)$, these maps are of hyperbolic flavour; several typical properties (exponential volume growth, anchored expansion and transience of the associated random walk) were established by Ray in \autocite{ray2014geometry}. We will denote the set of the laws of these maps by $(\Ha)_{\alpha \in (\frac{2}{3},1)}$.\label{fa Ha}

By convention, the maps generated according to $(\Ha)_{\alpha \in (\frac{2}{3},1)}$ are rooted, and in particular contain a root edge and an adjacent root vertex, both of which lie on the boundary. We denote the root vertex by $\rho$. We will consider critical site percolation for maps with law $\Ha$, in which vertices are coloured black with probability $p$, and coloured white with probability $1-p$, independently of the state of other vertices. Moreover, we suppose a white--black--white boundary condition, meaning that the root vertex is coloured black and all other boundary vertices are coloured white. This is for technical convenience and we believe that the results of this paper should also hold under other appropriate boundary conditions. Ray showed that, $\Ha$-a.s., the critical percolation probability of the generated map, that is, the infimum over $p$ such that there exists an infinite black cluster of the percolation cluster (almost-surely for the relevant percolation measure), is given by
\begin{equation}\label{pcdef}
p_c = \frac{1}{2}\left(1 - \sqrt{3-\frac{2}{\alpha}} \right),
\end{equation}
see \autocite[Theorem 2.7]{ray2014geometry}. Under the white--black--white boundary condition, we will consider the law of the root cluster conditioned to be large, viewed as a metric-measure space, with our first main result being as follows. Precise definitions of the various terms that appear in the statement will be given below.

\begin{theorem}\label{thm:main scaling lim}
Fix $\alpha \in (\frac{2}{3}, 1)$. Given $n \geq 1$, let $\C_n$ denote the cluster that contains the root vertex (denoted $\rho$) of critical percolation on a map with law $\Ha$ with white--black--white boundary condition, conditioned to have at least $n$ vertices. Let $d_n$ denote graph distance on $\C_n$, and $\nu_n$ denote counting measure on the vertices of $\C_n$. It is then the case that there exists a constant $\gamma \in (0,\infty)$, depending only on $\alpha$, such that: as $n \to \infty$,
\[\left(\C_n, \frac{ 1}{\gamma\sqrt{n}}d_n, \frac{1}{n}\nu_n, \rho\right) \overset{(d)}{\rightarrow} (\T, d_\T, \nu_\T, \rho_\T)\]
in distribution with respect to the pointed Gromov-Hausdorff-Prohorov topology, where the limiting space $(\T, d_\T, \nu_\T, \rho_\T)$ is the Brownian continuum random tree, conditioned to have mass at least $1$.
\end{theorem}

Note that this is an \textit{annealed} result, with the convergence in distribution holding under the annealed law of the cluster (in other words, averaged over both the law of the map and the percolation configuration). We highlight that conditioning the annealed law as we do affects not only the percolation configuration, but also the structure of the map generated according to $\Ha$. Our techniques do not extend to a quenched result (i.e.\ one that holds with respect to the relevant percolation measure, $\Ha$-a.s.). We further remark that this result is closely related to \autocite[Open Question 12.12]{curien2019peeling}, which asks about the scaling limit of critical clusters on a more general class of hyperbolic random planar maps.

In addition, as our second main result, we establish convergence of the simple random walk on $\C_n$ to Brownian motion on the continuum random tree. The latter object was first defined informally by Aldous \autocite[Section 5.2]{AldousCRTII}, formally constructed by Krebs \autocite{krebs1995brownian} and studied in detail in \autocite{DavidCRT}. Again, we postpone exact definitions of the various terms appearing in the statement until later in the article. Moreover, we highlight that this is also an annealed result (now with three layers of randomness in the discrete case); the laws of the random walk and limiting diffusion are averaged over the environments in which they move. We will restate the result more precisely as Theorem \ref{RWprecise} below. Note that it would be quite straightforward to extend the following theorem to include convergence of the associated mixing times, transition densities and local times; this is addressed in \cref{rmk:transition densities etc}.

\begin{theorem}\label{thm:main RW}
In the setting of \cref{thm:main scaling lim}, let $X^{\C_n}$ denote the discrete-time simple random walk on $\C_n$, started at $\rho$. Moreover, let $(X^\T_t)_{t \geq 0}$ denote Brownian motion on $(\T, d_\T, \nu_\T)$, started at $\rho_\T$. It is then the case that, with respect to the natural extension of the Gromov-Hausdorff-topology incorporating c\`{a}dl\`{a}g paths, one can add to the statement of \cref{thm:main scaling lim} that, as $n\rightarrow\infty$,
\[\left(X^{\C_n}_{\lfloor \theta n^{3/2}t \rfloor}\right)_{t \geq 0} \overset{(d)}{\rightarrow} \left(X^\mathcal{T}_t\right)_{t \geq 0},\]
where $\theta\in(0,\infty)$ is a constant depending only on $\alpha$.
\end{theorem}

The strategy for proving \cref{thm:main scaling lim} relies on a peeling exploration of the percolation interface that gives a decomposition of $\C_n$ into a series of small blocks glued along a tree structure. In fact the tree is a two-type critical Galton-Watson tree, which is well-known to rescale to the continuum random tree by a result of Miermont \autocite{miermont2008invariance}. Moreover, it will essentially follow from results of Angel and Ray \autocite{angelray2015, ray2014geometry} that the blocks are sufficiently small that they all vanish in the scaling limit, and it therefore follows from the results of Archer \autocite{archer2020random} that all the critical exponents agree with those of critical Galton-Watson trees. (The latter results are written for one-type trees, but the same principles apply in the two-type case.) By mimicking a proof of Stufler (who showed a similar result for one-type trees \autocite[Theorem 6.60]{StuflerTreeSurvey}), we can use these facts to show that the scaling limit of $\C_n$ is a constant multiple of the scaling limit of the underlying two-type tree. We note that similar results regarding scaling limits of decorated one-type trees (also treating the stable case) were obtained in \autocite{fleurat2023phase}.

The approach used to prove \cref{thm:main scaling lim} actually applies more generally, and can be used to show that $\C_n$ equipped with an effective resistance metric and the degree measure on its vertices also rescales to the continuum random tree (see \cref{sctn:RW} below). \cref{thm:main RW} then follows as a direct consequence of a general result of \autocite{DavidResForms}, which expresses convergence of stochastic processes in terms of convergence of these quantities. Since we follow this route to proving \cref{thm:main RW}, it should also be possible to apply related results in the literature to deduce convergence of the random walk transition densities, as in \autocite{CHllt}, and mixing times, as in \autocite{CHKmix}. We have not checked, but it is also reasonable to believe that the local time convergence criteria of \autocite{Noda} also holds in this setting.

\begin{rmk}
Percolation configurations on the maps $\Ha$ in the hyperbolic regime were also considered by Ray in the paper \autocite{ray2014geometry} using the same peeling exploration process, but there he investigated questions of a different nature. In particular, he obtained the value of $p_u$, the critical probability for uniqueness of the infinite cluster (as well as that of $p_c$ mentioned above), results on the spread of clusters along the boundary and properties of the ends (rays to infinity) of the clusters. Some interesting questions for supercritical percolation in this regime are left open, such as anchored expansion and positive random walk speed (see \autocite[Open Question 6.2]{ray2014geometry} and the discussion around it). As noted above, these properties have been established for the unpercolated maps \autocite{ray2014geometry, angel2016random}.
\end{rmk}

The results of \cref{thm:main scaling lim} and \cref{thm:main RW} contrast strongly with the behaviour of critical percolation clusters on uniform random planar maps. On the uniform infinite planar triangulation (UIPT), Curien and Kortchemski \autocite{LooptreePerc} showed using different techniques that the boundary of a critical cluster has the form of a stable looptree, meaning that the cluster once again has an underlying tree structure, but that the individual blocks do not vanish in the scaling limit. In this case, the cluster is conjectured to rescale to the $\frac{7}{6}$-stable map (see \autocite[Section 5.4]{BerCurMierPerconTriang}), which is very different to the continuum random tree. Using similar peeling techniques to those of this paper, Richier \autocite{RichierIICUIHPT} also showed similar results (stable looptree type results) for the boundary of the incipient infinite cluster on the uniform infinite half-planar triangulation. This corresponds to $\Ha$ when $\alpha=\frac{2}{3}$. In the case when $\alpha < \frac{2}{3}$, the maps themselves are similar in structure to finite variance critical Galton-Watson trees (see \autocite{ray2014geometry}) and so there is not expected to be a non-trivial percolation phase transition. (Cf. the recent result concerning Fortuin-Kasteleyn decorated planar maps with parameter $q>4$ of \autocite{feng2023triviality}.)

Note further that the appearance of the continuum random tree as a critical scaling limit on hyperbolic random planar maps should not come as a surprise; similar results have been obtained for oriented percolation on supercritical causal maps \autocite[Theorem 6]{marchand2023percolation}. The continuum random tree also arises as the scaling limit of Brownian bridges in hyperbolic spaces \autocite{chen2017long}; this latter result is somewhat similar in spirit to the convergence of the peeling exploration process. Moreover the continuum random tree is expected to appear more generally as a scaling limit for a much wider class of hyperbolic random planar maps (see \autocite[Open Question 12.12]{curien2019peeling} and \autocite[Section 5]{marchand2023percolation}), and as a scaling limit for large critical percolation clusters on high-dimensional Euclidean lattices (see \autocite{HS1,HS2} for rigorous results in this direction). In particular we anticipate that similar results should hold for critical percolation on the planar stochastic hyperbolic infinite planar triangulations of \cite{curien2016planar}, which are the full-plane analogues of the half-planar maps considered in the present paper.

Motivated by Kesten's pioneering work concerning the subdiffusivity of simple random walk on critical percolation clusters in two dimensions \autocite{KestenIICtree}, a great deal of effort has been devoted to understanding the scaling limit of such processes. In the high-dimensional case, work such as \autocite{BCF2, BCF1, Hey, KN} has brought closer the prospect of showing that the limiting process is Brownian motion on the continuum random tree, as appears in the statement of \cref{thm:main RW}. Moreover, general work built on the connections between random walks and electrical networks give techniques that enable a wider class of critical percolation models to be handled, see \autocite{CrHaKu,DavidResForms}. The conclusion of \cref{thm:main RW} not only showcases a new application of the methodology of \autocite{DavidResForms}, adding to models such as critical Galton-Watson trees and the critical Erd\"os-R\'enyi graph, but also gives further evidence in support of the universality of Brownian motion on the continuum random tree as the scaling limit of random walks on critical percolation clusters in (suitably-interpreted) high-dimensional spaces.

We remark that the use of the underlying tree structure in the proof of \cref{thm:main RW} is crucial. By contrast, most random planar map models are non-tree like, and the associated random walks are notoriously difficult to study. As is the case in \cref{thm:main RW} (and maybe more pertinently in \autocite{KestenIICtree}), these random walks are generally believed to be subdiffusive due to the fractal properties of random planar maps. This has been established quantitatively in \autocite{benjamini2013simple} for the uniform infinite planar quadrangulation, and in \autocite{curien2021infinite} for a much wider class of infinite discrete stable maps in the form of an upper bound of $\frac{1}{3}$ on the displacement exponent (see also \autocite[Fig.\ 1]{curien2021infinite} for similar exponents for other related random map models). In the special case of the UIPT, it was shown in \autocite{gwynne2021random} and \autocite{gwynne2020anomalous} that the true exponent is $\frac{1}{4}$. Their result also holds for some related models in the Brownian universality class.

The remainder of the paper is organised as follows. In \cref{sctn:background}, we give some preliminaries on Galton-Watson trees, looptrees, random walks and the Gromov-Hausdorff-Prohorov topology. In \cref{sctn:the model}, we formally introduce percolation on a map with law $\Ha$ and explain how the boundary of the root cluster can be encoded by a peeling exploration. In particular, we use this to establish that (up to a small error concerning the structure close to the root) the law of the cluster boundary is encoded by a two-type Galton-Watson tree, and explain how the cluster can be constructed from blocks glued along this tree structure. In \cref{sctn:dec tree limits}, we consider a general model of decorated two-type Galton-Watson trees and establish that, under certain conditions on the decorations, their scaling limit is the continuum random tree. We then tie everything together in \cref{sctn:main result proof} by introducing a coupling between the root cluster and a decorated tree model. We verify that the decorated tree model indeed satisfies the assumptions required to apply the results of \cref{sctn:dec tree limits}, and that, under the coupling, the difference between the two structures is negligible in the scaling limit, thus proving \cref{thm:main scaling lim}. Finally, in \cref{sctn:RW}, we explain how the techniques of the previous sections can be applied to study random walks on $\C_n$, and in particular prove \cref{thm:main RW}. To facilitate reading, in \cref{sec:notation}, we present a table that summarises the most important notation of the article.
\bigskip

\noindent
\textbf{Acknowledgements.} EA would like to thank Asaf Nachmias for helpful conversations and for bringing the paper \autocite{angelray2015} to her attention. Both authors would also like to thank Ben Hambly for his hospitality during visits they made to Oxford in 2022, during which time this work was initiated. The research of EA was supported by the ANR grant ProGraM (ANR-19-CE40-0025). DC was supported by JSPS Grant-in-Aid for Scientific Research (C) 19K03540 and the Research Institute for Mathematical Sciences, an International Joint Usage/Research Center located in Kyoto University.

\section{Background}\label{sctn:background}

In this section, we introduce a variety of preliminaries, including background on random trees and looptrees, random walks and excursions, Gromov-Hausdorff-type topologies and conditioning probability measures.

\subsection{Galton-Watson trees}\label{sctn:GW background}

Here, we briefly introduce Galton-Watson trees, starting with the Ulam-Harris labelling convention for discrete trees, following the formalism of \autocite{Neveu}. Define the set
\[\mathcal{U}=\cup_{n=0}^{\infty} {\N}^n.\]
By convention, ${\N}^0=\{ \emptyset \}$. If $u=(u_1, \ldots, u_n)$ and $v=(v_1, \ldots, v_m) \in \mathcal{U}$, we let $uv= (u_1, \ldots, u_n, v_1, \ldots, v_m)$ denote the concatenation of $u$ and $v$.

\begin{defn} A plane tree $T$ is a finite subset of $\mathcal{U}$ such that:
\begin{enumerate}[(i)]
\item $\emptyset \in T$;
\item if $v \in T$ and $v=uj$ for some $j \in \N$, then $u \in T$;
\item for every $u \in T$, there exists a number $k_u(T) \geq 0$ such that $uj \in T$ if and only if $1 \leq j \leq k_u(T)$.\label{fa ku}
\end{enumerate}
\end{defn}

The classical model of a Galton-Watson tree $T$ is built from a collection $(k_u)_{u \in \mathcal{U}}$ of independent and identically-distributed (i.i.d.)\ random variables taking values in $\{0,1,\dots\}$. In particular, $T$ is the unique plane tree such that $k_u(T)=k_u$ for all $u \in T$. (We assume that the reader is familiar with Galton-Watson trees, but if this is not the case, see \autocite[Section 1.2]{LeG2005randomtreesandapplications} for some initial background.) Here, we will mainly work with a generalisation of this model, namely a \textit{two-type} Galton-Watson tree, meaning that we have one offspring distribution $\mu_{\bullet}$ at even generations, and another distribution $\mu_{\circ}$ at odd generations. In other words, the root, which we will denote by $\rho$, reproduces with offspring distribution $\mu_{\bullet}$, its children reproduce with offspring distribution $\mu_{\circ}$, and the distributions continue to alternate at odd and even generations. Specifically, for any finite tree $t$,
\begin{equation*}\label{fa mu t}
\pr{T=t} = \prod_{u \in t_{\bullet}} \mu_{\bullet}(k_u) \prod_{u \in t_{\circ}} \mu_{\circ}(k_u),
\end{equation*}
where $t_{\bullet}$ denotes the set of vertices at even generations in $t$, and $t_{\circ}$ denotes the set of vertices at odd generations in $t$. We will sometimes refer to the vertices in $t_{\bullet}$ as ``black vertices'', and the vertices in $t_{\circ}$ as ``white vertices''. Moreover, we will henceforth assume that $\mu_{\circ}$ is supported on the positive integers, or in other words that $\mu_{\circ}(0)=0$.

\subsection{Kesten's tree}\label{kessec}

In this paper we will be only interested in two-type Galton--Watson trees that are critical, meaning that $\E{\mu_{\circ}}\E{\mu_{\bullet}}=1$. In this case we can also define a version that is conditioned to survive using a construction of Kesten as follows. Let $\hat{\mu}_{\circ}$ and $\hat{\mu}_{\bullet}$ denote size-biased versions of $\mu_{\circ}$ and $\mu_{\bullet}$ respectively (i.e.\ $\hat{\mu}_{\circ}(i)\propto i\mu_\circ(i)$ and  $\hat{\mu}_{\bullet}(i)\propto i\mu_\bullet(i)$, respectively). We construct an infinite tree recursively as follows.

\begin{defn}\label{def:Kesten tree}\autocite{KestenIICtree}.
The \textbf{Kesten tree} $T_{\infty}$ associated with the critical pair $(\mu_{\bullet}, \mu_{\circ})$ is a two-type Galton--Watson tree satisfying the following.
\begin{itemize}
\item Individuals are either normal or special.
\item The root of $T_{\infty}$ is special.
\item A normal individual produces only normal individuals according to $\mu_{\bullet}$ if it is at an even generation, and according to $\mu_{\circ}$ if it is at an odd generation.
\item A special individual produces individuals according to $\hat{\mu}_{\bullet}$ if it is at an even generation, and according to $\hat{\mu}_{\circ}$ if it is at an odd generation. Of these, one of them is chosen uniformly at random to be special, and the rest are normal.
\end{itemize}
\end{defn}
\noindent
Note that the special vertices form an infinite one-ended spine. For each $h \geq 0$, we let $v_h$ denote the special vertex at distance $h$ from the root.

Suppose that $\E{\mu_{\circ}}\E{\mu_{\bullet}}=1$ and let $T$ be an unconditioned two-type Galton--Watson tree with these offspring laws. We will use the following absolute continuity relation between $T$ and $T_{\infty}$, which is a special case of  \autocite[Lemma 1]{miermont2008invariance}. Let $F_T$ be a function $T \to \R$ such that for each $v \in T$, the value of $F_T(v)$ does not depend on any of the descendants of $v$. Then, for any $h \geq 0$,
\begin{equation}\label{eqn:abs cont Kesten}
\E{\sum_{\substack{v \in T: \\ d(\rho, v)=h}}F_T(v)} = \estart{F_T(v_h)}{\infty},
\end{equation}
where we denote by $d$ the graph distance on $T$, so that $d(\rho,v)=h$ means that $v$ is in generation $h$ of $T$.

\subsection{Looptrees}\label{sctn:looptrees}

In addition to trees, we also need the notion of looptrees. Given a two-type plane tree $t$ with no white leaves, $\Loop(t)$ has vertex set $t_{\bullet}$ and is obtained by drawing a loop around each white vertex through its children and parent according to the natural planar embedding. (Note that, although the notion of an infinite looptree does appear in the literature, in this work all the graphs we consider will be finite.) The procedure for constructing $\Loop(t)$ is described more formally as follows. Take a white vertex $u$; its children are $(ui)_{i=1}^{k_u}$. Additionally let its parent be $u0$. Then draw an edge from $ui$ to $u(i+1)$ for each $0 \leq i < k_u$, and additionally from $uk_u$ to $u_0$. Repeat this procedure for every $u \in t_{\circ}$. Then $\Loop(t)$ is the collection of vertices in $t_{\bullet}$, joined by the newly constructed edges, and rooted at the root of $t$. See \cref{fig:contour coding looptree} for an example (the root is the vertex labelled $0$).

\subsection{The continuum random tree}\label{sctn:CRT def}

We now formally define the (Brownian) continuum random tree (CRT) appearing in \cref{thm:main scaling lim}. The CRT is a canonical example of a random real tree and was first introduced by Aldous in \autocite{AldousCRTI}. We will be interested in the CRT conditioned to have mass at least $1$, which is a slight variation of Aldous' original definition. For background, we note that the CRT possesses many fractal and self-similarity properties, but we do not go into detail here. An extensive account was already given in the three papers \autocite{AldousCRTI}, \autocite{AldousCRTII} and \autocite{AldousCRTIII} by Aldous, including several possible constructions. Since then, a vast literature has developed on the topic, which an interested reader can readily explore.\label{fa exc defs}

Perhaps the quickest way to define the object of interest here is to appeal to its coding by an associated contour function. In the case of the continuum random tree, this role is played by a Brownian excursion. Specifically, write $N$ for the It\^{o} excursion measure (see \autocite[Section 2]{LeGIto} for a definition) and let $\B$ be the Brownian excursion conditioned to have lifetime (duration) $\zeta$ at least 1, i.e.\ with law given by $N(\cdot\:\vline\:\zeta\geq 1)$. The continuum random tree is then constructed as follows. Conditionally on the lifetime of $\B$ being equal to $\zeta$ for some $\zeta\geq 1$, we first define a pseudo-distance on $[0,\zeta]$ by setting
\[d_{\B}(s,t)\equiv d_{\B}(t,s) = \B(s) + \B(t) - 2 \inf_{s \leq r \leq t} \B(r)\]
whenever $s \leq t$. We then define an equivalence relation on $[0,\zeta]$ by saying $s \sim t$ if and only if $d_{\B}(s,t) = 0$, and set
\[(\T,d_\T) := ([0,\zeta]/ \sim, d_{\B}).\]\label{eqn:CRT def}
A simulation is shown in \cref{fig:CRT}. We define $\pi$ to be the canonical projection from $[0,\zeta]$ to $\T$, and suppose the root $\rho_\T$ of $\T$ is given by $\pi(0)$. Moreover, we define a Borel measure $\nu_\T$ on $\T$ as the push-forward of Lebesgue measure on $[0,\zeta]$ by $\pi$. In this way, we have defined the random compact metric-measure space $(\T,d_\T,\nu_\T,\rho_\T)$ appearing in the statement of \cref{thm:main scaling lim}.

\begin{figure}[t]
\begin{subfigure}{0.47\textwidth}
\begin{center}
\includegraphics[width=0.9\textwidth]{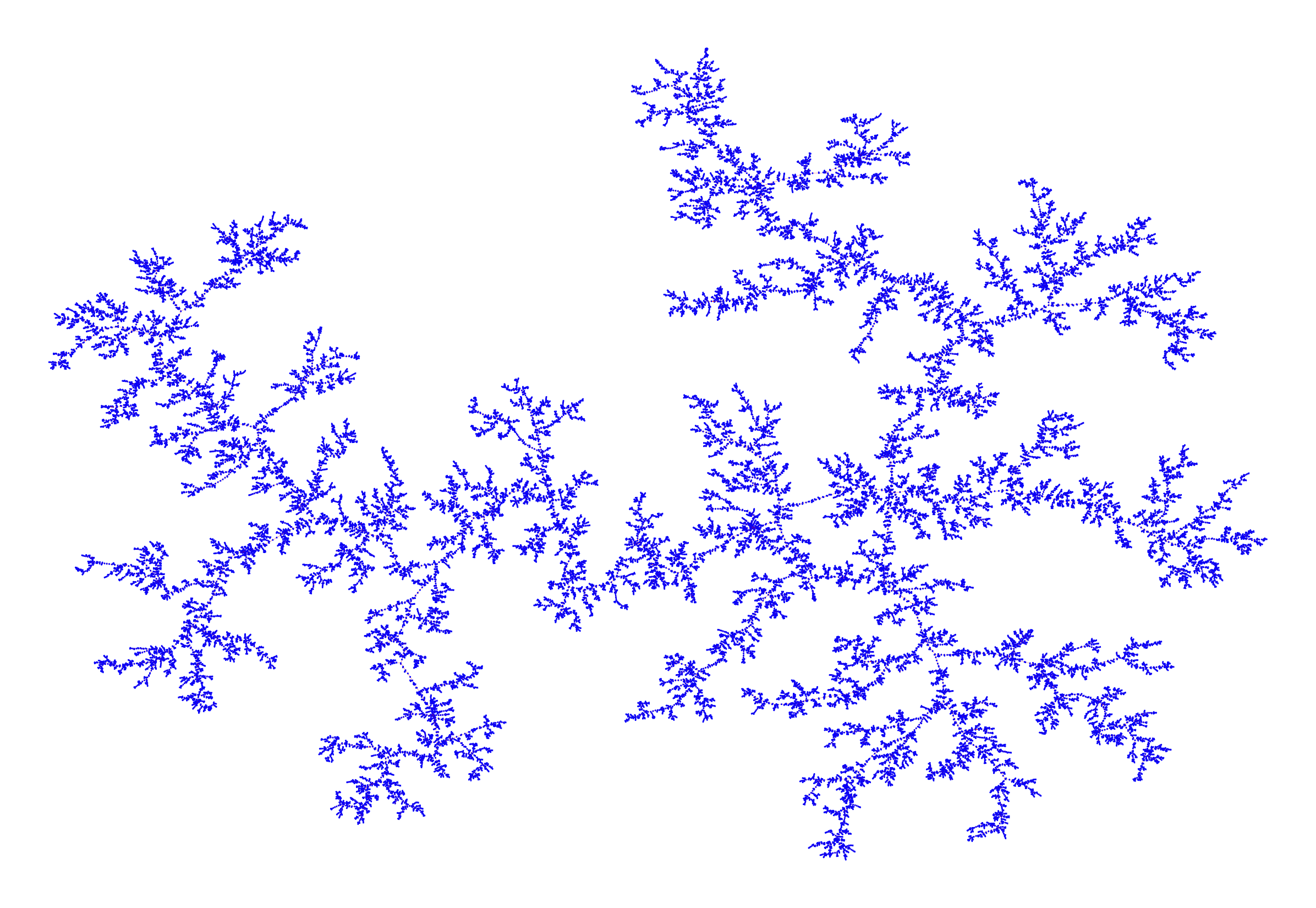}
\end{center}
\end{subfigure}
\begin{subfigure}{0.47\textwidth}
\begin{center}
\includegraphics[width=0.9\textwidth]{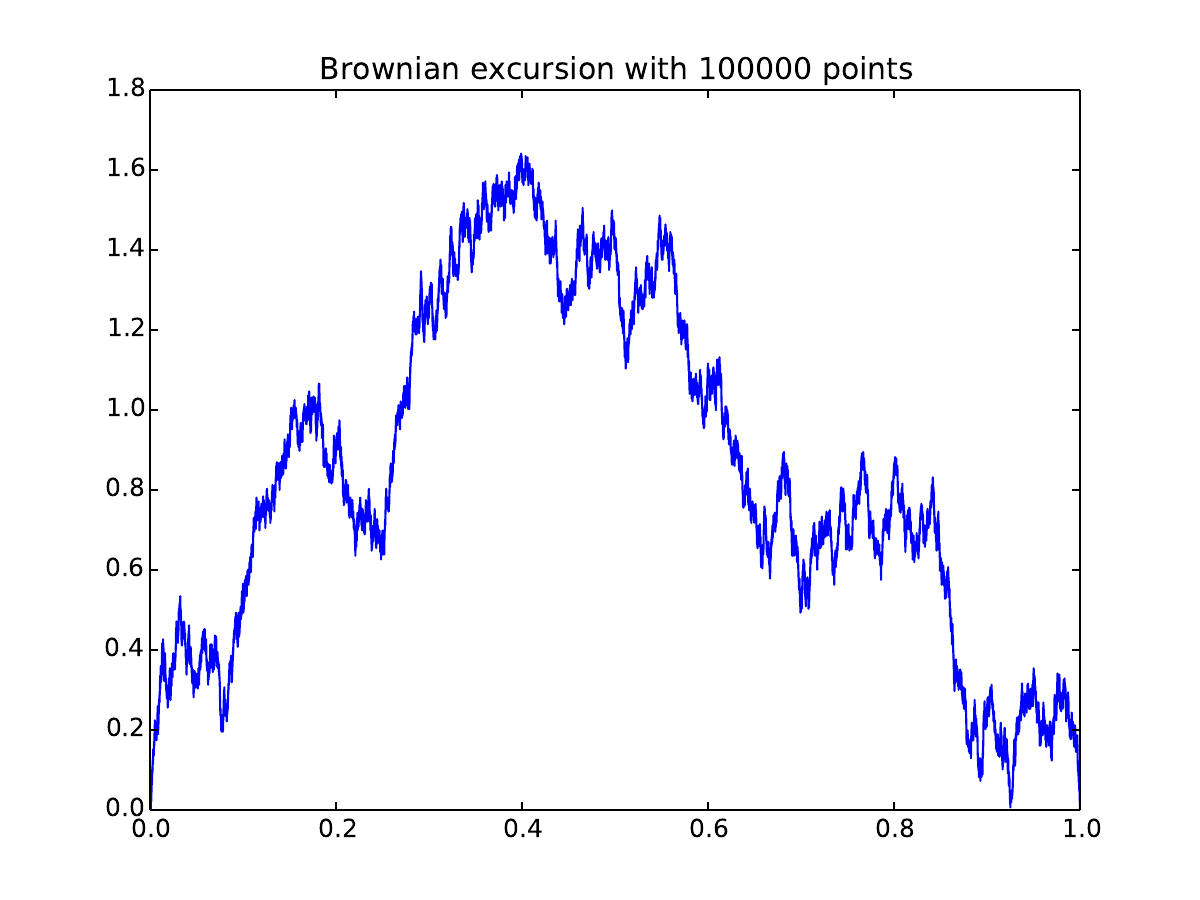}
\end{center}
\end{subfigure}
\caption{A simulation of the CRT and the associated excursion, both by Laurent M\'enard.}\label{fig:CRT}
\end{figure}

Although we will not need the exact statement in what follows, to provide further context for \cref{thm:main scaling lim}, we remark that $(\T,d_\T,\nu_\T,\rho_\T)$ arises naturally as a scaling limit of discrete critical Galton-Watson trees. In particular, if $\T_n$ is a Galton-Watson tree with non-trivial, critical (mean one) offspring distribution with finite variance $\sigma^2$, conditioned to have at least $n$ vertices, $d_{\T_n}$ is the graph distance on $\T_n$, $\nu_{\T_n}$ is the counting measure, and $\rho_{\T_n}$ is the initial ancestor of $\T_n$, then
\[\left(\T_n,\frac{\sigma}{2\sqrt{n}}d_{\T_n},\frac1n\nu_{\T_n},\rho_{\T_n}\right)\overset{(d)}{\rightarrow} \left(\T,d_\T,\nu_\T,\rho_\T\right)\]
in distribution as $n \rightarrow \infty$ with respect to the Gromov-Hausdorff-Prohorov topology introduced below \autocite{AldousCRTI, AldousCRTII, LeG2005randomtreesandapplications}.

We note that sometimes different normalisations of the CRT are used in the literature, resulting in different scaling factors to those appearing above; here we are following the convention of Le Gall \autocite{LeG2005randomtreesandapplications}.

\subsection{Coding trees and looptrees}\label{sctn:coding}

\begin{figure}[t]
\includegraphics[width=\textwidth]{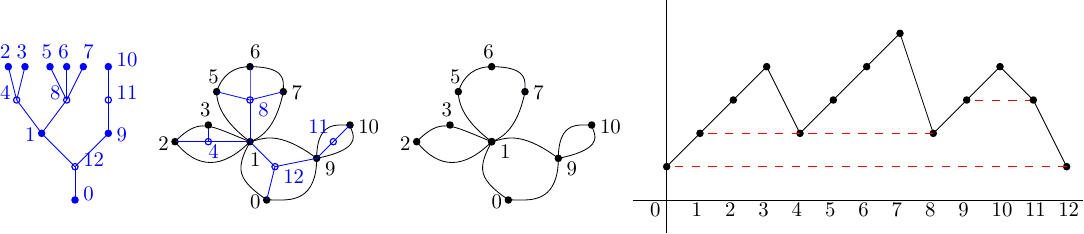}
\caption{A two-type tree, its looptree, and associated excursion.}\label{fig:contour coding looptree}
\end{figure}

As for the continuum random tree, it is possible to code two-type plane trees and looptrees using excursions. In particular, the kinds of excursions we now consider are discrete processes of the form $Z=(Z_i)_{i=0}^n$ with jump sizes $Z_{i}-Z_{i-1}$ in $\{1, -1, -2, -3, \ldots \}$, which also satisfy $Z_0=1$, $Z_i \geq 1$ for all $1 \leq i \leq n$, and $Z_n = 1$. We define an equivalence relation on $\{0, 1, \ldots, n\}$ by setting $i \sim j$ precisely if
\begin{equation}\label{eqn:contour equiv def}
Z_i = Z_j = \inf_{i \wedge j \leq k \leq i \vee j} Z_k,
\end{equation}
where we use $\wedge$ and $\vee$ to denote minimum and maximum, respectively. The looptree $L_Z$ is equal to the quotient space $\{0, 1, \ldots, n\} / \sim$. It is rooted at the vertex corresponding to (the equivalence class of) $0$. Again, see \cref{fig:contour coding looptree} for an example. Moreover, we let $T_Z$ be the unique two-type tree such that $L_Z = \Loop (T_Z)$; this can be constructed from $L_Z$ by adding a white vertex inside each loop of $L_Z$, adding an edge between each such white vertex to each of the black vertices on the corresponding loop, and then deleting the edges of the original loops. The root of $L_Z$ is also the root of $T_Z$.

This coding mechanism also assigns labels to the vertices of $T_Z$ and $L_Z$ in a natural way. Note that, apart from the root, each black vertex of the tree can be associated with the right endpoint of a positive jump in the coding function, and each white vertex can be associated with the right endpoint of a negative jump in the coding function. We therefore label each vertex in $T_Z$ with the time of the right endpoint of the corresponding jump, and similarly for all the vertices in $L_Z$. See \cref{fig:contour coding looptree} for an example.

Note that, since the labelling of $T_Z$ is readily deduced from the structure of the two-type tree itself (following the contour of the tree, black vertices are labelled on the first visit and white vertices are labelled on the last), $Z$ can easily be reconstructed from $T_Z$: $Z_i-Z_{i-1}$ is determined by the type and number of children of the vertex labelled $i$ in $T_Z$. This ensures that the correspondence between excursions of the form considered above and the set of rooted two-type plane trees with no white leaves is a bijection.

In this paper, we will specifically be interested in the case where $Z$ is obtained from a centred random walk excursion with some step distribution $\mu$ supported on $\{1, -1, -2, -3, \ldots \}$. To connect $Z$ to a two-type Galton-Watson tree, we make the following observation. Suppose $T$ is a two-type Galton-Watson tree with offspring distributions $\mu_{\bullet}$ and $\mu_{\circ}$ satisfying the following conditions.
\begin{enumerate}
\item The measure $\mu_0$ is supported on $\mathbb{N}$, i.e.\ $\mu_{\circ}(0)=0$.
\item The measure $\mu_{\bullet}$ is geometric, supported on the non-negative integers, so that $\mu_{\bullet}(k) = \mu_{\bullet}(0)(1-\mu_{\bullet}(0))^k$ for all $k \geq 0$.
\item It holds that $\E{\mu_{\bullet}}\E{\mu_{\circ}}=1$.
\end{enumerate}
\noindent
In addition let $\Zr = (\Zr_n)_{n \geq 0}$ be a centred random walk with $\Zr_0=1$ and increment distribution $\mur$ where
\begin{equation}\label{eqn:mu def general case background reversed}
\mur(i) = \begin{cases}
\mu_{\bullet}(0), & \text{ if } i=-1, \\
(1-\mu_{\bullet}(0))\mu_{\circ}(i), & \text{ if } i\geq 1, \\
0, & \text{ otherwise}.
\end{cases}
\end{equation}
Note that $\Zr$ is indeed centred, since
\[\E{\Zr_1-\Zr_0} = -\mu_{\bullet}(0) + \sum_{i \geq 1}i(1-\mu_{\bullet}(0))\mu_{\circ}(i) = \mu_{\bullet}(0)(\E{\mu_{\bullet}}\E{\mu_{\circ}}-1)=0.\]
Let $\tau_{0} = \inf\{n \geq 0: \Zr_n \leq 0\}$. Note that necessarily $\Zr_{\tau_{0}}=0$ and $\Zr_{\tau_{0}-1}=1$. We can therefore let $(Z_n)_{0 \leq n \leq \tau_{0} - 1}$ be the time reversal of $\Zr$ given by
\[Z_n = \Zr_{\tau_{0}-1-n}, \ \ \ \ 0 \leq n \leq \tau_{0}-1.\]
It follows by construction that $Z$ satisfies the properties listed above for coding a two-type tree.

\begin{prop}\label{prop:tree offspring RW}
Let $T_Z$ be the rooted two-type plane tree constructed from $Z$, defined from $\Zr$ as above. Let $T$ be a two-type Galton-Watson tree with offspring distributions $\mu_{\bullet}$ and $\mu_{\circ}$ as above. Then $T_Z \overset{(d)}{=} T$.
\end{prop}
\begin{proof}
Let $t$ be a finite two-type tree with no white leaves. Then, since the random walk encoding is a bijection, it follows by construction that
\[\pr{T_Z=t} = \prod_{i=0}^{|t|-1} \pr{\Zr_{i+1} - \Zr_{i} = f(v_i)} = \prod_{i=0}^{|t|-1} \mur(f(v_i)).\]
where $v_i$ is the tree vertex with label $|t| - 1 - i$, and
\[f(v_i) = \begin{cases} -1 &\text{ if } v_i \in t_{\bullet} \\
k_{v_i} &\text{ if } v_i \in t_{\circ}.\end{cases}\]
Note that, by definition, $\sum_{u \in t_{\bullet}}k_u = |t|_{\circ}$. By \eqref{eqn:mu def general case background reversed} we deduce that
\begin{align*}
\prod_{i=0}^{|t|-1} \mur(f(v_i)) &= \prod_{u \in t_{\bullet}} \overleftarrow{\mu}(-1) \prod_{u \in t_{\circ}} \overleftarrow{\mu}(k_u) \\
&=\prod_{u \in t_{\bullet}} \mu_{\bullet}(0) \prod_{u \in t_{\circ}} (1-\mu_{\bullet}(0))\mu_{\circ}(k_u) \\
&= (1-\mu_{\bullet}(0))^{\sum_{u \in t_{\bullet}} k_u} \prod_{u \in t_{\bullet}} \mu_{\bullet}(0) \prod_{u \in t_{\circ}} \mu_{\circ}(k_u) \\
&= \prod_{u \in t_{\bullet}} \mu_{\bullet}(k_u) \prod_{u \in t_{\circ}} \mu_{\circ}(k_u)\\
&=\pr{T=t},
\end{align*}
as required.
\end{proof}

\begin{rmk}
\begin{enumerate}
\item If $\tau_{0}=1$, then the corresponding tree consists only of the root vertex.
\item The reader may wonder why we didn't just define the trees and looptrees directly from excursions in the form of $\Zr$ and therefore remove the need to introduce a time-reversal. The reason is because the peeling process that we use to explore the cluster naturally encodes a random walk sharing the properties of $Z$. Therefore a time reversal needs to be applied at some stage of the argument and we found it simplest to work through the details in the present section.
\end{enumerate}
\end{rmk}

A consequence of this construction is that when $\mu_{\circ}$ has finite variance (as will be the case in this paper), there exists a constant $c$ such that
\begin{equation}\label{eqn:progeny asymp}
 \pr{|T| \geq x} = \pr{\tau_{0} \geq x}\sim cx^{-1/2}
\end{equation}
as $x \to \infty$ (the asymptotic on $\tau_0$ can be seen for example by combining  \autocite[Lemma 2.11 and 2.12]{le2012scaling}, which hold for general skip-free descending random walks, such as $\Zr$, or alternatively applying \autocite[Theorem 1]{DW}).

For technical reasons that will become apparent later when we consider $Z$ to be a certain random walk path, we will need to extend the definitions of $T_Z$ and $L_Z$ to excursions $(Z_i)_{i=0}^n$ satisfying the above conditions (those written above \eqref{eqn:contour equiv def}) but with the constraint $Z_n=1$ replaced by $Z_n\leq 0$. In the case where $Z_n \leq 0$, we extend $Z$ before time $0$ by setting $Z_i = i+1$ for $Z_n-1 \leq i \leq 0$, and denote the resulting excursion $Z^*$. We denote the resulting extended looptree (again obtained as a quotient space by applying \eqref{eqn:contour equiv def}) by $L_Z^*$, and the corresponding tree by $T_Z^*$. See \cref{fig:contour coding looptree extended}.\label{fa extended}

\begin{figure}[t]
\begin{center}
\medskip
\includegraphics[width=\textwidth]{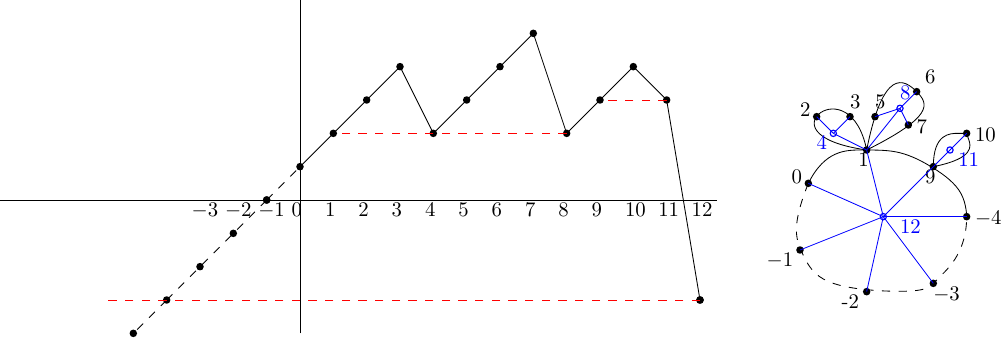}
\end{center}
\caption{An extended excursion $Z^*$ and its associated looptree $L_Z^*$.}\label{fig:contour coding looptree extended}
\end{figure}

In particular we will be interested in the case where $Z$ is a random walk started at $1$ and stopped at $\tau:=\inf \{n \geq 1: Z_n \leq 0\}$. We will allow large negative jumps so $Z_{\tau}$ can be strictly less than $0$, but it still follows from standard results on random walks (\cite[Theorem 1]{DW}, for example) that, similarly to \eqref{eqn:progeny asymp}, there exists a constant $c$ such that
\begin{equation}\label{eqn:progeny asymp 2}
\pr{\tau \geq x} \sim cx^{-1/2}.
\end{equation}

The reason for considering such extended excursions is as follows. In Section \ref{sctn:peeling example}, we will explore the boundary of a critical cluster using a Markovian exploration process and show that it can be encoded by a random walk excursion satisfying the assumptions on $Z$ as above, with $Z_n\leq 0$. The boundary of the cluster will (essentially) be of the form $L_Z^*$. Moreover, when we condition the cluster to be large, $T_Z^*$ is very similar to a two-type Galton--Watson tree (the only difference being this extension at the root face), so we will be able to apply general scaling limit results about decorated Galton--Watson trees to obtain the scaling limit of the critical cluster.

To apply results from the pre-existing literature in \cref{sctn:dec tree limits}, will also need to briefly introduce the \textit{depth-first} or \textit{height function} ordering. This can be defined using the Ulam-Harris notation introduced in \cref{sctn:GW background}: if $u=(u_1, \ldots, u_n)$ and $v=(v_1, \ldots, v_m) \in \mathcal{U}$, set $m(u,v) = \inf\{j \geq 1: u_j \neq v_j\}$. Then $u$ comes before $v$ in the depth-first ordering if and only if one of the following two conditions is satisfied:
\begin{enumerate}
\item $u$ is an ancestor of $v$,
\item $u_{m(u,v)} < v_{m(u,v)}$.
\end{enumerate}
Moreover, if $v_i$ denotes the $i^{th}$ vertex according to this ordering, we define $H_i$ to be the height of $v_i$ (i.e.\ its distance from the root).

\subsection{Decorated trees and looptrees}\label{sctn:dec trees}

We will show in \cref{sctn:cluster boundary looptree} that the boundary of the critical cluster essentially has the structure of a looptree generated by a two-type Galton-Watson tree, and that the whole cluster can be recovered by inserting critically percolated Boltzmann triangulations into the loops independently. (This is a consequence of the arguments of \autocite{ray2014geometry}.) To describe this kind of structure, we introduce the notion of a decorated two-type tree. This is similar to the decorated (one-type) trees considered in \autocite{archer2020random} and \autocite{senizergues2022decorated}, although here we focus on a different regime.

To construct the decorated tree, we will start by supposing that we have a sequence $(\Pb_n)_{n \geq 1}$ of distributions on random planar maps of the form $((G_n, d_n, m_n, \ell_n))_{n \geq 1}$, where, for all $n \geq 1$:\label{fa dec graph}
\begin{itemize}
  \item $G_n = (V_n, E_n)$ is a connected graph embedded in the plane with a simple boundary face of $n$ vertices (and edges);
  \item $d_n$ is a metric on $G_n$;
  \item $m_n$ is a measure on the vertices of $G_n$;
  \item the boundary vertices $\partial G_n$ come pre-equipped with a labelling $\ell_n: \partial G_n \to \{0, 1, \ldots, n-1\}$ representing the clockwise ordering of the boundary vertices when $G_n$ is embedded in the plane.
\end{itemize}
Moreover, we assume that the law of $G_n$ is invariant under a cyclic shift of the boundary labels. (We note that our construction and the associated theorems will hold more generally for graph families $G_n$ that are not planar; here we restrict to the planar case for simplicity.)

Given a two-type tree $t$ where all branches are of even length (i.e.\ all leaves are in $t_{\bullet}$), we then construct its decorated version $\Dec(t)$ via the following procedure. For each $u \in t_{\circ}$, let $G^{(u)}$ denote a copy of $(G_{\deg u}, d_{\deg u}, m_{\deg u}, \ell_{\deg u})$, where $\deg u:=k_u+1$ is the degree of $u$ within $t$; we suppose $G^{(u)}$ is sampled independently for each vertex. Informally, we form the decorated tree by replacing each vertex $u \in t_{\circ}$ with $G^{(u)}$, and identifying each boundary edge of $G^{(u)}$ with a neighbour of $u$ in $t$ such that the boundary vertex labelled $0$ is matched to the parent of $u$, the boundary vertex of $1$ is matched to the first child of $u$, etc. More formally, we define an equivalence relation on $\cup_{u \in t_{\circ}} G^{(u)}$ as follows. Firstly, for each $u \in t_{\circ}$ and each $0 \leq j \leq k_u$, let $g^{(u)}_j$ denote the boundary vertex of $G^{(u)}$ with label $j$. Each $v \in t_{\bullet}$ has precisely one parent (apart from the root, which has none) and $k_v\geq 0$ children; denote these $p(v)$ and $(c_j(v))_{1\leq j\leq k_v}$ respectively (where there are no children when $k_v=0$). Then for $x, y \in \cup_{u \in t_{\circ}} G^{(u)}$, set $x \sim y$ if and only if one of the following holds:
\begin{enumerate}
\item there exists $v \in t_{\bullet}$, $1 \leq i \leq k_{p(v)}$ and $1\leq j \leq k_v$ such that  $v=p(v)i$, $x=g^{p(v)}_i$ and $y = g^{c_j(v)}_0$ (or vice versa with the roles of $x$ and $y$ reversed);
\item there exists $v \in t_{\bullet}$ and $1\leq i < j \leq k_v$ such that $x=g^{c_i(v)}_0$ and $y = g^{c_j(v)}_0$ (or vice versa with the roles of $x$ and $y$ reversed).
\end{enumerate} The decorated tree is then equal to the space
\[\Dec (t) := \bigcup_{u \in t_{\circ}} G^{(u)} / \sim.\]\label{fa dec tree}

We endow $\Dec(t)$  with the metric $d$ obtained by adding contributions along the branches of $t$ in the natural way, and we denote this metric by $d$. Formally, if $x, y \in \Dec (t)$ with $x,y \in  G^{(u)}$ for some $u\in t_{\circ}$, then we simply define $d(x,y)=d_{G^{(u)}}(x,y)$, where $d_{G^{(u)}}$ represents the metric inherited from $G^{(u)}$. Otherwise, if $x \in  G^{(u)}$ and $y \in  G^{(v)}$ for $u,v \in t_{\circ}$ with $u\neq v$, we let $v_0, v_1, \ldots, v_n$ denote the shortest path of (underlying tree) vertices between $u$ and $v$ in $t_{\circ}$, so that $u=v_0$ and $v=v_n$. We then define $d$ by setting
\begin{equation}\label{eqn:decorated metric def}
\begin{split}
    d (x,y) = &d_{G^{(v_0)}}\left(x, b(v_0,v_1)\right) + \sum_{1 \leq i \leq n-1} d_{G^{(v_i)}}\left(b(v_{i-1},v_i), b(v_i, v_{i+1})\right) + d_{G^{(v_n)}}\left(b(v_{n-1},v_n), y\right),
\end{split}
\end{equation}
where $b(v_{i-1}, v_i)$ denotes $g_{0}^{v_{i}}$ if $v_i$ is a descendant of $v_{i-1}$ or if $v_i$ and $v_{i-1}$ are siblings, and $b(v_{i-1}, v_i)$ denotes $g_{0}^{v_{i-1}}$ if $v_{i-1}$ is a descendant of $v_{i}$. (We note that the choice of $u$ and $v$ in the definition of $d$ may not be unique if $x$ or $y$ are boundary vertices of the relevant graph, but the metric is nonetheless well-defined.)

We also define a measure $m$ on $\Dec (t)$ by setting, for each vertex $x \in \Dec (t)$,
\begin{equation}\label{eqn:decorated meas def}
m (x) = \sum_{u \in t_{\circ}} \mathbbm{1}_{\{[x] \cap G^{(u)} \neq \emptyset\}} m_{G^{(u)}}([x]),
\end{equation}
where $m_{G^{(u)}}([x])$ denotes the measure inherited from $G^{(u)}$, and $[x]$ refers to the equivalence class of $x$ under $\sim$ (or its relevant representative in $G^{(u)}$).

Finally, we define the root $\rho$ of $\Dec (t)$ to be the equivalence class of the root of $t$.

An example is shown in \cref{fig:decorated looptree}. Throughout the article we will use the general notation $(\Dec(t), d, m, \rho)$ to denote a quadruplet as constructed above.

\begin{figure}[t]
\includegraphics[width=\textwidth]{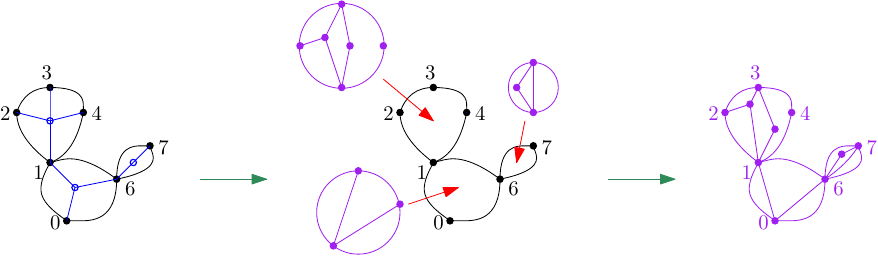}
\caption{A two-type tree and corresponding looptree, and a decorated version.}\label{fig:decorated looptree}
\end{figure}

\begin{rmk}\label{listrem}
The following points follow directly from the construction of a decorated tree.
\begin{enumerate}
\item If $d_n$ is the graph metric on $G_n$, then $d$ is also the graph metric on $\Dec (t)$.
\item If $d_n$ is the effective resistance metric on $G_n$ corresponding to unit conductances on the edges of $G_n$, then $d$ is the analogous effective resistance metric on $\Dec (t)$. (See \eqref{effres} later for a precise definition of effective resistance.)
\item If $m_n$ is the degree measure on $G_n$ (that is, $m_n(x)$ counts the number of vertices incident to the vertex $x$), then $m$ is the degree measure on $\Dec (t)$.
\item Up to a unit error, the counting measure on $\Dec (t)$ can be recovered as follows. For each vertex $x \in G_n$, let $m_n(x)=1$ except if $x$ is the boundary vertex such that $\ell_n(x)=0$; in the latter case, set $m_n(x)=0$. Apart from the root, every vertex $v \in \Dec (t)$ then has $m(v)=1$. We will obtain the scaling limit of decorated trees endowed with the counting measure using this formalism, since the difference at the root is negligible in the scaling limit.
\end{enumerate}
For the main results of this paper we will only consider the examples mentioned in this remark, but the results of \cref{sctn:dec tree limits} apply more generally.
\end{rmk}

\subsection{Brownian and random walk excursions}

Heuristically, a Brownian excursion of duration $\zeta$ is given by Brownian motion on the interval $[0,\zeta]$ conditioned to be zero at its endpoints and strictly positive in between. We refer to \autocite[Section IV.4]{BertoinLevy} for further background and a formal construction. We will use the notation $\B$ to denote a Brownian excursion in this paper. Due to its appearance in the construction of the CRT in \cref{sctn:CRT def}, it will be an important object in the present study. Later we will use the following two results to compare the probability of an event defined in terms of a Brownian excursion to that of the same event defined in terms of a Brownian motion.

\begin{lem}[Vervaat Transform] \autocite[Th\'eor\`eme 4]{Chaumont}. \label{lem:Vervaat}
\begin{enumerate}
\item Let $\B$ be a Brownian excursion conditioned to have lifetime $\zeta$ for some $\zeta \geq 1$. Take $U \sim$ \textsf{Uniform}$([0,\zeta])$. Then the process $(\Bb_t)_{0 \leq t \leq \zeta}$ defined by
\[\Bb_t = \begin{cases} \B_{U+t}-\B_U, & \text{ if } U+t \leq \zeta,\\
\B_{U+t-\zeta}-\B_U, & \text{ if } U+t > \zeta.
\end{cases}\]
has the law of a Brownian bridge on $[0,\zeta]$.
\item Now let $\Bb$ be a Brownian bridge on $[0,\zeta]$, and let $m$ be the (almost-surely unique) time at which it attains its minimum. Define an excursion $\B$ by
\[\B_t = \begin{cases} \Bb_{m+t}-\Bb_{m}, & \text{ if } m+t \leq \zeta,\\
\Bb_{m+t-\zeta}-\Bb_{m}, & \text{ if } m+t > \zeta.
\end{cases}\]
Then $\B$ has the law of a Brownian excursion of duration $\zeta$.
\end{enumerate}
\end{lem}

The probability of an event that depends on $(\Bb_t)_{t\in[0,T]}$ for some $T<\zeta$, where $\Bb$ is a Brownian bridge on $[0,\zeta]$, can be upper bounded using the probability of the same event for Brownian motion. Indeed, the law of $(\Bb_t)_{t\in[0,T]}$ is absolutely continuous with respect to the law of Brownian motion $(B_t)_{t\in[0,T]}$, with Radon-Nikodym derivative,
\begin{equation}\label{eqn:RN deriv Vervaat}
\frac{p_{\zeta-{T}}(-B_{T})}{p_{\zeta}(0)},
\end{equation}
where $p_t(x)$ is the transition density of Brownian motion defined with respect to Lebesgue measure on $\R$, see \autocite[Section VIII.3, Equation (8)]{BertoinLevy}. In conjunction with the first part of the previous lemma, this means that, for each fixed $\zeta$, we have for any $T<\frac{3}{4}\zeta$ and for any event $E$ defined on $C([0,T],\R)$ that there exists $C<\infty$ such that
\begin{align}\label{eqn:Vervaat cts}
\pr{E(\B_{[U,U+T]}-\B_U)} =\pr{E(\Bb_{[0,T]})}\leq C\pr{E(B_{[0,T]})},
\end{align}
where $U \sim$ \textsf{Uniform}$([0,\zeta])$, and we take the time indices modulo $\zeta$ as in \cref{lem:Vervaat}.

There is also a discrete version of the above conclusion. In particular, if $\X$ is a discrete time, centred, finite variance random walk excursion conditioned to return to $0$ for the first time at time $\zeta > 1$, and conditioned to be positive on the interval $(0, \zeta)$, then, provided this is not conditioning on a null event, there exists a constant $C<\infty$ (depending on the jump distribution of $\X$ but not on the choice of $\zeta$) such that for any $T<\frac{3}{4}\zeta$ and for any event $E$ defined in terms of a function from $[0,T]$ to $\R$
that
\begin{align}\label{eqn:Vervaat dis}
\pr{E(\X_{[U,U+T]}-\X_U)} \leq C\pr{E(X_{[0,T]})},
\end{align}
where $U \sim$ \textsf{Uniform}$(\{0, \ldots, \zeta\})$, and again we take the time indices modulo $\zeta$ as appropriate. This similarly holds if we take floor or ceiling functions for the indices in conjunction with $U \sim$ \textsf{Uniform}$([0,\zeta])$. This result follows since the discrete equivalent of \eqref{eqn:RN deriv Vervaat} is uniformly bounded under these conditions, as a consequence of the local limit theorem of \autocite[p.\ 233]{gk}.

\subsection{Gromov-Hausdorff-type topologies}\label{sctn:GHP topology}

We now introduce the Gromov-Hausdorff-type topologies that will be needed to state our main results precisely. We start with the pointed Gromov-Hausdorff-Prohorov (GHP) topology under which \cref{thm:main scaling lim} is stated. To this end, let $\mathbb{K}_c$ denote the set of quadruples $(K,d,\mu,\rho)$ such that $(K,d)$ is a compact metric space, $\mu$ is a locally-finite Borel measure on $K$, and $\rho$ is a distinguished point of $K$. Suppose $(K,d,\mu,\rho)$ and $(K',d',\mu',\rho')$ are elements of $\mathbb{K}_c$. Given a metric space $(M, d_M)$, and isometric embeddings $\phi, \phi'$ of $(K,d)$ and $(K',d')$, respectively, into $(M, d_M)$, we define $d_{M}\big((K,d,\mu,\rho, \phi), (K',d',\mu',\rho', \phi')\big)$ to be equal to
\begin{align*}
d_M^H(\phi (K), \phi' (K')) + &d_M^P(\mu \circ \phi^{-1}, \mu' \circ {\phi'}^{-1} ) + d_M(\phi (\rho), \phi' (\rho')).
\end{align*}
Here $d_M^H$ denotes the Hausdorff distance between two sets in $M$, and $d_M^P$ denotes the Prohorov distance between two measures, as defined in \autocite[Chapter 1]{BillsleyConv}, for example. The pointed Gromov-Hausdorff-Prohorov distance between $(K,d,\mu,\rho)$ and $(K',d',\mu',\rho')$ is then given by
\begin{equation}\label{eqn:GHP def}
\dGHP{(K,d,\mu,\rho), (K',d',\mu',\rho')} = \inf_{\phi, \phi', M} d_{M}\big((K,d,\mu,\rho, \phi), (K',d',\mu',\rho', \phi')\big)
\end{equation}
where the infimum is taken over all isometric embeddings $\phi, \phi'$ of $(X,d)$ and $(X',d')$, respectively, into a common metric space $(M, d_M)$. This defines a metric on the space of equivalence classes of $\mathbb{K}_c$ (see \autocite[Theorem 2.5]{AbDelHoschNoteGromov}, for example), where we say that two spaces $(K,d,\mu,\rho)$ and $(K',d',\mu',\rho')$ are equivalent if there is a measure and root preserving isometry between them. Moreover, $\mathbb{K}_c$ is a Polish space with respect to the topology induced by $d_{GHP}$ (again, see \autocite[Theorem 2.5]{AbDelHoschNoteGromov}, for example).

The pointed Gromov-Hausdorff (GH) distance $d_{GH}(\cdot, \cdot)$, defined by removing the Prohorov term from (\ref{eqn:GHP def}) above, can also be defined in terms of \textit{correspondences}. A correspondence $\mathcal{R}$ between $(K,d,\mu,\rho)$ and $(K',d',\mu',\rho')$ is a subset of $K \times K'$ such that for every $x \in K$, there exists $y \in K'$ with $(x,y) \in \mathcal{R}$, and similarly for every $y \in K'$, there exists $x \in K$ with $(x,y) \in \mathcal{R}$. We define the \textit{distortion} of a correspondence by
\[\textsf{dis} (\mathcal{R}) = \sup_{(x,x'), (y, y') \in \mathcal{R}} |d(x,y) - d'(x',y')|.\]
It is then straightforward to show that
\[\dGH{(K,d,\mu,\rho), (K',d',\mu',\rho')} = \frac{1}{2} \inf_{\mathcal{R}} \textsf{dis}(\mathcal{R}),\]
where the infimum is taken over all correspondences $\mathcal{R}$ between $(K,d,\mu,\rho)$ and $(K',d',\mu',\rho')$ that contain the point $(\rho, \rho')$ (see \autocite[Theorem 4.11]{evans2007probability}, for example). We further note that a correspondence comes with a canonical embedding in which the Hausdorff distance is at most $\frac{1}{2} \textsf{dis}(\mathcal{R})$. This is the space $K \sqcup K'$ endowed with the metric
\[D(x,y) = \begin{cases}
d(x,y), & \text{ if } x, y \in K, \\
d'(x,y), & \text{ if } x, y \in K', \\
\inf_{u, v \in \mathcal{R}} (d(x,u) + d'(y,v) + \frac{1}{2} \textsf{dis}(\mathcal{R})) & \text{ if } x \in K, y \in K'.
\end{cases}\]
(Again see \autocite[Theorem 4.11]{evans2007probability}, for example.) To prove \cref{thm:main scaling lim}, we will prove pointed Gromov-Hausdorff-Prohorov convergence by first proving pointed Gromov-Hausdorff convergence using correspondences, and then prove Prohorov convergence of measures on the associated canonical metric space embedding.

For the purposes of giving a precise statement of \cref{thm:main RW}, we extend the pointed Gromov-Hausdorff-Prohorov topology to incorporate c\'{a}dl\'{a}g paths using the general framework of \autocite{Khezeli}. (See also \autocite[Section 5]{CroydonER} for an earlier, related description of a topology on continuous paths built on length spaces.) In particular, we let $\tilde{\mathbb{K}}_c$ denote the set of quintuplets $(K,d,\mu,\rho,X)$, where $(K,d,\mu,\rho)\in \mathbb{K}_c$ and $X$ is a c\'{a}dl\'{a}g path from $[0,\infty)$ to $K$. Similarly to above, given a metric space $(M, d_M)$, and isometric embeddings $\phi, \phi'$ of $(K,d)$ and $(K',d')$, respectively, into $(M, d_M)$, we define $\tilde{d}_{M}\big((K,d,\mu,\rho, X,\phi), (K',d',\mu',\rho',X', \phi')\big)$ to be equal to
\begin{align*}
d_M^H(\phi (K), \phi' (K')) + &d_M^P(\mu \circ \phi^{-1}, \mu' \circ {\phi'}^{-1} ) + d_M(\phi (\rho), \phi' (\rho'))+d_M^{J_1}(\phi (X), \phi' (X')),
\end{align*}
where $d_M^{J_1}$ is the metrization of the Skorohod $J_1$-topology for c\'{a}dl\'{a}g paths on $M$ described in \autocite[Example 3.44]{Khezeli}. (See also \autocite[Section 2.3]{Noda} for further details.) Defining
\[d_{\tilde{\mathbb{K}}_c}\left({(K,d,\mu,\rho,X), (K',d',\mu',\rho',X')}\right) = \inf_{\phi, \phi', M} \tilde{d}_{M}\big((K,d,\mu,\rho,X, \phi), (K',d',\mu',\rho',X', \phi')\big),\]\label{fa dKc}
where again the infimum is taken over all isometric embeddings $\phi, \phi'$ of $(X,d)$ and $(X',d')$, respectively, into a common metric space $(M, d_M)$, yields a distance on $\tilde{\mathbb{K}}_c$. Moreover, $d_{\tilde{\mathbb{K}}_c}$ defines a metric on the space of equivalence classes of $\tilde{\mathbb{K}}_c$, where for two spaces $(K,d,\mu,\rho,X)$ and $(K',d',\mu',\rho',X')$ to be equivalent, we require a measure, root and c\`{a}dl\`{a}g path preserving isometry between them. As above, $\tilde{\mathbb{K}}_c$ is a Polish space with respect to the topology induced by $d_{\tilde{\mathbb{K}}_c}$. (In \autocite{Khezeli, Noda}, the metric and topology are built on a larger space, in which killing of the c\'{a}dl\'{a}g path components is allowed. It is easy to check that our space is a closed subset of this larger metric space, and so inherits the Polish property.)

\subsection{Conditionings}

Later, in \cref{sctn:main result proof}, it will be convenient to apply the following lemma to switch between various conditionings on the critical root cluster.

\begin{lem}\label{lem:switch conditioning}
Suppose that $(A_n)_{n \geq 1}$, $(B_n)_{n \geq 1}$ are two sequences of events on a probability space $(\Omega, \F, \Pb)$, where $\mathcal{F}$ is the Borel $\sigma$-algebra for some underlying metric, such that $\prcond{A_n}{B_n}{}=1+o(1)$ and $\prcond{B_n}{A_n}{}=1+o(1)$ as $n \to \infty$. Let $(X_n)_{n \geq 1}$ be a sequence of random variables on $\Omega$ such that $X_n$ given $A_n$ converges in distribution to a limiting random variable $X$ as $n \to \infty$. Then $X_n$ given $B_n$ also converges in distribution to $X$ as $n \to \infty$.
\end{lem}
\begin{proof}
Take any open set $E \in \F$. Then
\begin{align*}
\prcond{X_n \in E}{B_n}{}& \geq \prcond{X_n \in E}{B_n, A_n}{}\prcond{A_n}{B_n}{}\\
 &= \prcond{X_n \in E, B_n}{A_n}{} \frac{\prcond{A_n}{B_n}{}}{\prcond{B_n}{A_n}{}} \\
&\geq \left[\prcond{X_n \in E}{A_n}{} - \prcond{B_n^c}{A_n}{}\right] \frac{\prcond{A_n}{B_n}{}}{\prcond{B_n}{A_n}{}} \\
&= \prcond{X_n \in E}{A_n}{} - o(1).
\end{align*}
The reverse inequality also holds by symmetry. The result therefore follows by the Portmanteau theorem.
\end{proof}

\section{Hyperbolic maps and percolation}\label{sctn:the model}

In this section we formally define our model of site percolation on half-planar maps. We then explain how the boundary of the root cluster can be explored using a peeling process, and how it can be coded by a random walk excursion and a looptree similar to those appearing in \cref{sctn:looptrees}. Finally, we explain how this result can be used to give an equivalent construction of the cluster as a decorated tree (\cref{def:dec tree construction of cluster}).

\subsection{Half-planar maps}

A \textbf{planar map} is a connected planar graph embedded in the sphere (defined up to orientation preserving homeomorphisms). The \textbf{edges} and \textbf{vertices} of the map are those of the original graph, and the \textbf{faces} are the connected components of the complement of the embedded edges and vertices. A map with a \textbf{boundary} is a map in which one face is designated the \textbf{boundary face}, and the boundary of that map is the collection of edges and vertices incident to the boundary face. The boundary need not be simple, but in the case where the map is finite and the boundary face is an $m$-gon, we say that the map is a \textbf{map of the $m$-gon}. For example, the inserted graphs in \cref{fig:decorated looptree} are maps of the $4$-gon, $3$-gon and $2$-gon. We also say that an edge or a vertex is \textbf{internal} if it is not on the boundary. We will also be interested in \textbf{half-planar maps}. These are connected planar graphs embedded in the upper-half plane that have as a boundary a simple doubly-infinite path, which we will always assume has vertices given by $\mathbb{Z}$ and nearest neighbour edges. We will restrict to \textbf{triangulations}, i.e.\ planar/half-planar maps in which all faces, except possibly the boundary face, are triangles. We restrict to \textbf{rooted} planar/half-planar maps: that is, each map contains a single edge known as the root edge, which lies on the boundary. The root edge will be oriented from left to right (anticlockwise in the case of finite planar maps), and this orientation will be indicated with an arrow. The \textbf{root vertex} is the leftmost endpoint of the root edge.

The main focus of this paper is on \textit{random} rooted half-planar maps with laws satisfying two important properties: translation invariance and the domain Markov property. These are defined as follows. Given a rooted half-planar map $M$, we define $\theta (M)$ to be the rooted half-planar map that is the same as $M$ but such that the root edge of $\theta(M)$ is the boundary edge immediately to the right of the root edge of $M$. If $\mathbb{H}$ is a law on rooted half-planar maps, then $\mathbb{H}$ is said to be \textbf{translation-invariant} if $\mathbb{H}\circ\theta^{-1} = \mathbb{H}$. In addition, suppose that $Q$ is a map of the simple $m$-gon for some $m > 0$, and for $0<k<m$ define $A_{Q,k}$ to be the event that $M$ contains a submap that is isomorphic to $Q$ and contains the $k$ boundary edges immediately to the right of the root edge of $M$, but no other boundary edges or vertices, and is such that the root edge of $Q$ corresponds to the edge immediately to the right of the root of $M$. Then $\mathbb{H}$ satisfies the \textbf{domain Markov property} if, for any such $Q$ and $k$, conditionally on the event $A_{Q,k}$, the map $M \setminus Q$ also has law $\mathbb{H}$.\label{fa M}

In \autocite{angelray2015} the authors studied and classified all simple (equivalently, loopless) triangulations of the half-plane satisfying translation invariance and the domain Markov property. They showed that these maps form a one parameter family indexed by a parameter $\alpha \in [0,1)$, whereby $\alpha$ gives the probability that the triangle incident to a fixed boundary edge is incident to an internal vertex. In this paper, we will be interested in maps of the half-plane of hyperbolic flavour. This corresponds to the regime where $\alpha \in (\frac{2}{3}, 1)$, and we will denote the law of the corresponding map $\Ha$. See \autocite[Section 1]{angelray2015} for further background.

\subsection{Boltzmann Maps}

As well as random half-planar maps, we will need to consider probability measures on triangulations of the $m$-gon. Given a real number $q \in (0, \frac{2}{27}]$, the associated \textbf{Boltzmann measure} (see \autocite[Section 3.3]{curien2019peeling}, for example) on triangulations of the $m$-gon is defined by setting
\[\prstart{T}{m,q} = \frac{q^{\# T }}{\Zmq},\]
for any finite rooted triangulation of the $m$-gon $T$, where $\# T$ denotes the number of internal vertices of $T$ and $\Zmq$ is a normalising constant. The restriction that $q \leq \frac{2}{27}$ is exactly what is required to ensure that $\Zmq < \infty$ for all $m \geq 1$.

\subsection{Peeling of half-planar maps}

As alluded to above, the work \autocite{angelray2015} characterises the random maps with laws $(\Ha)_{\alpha \in [0,1)}$ in terms of their \textbf{peeling probabilities}. We will use the formalities of the peeling process introduced by Angel \autocite{angel2003growth} and developed by Curien \autocite{curien2019peeling} in order to explore unknown random planar maps by revealing faces one by one, although the origins of the technique date back to the physics literature \autocite{ambjorn1994quantization, watabiki1995construction}. This peeling process can be formally described using a \textbf{peeling algorithm}. In particular, for half-planar maps this algorithm starts by supposing the whole map $M$ to be unknown, apart from its infinite boundary face. The algorithm then continues by selecting (``peeling") the root edge and revealing the internal face incident to it. Denote this face $F$. On a triangulation, this can have one of the following outcomes.
\begin{itemize}
  \item If the third vertex of $F$ is an internal vertex, then $M \setminus F$ is another map of the half-plane. Appropriately rooted, by translation invariance and the domain Markov property, this has the same law as $M$.
  \item If the third vertex is instead a boundary vertex at distance $m \geq 1$ from the root edge, then removing $F$ from $M$ creates two connected components, exactly one of which is infinite. The infinite component has a simple boundary; appropriately rooted, by translation invariance and the domain Markov property, this has the same law as $M$. The other component has a finite simple boundary and is therefore a triangulation of the $(m+1)$-gon.
\end{itemize}
See \cref{fig:peeling options} for an illustration. The algorithm then proceeds to explore the infinite component of $M \setminus F$ by selecting a new edge to peel on the boundary of $M \setminus F$, and continues inductively.

\begin{figure}[t]
\centering
\includegraphics[width=0.7\textwidth]{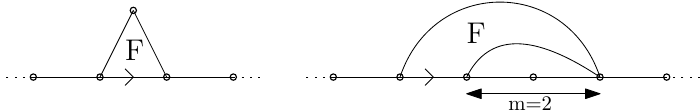}
\caption{Discovering an internal vertex (left) or another boundary vertex (right). The root edge is indicated by an arrow. The distance between the root edge and another boundary vertex is the number of edges on the boundary lying in between them.}\label{fig:peeling options}
\end{figure}

In \autocite{angelray2015}, the authors showed that for each $\alpha \in [0,1)$, there exists precisely one probability measure on maps of the half-plane, denoted by $\Ha$, satisfying translation invariance and the domain Markov property for which the probability of revealing an internal vertex on the first step of the peeling algorithm is exactly $\alpha$. In this paper we are only concerned with the regime where $\alpha \in (\frac{2}{3},1)$. For fixed $\alpha$, they also determined the probabilities of all of the possible peeling outcomes described above, and the law of the triangulation of the $(m+1)$-gon that can appear on discovering a boundary vertex. We summarise their results in \cref{prop:peeling probs}. In order to state this, for $m \geq 1$, let $p_m$ denote the probability that the third vertex of $F$ is on the boundary of $M$, at distance $m$ from the root edge (as in the right-hand image of \cref{fig:peeling options}), on either side of the root edge. (By symmetry, the probability that this happens specifically on the right or the left of the root is therefore $\frac{1}{2}p_m$.)

\begin{prop}\autocite[Equation (3.7)]{angelray2015}. \label{prop:peeling probs}
Let $\alpha \in (\frac{2}{3},1)$, and suppose $M\sim \Ha$. Moreover, let $F$ be the first face uncovered by the peeling algorithm.
\begin{enumerate}[(i)]
\item For each $m \geq 1$,
\[p_m = \frac{2}{4^m}\frac{(2m-2)!}{(m-1)!(m+1)!} \left(\frac{2}{\alpha} - 2 \right)^m ((3\alpha-2)m+1),\]
so that $p_m \sim c \left(\frac{2}{\alpha} - 2 \right)^m m^{-\frac{3}{2}}$ as $m \to \infty$ for a constant $c \in (0, \infty)$.
\item On the event that the third vertex of $F$ is on the boundary of $M$, at distance $m$ from the nearest endpoint of the root edge, then the finite component of $M \setminus F$ has the law of $\Pb_{m+1, q}$ with $q = \frac{1}{2} \alpha^2 (1-\alpha)$.
\end{enumerate}
\end{prop}
\begin{proof} Part (i) is given at \autocite[(3.7)]{angelray2015}. Part (ii) follows from \cite[Lemma 3.7]{angelray2015}, using Euler's characteristic formula (and the fact we are dealing with triangulations) to replace the number of faces of the relevant map by the number of vertices.
\end{proof}

\subsection{Critical site percolation on $\Ha$}\label{sctn:peeling example}

We now introduce site percolation on a map with law $\Ha$ and explain how a percolation cluster can be explored via a peeling algorithm as is done in \autocite{ray2014geometry}. We fix $\alpha \in (\frac{2}{3}, 1)$ and $p \in (0,1)$, sample a map $M$ with law $\Ha$, and then independently colour each vertex of $M$ black with probability $p$ or white with probability $1-p$. We will impose the boundary condition that the root vertex is coloured black, and all other boundary vertices are coloured white. As is the case on $\Z^d$, Ray \autocite[Theorem 2.7]{ray2014geometry} showed there is a (deterministic) critical probability $p_c$ above which the root cluster has a positive probability of being infinite, and below which the root cluster is almost-surely finite, $\Ha$-almost-surely, and computes its explicit value to be that given at \eqref{pcdef}.

Below we outline an algorithm that explores the boundary of the root cluster in the above percolation model. We introduce a process $(B_i)_{i \geq 0}$ that records the evolution of the algorithm via the length of the black boundary. We will explain afterwards, in \cref{sctn:cluster boundary looptree}, how the sequence $(B_i)_{i \geq 0}$ can be used to reconstruct the boundary of the percolation cluster. We give an example configuration in \cref{fig:clusterexample}, and give three of the associated peeling steps in \cref{fig:peeling examples}. Note that Ray uses the same algorithm in the proof of \autocite[Theorem 2.7]{ray2014geometry} to compute the value of $p_c$. For the rest of the paper, we assume $p=p_c$.

\begin{figure}[t]
\includegraphics[width=\textwidth]{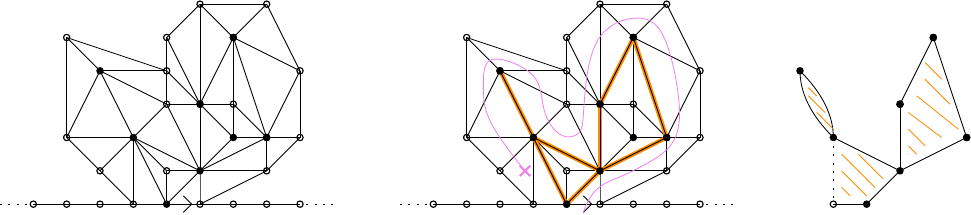}
\caption{On the left, a possible configuration for the root cluster. In the middle, the same configuration, but with the cluster boundary and peeling interface identified. On the right, the final structure discovered by the peeling process. Note that the only difference with the true cluster boundary is in the size of the final face.}\label{fig:clusterexample}
\end{figure}

\begin{figure}[t]
\includegraphics[width=\textwidth]{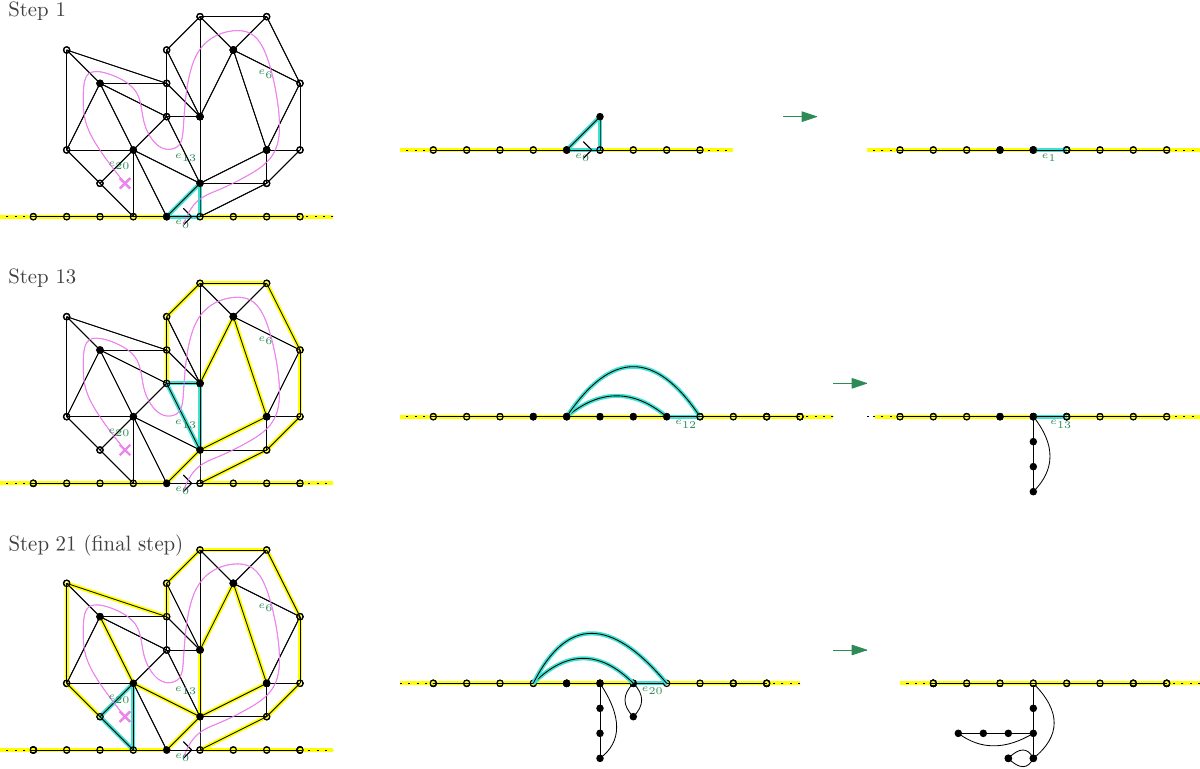}
\caption{Three selected peeling steps of the cluster given in previous figure. On the left hand side we show how the exploration evolves along the cluster boundary (this is for illustrative purposes, since the boundary and the interface are not known in advance of the peeling). In the middle we show the face discovered at the start of the given peeling step, and on the right hand side we show how this is updated after the peeling step.}\label{fig:peeling examples}
\end{figure}

\begin{enumerate}
\item Set $\hat{T}=\infty$, set $B_0=1$, set $M_0 = M$, and select the root edge to be $e_0$, the first edge to peel. Note that the leftmost endpoint of the root edge is black, and the rightmost endpoint is white.\label{fa B}
\item Peel $e_0$; that is, reveal the triangle adjacent to $e_0$. Denote this face $F_0$.
\item If the third vertex of $F_0$ is internal and is white, define $e_1$ to be the edge joining this new vertex to the black endpoint of the root edge, and set $B_1=1$. If it is internal and black, define $e_1$ to be the edge joining this vertex to the white endpoint of the root edge and set $B_1=2$. Otherwise, the third vertex of $F_0$ is a boundary vertex. If it is white and to the right of $e_0$ at some distance $m \geq 1$ from $e_0$, we define $e_1$ to be the edge joining this vertex to the black endpoint of $e_0$, and set $B_1=1$. Otherwise, it must be white and to the left of the root at some distance $m \geq 1$ from $e_0$; in this case we terminate the algorithm, set $B_1=1-m$, and redefine $\hat{T} = 1$. On the event that we have not terminated the algorithm, define $M_1$ to be the unique infinite connected component of $M_0 \setminus F_0$, rooted at $e_1$.
\item We proceed inductively. If we did not terminate the algorithm on the $i^{th}$ peeling step (i.e. we did not set $\hat{T}=i$), we instead obtained a map $M_i$ with law $\Ha$ and with root edge $e_i$ such that $e_i$ is on the boundary of $M_i$, and such that the leftmost endpoint of $e_i$ is black, and the rightmost endpoint is white. We also obtained a number $B_i \geq 1$. Moreover, we note that all vertices to the right of $e_i$ are white, whilst there is a string of length $B_i$ of black vertices starting on the left of $e_i$ (this follows by induction). Now peel $e_i$; that is, reveal the internal triangle adjacent to $e_i$, and denote this face $F_i$.
\item If the third vertex of $F_i$ is internal and is white, define $e_{i+1}$ to be the edge joining this vertex to the black endpoint of the root edge, and set $B_{i+1} = B_i$. If it is internal and black, define $e_{i+1}$ to be the edge joining this vertex to the white endpoint of the root edge, and set $B_{i+1} = B_i + 1$. Otherwise, the third vertex of $F_i$ is a boundary vertex of $M_i$. If it is white and to the right of $e_i$, at distance $m \geq 1$ from $e_i$, we define $e_{i+1}$ to be the edge joining this vertex to the black endpoint of $e_i$, and set $B_{i+1} = B_i$. If it is black and to the left of $e_i$, at distance $m \geq 1$ from $e_i$, we define $e_{i+1}$ to be the edge joining this vertex to the white endpoint of $e_i$, and set $B_{i+1} = B_i-m$. Otherwise, it must be white and to the left of the $e_i$, at some distance $m\geq B_i$ from $e_i$; in this case we terminate the algorithm and set $\hat{T} = i+1$ and $B_{i+1} = B_i-m$. On the event that we have not terminated the algorithm, define $M_{i+1}$ to be the unique infinite connected component of $M_i \setminus F_i$.
\item Repeat steps 4 and 5 until the algorithm is terminated. (By \autocite[Theorem 2.7]{ray2014geometry}, this happens almost-surely when $p$ is critical.)
\end{enumerate}
As noted in step 4 above, it follows by construction that $B_i$ gives the number of black boundary vertices at the beginning of the $(i+1)^{st}$ peeling step. This will be strictly positive precisely up until we terminate the algorithm by swallowing all of the black boundary vertices when we discover a white vertex to the left in the final peeling step. It therefore follows that
\[\hat{T} = \inf \{i \geq 0: B_i \leq 0\}.\]
Moreover, the cluster boundary can be recovered from the final step of the peeling algorithm as it given by the structure ``hanging off" the boundary (see the bottom right picture in \cref{fig:peeling examples} and the right picture in \cref{fig:clusterexample}). We make this precise in the next subsection.

\subsection{Reconstructing the cluster boundary from $(B_i)_{i \geq 0}$}\label{sctn:cluster boundary looptree}

We will use the function $(B_i)_{i \geq 0}$ to define a contour function of a looptree. It turns out that this looptree will (almost) be the boundary of the root cluster. In fact, the looptree will have exactly the same structure as the structure hanging off the boundary vertex after the final peeling step, for example as shown in \cref{fig:peeling examples}. In the figure, we have shown how the cluster faces ``hang off'' the boundary after they are discovered, and are then carried through the rest of the peeling steps. This picture is good to have in mind for intuition, but the final structure can be recovered directly from the process $(B_i)_{i \geq 1}$, as we now describe.

First, given $\hat{T}$ (the termination time) and the sequence $(B_i)_{0 \leq i \leq \hat{T}}$, we define a sequence of stopping times $(\tau_i)_{i =0}^{\infty}$ by setting $\tau_0=0$ and
\[\tau_i = \inf \{j \geq \tau_{i-1}:\: B_j \neq B_{\tau_{i-1}}\} \wedge \hat{T}, \hspace{1cm} i \geq 1.\]
We then define the contour function $Z$ by setting $Z_i = B_{\tau_i}$ for all $0 \leq i \leq \hat{\tau}$, where $\hat{\tau} = \inf\{i \geq 0: \tau_i = \hat{T}\}$. The following proposition will be crucial to our analysis.

\begin{prop}\label{prop:C is centred RW}
When $p=p_c$, $Z$ is a centred random walk with i.i.d.\ increments, started from 1, and terminated the first time it hits $\{x \leq 0\}$. More precisely, the distribution of the increments of $Z$, $\mu=(\mu(i))_{i\in\{1,-1,-2,\dots\}}$, is given by
\begin{equation}\label{eqn:mu def}
\mu(i):=\left\{
         \begin{array}{ll}
           c_\alpha \alpha p_c, & \hbox{if $i=1$;} \\
           \frac12 c_\alpha p_{-i}, & \hbox{if $i=-1,-2,\dots$,}
         \end{array}
       \right.
\end{equation}
where
\[c_\alpha:=2\left(1-\sqrt{\alpha(3\alpha-2)}\right)^{-1}.\]
In particular, the increments of $Z$ have exponentially decaying tails.
\end{prop}
\begin{proof}
It follows by construction that $(B_i)_{i\geq 0}$ is a random walk started from 1 and stopped when it is no longer strictly positive, whose increments satisfy: $B_{i+1}-B_i = 1$ if the $(i+1)^{st}$ peeling step discovers an internal black vertex, $B_{i+1}-B_i=-m$ if the next peeling step discovers a boundary vertex to the left of the root edge at distance $m$ from the root (i.e.\ with precisely $m$ edges in between the left-most endpoint of the root edge and the newly-discovered vertex), and $B_{i+1}-B_i=0$ otherwise. Additionally, the probabilities of these first two outcomes are given by $\alpha p_c$ and $\frac12 p_m$, respectively.

From the observations of the previous paragraph and the strong Markov property, we obtain that $Z$ is also a random walk started from 1 and stopped when it is no longer strictly positive. Moreover, its increment distribution is as given in the statement of the lemma, with the normalising constant computed as follows:
\[c_\alpha^{-1}=\alpha p_c +\frac12 \sum_{m\geq 1}p_m=\alpha p_c+\frac12(1-\alpha)=\frac12\left(1-\sqrt{\alpha(3\alpha-2)}\right),\]
where the final equality comes from the description of $p_c$ at \eqref{pcdef}. Concerning the centring of the increments, we use that
\[\sum_{m\geq 1} m p_m = \alpha - \sqrt{\alpha(3\alpha-2)},\]
which was used in the proof of \autocite[Lemma 4.2]{ray2014geometry}, to deduce that
\[c_\alpha^{-1}\sum_{m= 1,-1,-2,\dots}m\mu_m=\alpha p_c-\frac12\sum_{m\geq 1}mp_m=0,\]
as required. (Again, the final equality comes from the description of $p_c$ at \eqref{pcdef}.) Finally, it follows from \cref{prop:peeling probs} that the jumps of $Z$ have exponentially decaying tails.
\end{proof}

\begin{rmk}
Since $(Z_i)_{i \geq 0}$ tracks the evolution of the black boundary length (at points of increase or decrease in the peeling algorithm), it follows that the cluster has a positive chance of being infinite precisely when $\E{Z_1-Z_0}>0$, and therefore that $\E{Z_1-Z_0}=0$ precisely when $p$ is critical. This is the strategy Ray uses to identify the value of $p_c$ in \autocite{ray2014geometry}. However, as we do in the previous proof, one can also verify this by substituting the value of $p_c$ back into the relevant sum.
\end{rmk}

Note that $Z$ has the properties required to define an extended looptree, as described in \cref{sctn:coding} and illustrated in \cref{fig:contour coding looptree extended}. Moreover, let $\CC$ be the root cluster of the critical percolation process. This can be obtained by taking the structure ``hanging off'' the boundary at the end of the final peeling step, and filling it in with percolated Boltzman triangulations. In particular, up to the discrepancy of the final loop, the boundary $\partial \CC$ is the extended looptree $L_Z^*$ constructed from $Z$, as described in \cref{sctn:coding} (see the right image of \cref{fig:clusterexample}, which shows precisely $L_Z^*$, and the bottom right image of \cref{fig:peeling examples}, which shows how this is obtained in the final peeling step). The following proposition gives the offspring laws of the two-type Galton-Watson tree associated with the reversed process $\overleftarrow{Z}$.

\begin{prop}\label{prop:tree distribution}
Let $Z$ be as defined above and $p=p_c$. Let $\overleftarrow{Z}$ be the random walk with step distribution $\overleftarrow{\mu}(i)=\mu(-i)$. The two-type Galton-Watson looptree associated with an excursion of $\overleftarrow{Z}$ as in \cref{prop:tree offspring RW} (with offspring laws as around \eqref{eqn:mu def general case background reversed}) has offspring laws given by
\[\mu_{\circ}(m) = \frac{p_{m}}{1-\alpha}, \:m \geq 1, \hspace{1cm} \mu_{\bullet}(m) = \frac{2\alpha p_c}{1-\alpha+2\alpha p_c}\left(\frac{1-\alpha}{1-\alpha+2\alpha p_c}\right)^{m}, \:m \geq 0.\]
\end{prop}
\begin{proof}
It follows directly from \eqref{pcdef}, \cref{prop:tree offspring RW} and \cref{prop:C is centred RW} that $\mu_{\bullet}$ is geometric random variable (supported on the non-negative integers) with parameter
\[c_\alpha\alpha p_c=\frac{\alpha-\sqrt{\alpha(3\alpha-2)}}{1-\sqrt{\alpha(3\alpha-2)}}=\frac{2\alpha p_c}{1-\alpha+2\alpha p_c}.\]
It similarly follows that
\[ \mu_{\circ}(i) =
(1-\mu_{\bullet}(0))^{-1}\mur(i) = \frac{c_\alpha p_{i}}{2(1-c_\alpha\alpha p_c)} = \frac{p_i}{1-\alpha}.\]

Alternatively this can be seen by directly analysing the random walk coding mechanism. The random walk jump distribution $\overleftarrow{\mu}$ is given by \cref{prop:C is centred RW}. Now, recalling the construction illustrated in \cref{fig:contour coding looptree}, we have that a loop of length $i+1$ (around a white vertex with $i$ offspring) is coded by a positive jump of $\overleftarrow{Z}$ of size $i$. Therefore, it follows that
\[\mu_{\circ}(i) = \frac{\frac{1}{2}c_\alpha p_i}{\sum_{k \geq 1}\frac{1}{2}c_\alpha p_k} = \frac{p_i}{\sum_{k \geq 1}p_k}= \frac{p_i}{1-\alpha},\]
i.e.\ the probability that a positive jump of $\overleftarrow{Z}$ is of size $i$. Similarly it follows that the number of offspring of a black vertex, say corresponding to time $j$ of $\overleftarrow{Z}$ (choosing the minimal such $j$), is equal to the number of excursions of $\overleftarrow{Z}$ above $\overleftarrow{Z}_j$ before $\overleftarrow{Z}$ drops to $\overleftarrow{Z}_j-1$. The number is therefore a geometric random variable (supported on $\N_0$) with parameter
\begin{equation*}\label{eqn:fa No}
c_\alpha\alpha p_c=\frac{\alpha-\sqrt{\alpha(3\alpha-2)}}{1-\sqrt{\alpha(3\alpha-2)}}=\frac{2\alpha p_c}{1-\alpha+2\alpha p_c}.\qedhere
\end{equation*}
\end{proof}

\begin{rmk}\label{exprem}
It holds that
\[\E{\mu_{\circ}} = \frac{2\alpha p_c}{1-\alpha},\qquad \E{\mu_{\bullet}}=\frac{1-\alpha}{2\alpha p_c}.\]
\end{rmk}

By working backwards through the steps described above, it is possible to deduce that, by starting from a random walk path $Z$ as in \cref{prop:C is centred RW}, one can reconstruct the critical root cluster as follows.

\begin{defn}[Decorated-tree construction of the critical cluster]\label{def:dec tree construction of cluster}
\begin{enumerate}
\item Let $Z$ denote a random walk path with step distribution $\mu$ as in \eqref{eqn:mu def}, started at $1$ and stopped at the first time it hits the non-positive integers.
\item Construct the associated two-type extended looptree $L_Z^*$, as described in \cref{sctn:coding} and illustrated in \cref{fig:contour coding looptree extended}.
\item The root vertex of the looptree is the vertex with label $0$. The root face of the looptree is the face containing both $0$ and $1$. For every non-root face of degree $k$, insert a Boltzmann triangulation with law $\Pb_{k, q}$ where $q = \frac{1}{2} \alpha^2 (1-\alpha)$. In other words, sample such a Boltzmann map and then identify the boundary of the $k$-gon with the boundary of the face in the looptree, starting by identifying a uniform point on each and continuing in such a way as to respect the planar embedding of each. Then perform Bernoulli site percolation on the triangulation of each $k$-gon with parameter $p_c$ and with black boundary condition. Do this independently for each non-root face.
\item For the root face, we again insert a Boltzmann triangulation of the $k$-gon with law $\Pb_{k, q}$ as in step 3 (i.e.\ with $k$ being the length of the face boundary and $q = \frac{1}{2} \alpha^2 (1-\alpha)$), and perform Bernoulli site percolation with parameter $p_c$, but this time the boundary condition is a bit different. Note that the boundary length of the root face is equal to one more than the size of the final jump of $Z$. Let $Z_{\tau}$ denote the terminal value of $Z$; the boundary condition for the final insertion will be that the root vertex of the $k$-gon will be black, and moving clockwise around the boundary it will be followed by an additional $k+Z_{\tau}-2$ black vertices, then $1-Z_{\tau}$ consecutive white vertices. In order to specify the embedding precisely, we map the root vertex of the $k$-gon to the vertex labelled $0$ in the looptree and continue to match vertices in such a way as to maintain the planarity of both graphs.
\item Finally, retain only the connected component of black vertices that contains the root vertex (i.e.\ the vertex labelled 0), and denote the resulting graph $\CC_Z$.
\end{enumerate}
\end{defn}

The following proposition follows directly from the definition of $Z$ and Propositions \ref{prop:peeling probs} and \ref{prop:C is centred RW}.

\begin{prop}\label{prop:def36ok}
The structure $\CC_Z$ constructed as in \cref{def:dec tree construction of cluster} is equal in distribution to $\CC$, the critical root cluster under white--black--white boundary condition.
\end{prop}

\section{Scaling limit of decorated critical trees}\label{sctn:dec tree limits}

In this section we will show that the scaling limit for a general class of decorated two-type Galton-Watson trees is the CRT (see Proposition \ref{prop:GHP convergence of enriched trees} below). This is a key step to proving Theorem \ref{thm:main scaling lim}. In particular, in \cref{sctn:main result proof} we will show that the critical cluster is very close to a decorated tree, and therefore has the same scaling limit.

As the underlying tree, in the notation of \cref{sctn:GW background}, we let $T_n$ be a two-type Galton-Watson tree conditioned to have at least $n$ vertices, with offspring distributions $\mu_{\bullet}$ and $\mu_{\circ}$ satisfying the following four conditions.

\begin{assn}\label{assn:offspring}
\begin{enumerate}[(1)]
\item The measure $\mu_{\circ}$ is supported on $\mathbb{N}$, i.e.\ $\mu_{\circ}(0)=0$.
\item The measure $\mu_{\bullet}$ is geometric, supported on the non-negative integers.
\item It holds that $\E{\mu_{\bullet}}\E{\mu_{\circ}}=1$.
\item If $X$ is a random variable with law $\mu_{\circ}$ or $\mu_{\bullet}$, there exists a $\lambda>0$ such that $\E{e^{\lambda X}} < \infty$.
\end{enumerate}
\end{assn}

Under the first three assumptions, note that the reversed walk excursion $\Zr$ that codes $T_n$ using the bijection established in \cref{prop:tree offspring RW} is a centred random walk excursion with increment distribution $\mur$ as in \eqref{eqn:mu def general case background reversed}. The time reversal of $\Zr$ is a random walk with increment distribution $\mu$, where
\begin{equation}\label{eqn:mu def general case}
\mu(i) = \begin{cases}
\mu_{\bullet}(0), & \text{ if } i=1, \\
(1-\mu_{\bullet}(0))\mu_{\circ}(-i), & \text{ if } i\leq -1, \\
0, & \text{ otherwise}.
\end{cases}
\end{equation}
Moreover, the conditioning on $T_n$ means that the excursion of $\Zr$ is conditioned to have lifetime at least $n$.

Letting $d_n$ and $\nu_n$ respectively denote the graph metric and counting measure on the vertices of $T_n$, and $\rho_n$ its root, it is shown in \autocite[Theorem 2]{miermont2008invariance} that there exists $\sigma>0$ such that
\begin{equation}\label{eqn:miermont lim}
\left(T_n, \frac{\sigma}{2\sqrt{n}}d_n, \frac{1}{n} \nu_n, \rho_n \right) \to (\T, d_\T, \nu_\T, \rho_\T)
\end{equation}
as $n \to \infty$. The constant $\sigma$ is defined in \autocite[Equation (2)]{miermont2008invariance} and can be calculated explicitly when $\mu_{\bullet}$ and $\mu_{\circ}$ are as above, using \cref{prop:peeling probs} and generating functions for Catalan numbers. The quantity $\sigma^2$ is analogous to the variance in the one-type case. However, the precise value of $\sigma$ and its derivation are not particularly enlightening, so we omit the computation and simply refer to $\sigma$ throughout the proofs.

We will now construct $\Dec(T_n)$, a decorated version of $T_n$ as described in \cref{sctn:dec trees} and illustrated in \cref{fig:decorated looptree}. In particular, as in \cref{sctn:dec trees}, we let $((G_i, d_i, m_i, \ell_i))_{i \geq 1}$ be a sequence of graphs parametrised by their boundary length, with respective laws $(\Pb_i)_{i \geq 1}$. To avoid trivialities, we suppose that $m_i(G_i)>0$, $\Pb_i$-a.s., for each $i$. For the rest of \cref{sctn:dec tree limits}, we will make the following further assumptions on the sequence $(\Pb_i)_{i \geq 1}$. Where the relevant index $i$ is clear, we write $\Pb$, $\mathbb{E}$ and {Var} for probability, expectation and variance under $\Pb_i$. For any $i \geq 1$, we also let $\# G_i$ denote the number of vertices of $G_i$.

\begin{assn}\label{assn:inserted graph general assumption}
\begin{enumerate}[(1)]
\item For all $i$, we have that
\[\sup_{u, v \in \partial G_i} d_{i}(u,v) \leq i,\qquad \mathbb{P}_i\mbox{-a.s.},\]
where $\partial G_i$ denotes the boundary of $G_i$. Moreover if $\diam(G_i,d_i):=\sup_{u, v \in G_i} d_{i}(u,v)$, then
\[\limsup_{i\rightarrow\infty}\frac{\E{(\diam ((G_i, d_i)))^3}}{i^3}<\infty.\]
\item The measures $(m_i)_{i\geq 1}$ satisfy
\begin{equation}\label{eqn:inserted graphs mean variance}
\limsup_{i\rightarrow\infty}\frac{\Var{m_i (G_i)}}{i}<\infty,\qquad\limsup_{i\rightarrow\infty}\frac{\E{(m_i (G_i))^3}}{i^3}<\infty.
\end{equation}
\end{enumerate}
\end{assn}

\begin{rmk}
It would be sufficient to control a $(2+\epsilon)$-moment in place of the third moment for $i^{-1}\diam ((G_i, d_i))$ and $i^{-1}m_i (G_i)$, but in checking the assumptions in our case it is convenient to control the third moment anyway.
\end{rmk}

In what follows, we let $\td_n$ denote the metric on $\Dec(T_n)$ obtained as in \eqref{eqn:decorated metric def}. We also let $\tnu_n$ denote the measure on $\Dec(T_n)$ as in \eqref{eqn:decorated meas def}, and we let $\trho_n$ be the root inherited from $T_n$. Before we state the main result of this section, we need to introduce some relevant scaling constants.

\begin{defn}[Scaling constants]\label{def:scaling consts}
\begin{enumerate}[(a)]
\item We set
\[\beta = \left( \sum_{i\geq 1} \mu(-i) \E{m_{i+1} (G_{i+1})} \right)^{-1} = \left( (1-\mu_{\bullet}(0)) \sum_{i\geq 1} \mu_{\circ}(i)\E{m_{i+1} (G_{i+1})} \right)^{-1}.\]
We note that \cref{assn:inserted graph general assumption} implies that $\E{m_i (G_i)}\leq Ci$, and thus we can deduce from the basic assumptions on $\mu_\circ$ that the sum in the above expression is finite. Hence $\beta\in(0,\infty)$.
\item For $i \geq 1$, let $\eta_i$ denote the (random) distance between two uniformly chosen vertices on the boundary of $G_i$. Let
\[\chi=\E{\mu_{\bullet}}\sum_{i \geq 1}i\mu_{\circ}(i)\E{\eta_{i+1}}.\]
Note that $\eta_i \leq i$ almost-surely under \cref{assn:inserted graph general assumption}. Hence we can deduce from the basic assumptions on $\mu_\bullet$ and $\mu_\circ$ that $\chi\in (0,\infty)$.
\item The constant $\sigma$ is that appearing in \eqref{eqn:miermont lim}.
\end{enumerate}
\end{defn}

A key input to \cref{thm:main scaling lim} and the main result of this section is the following.

\begin{prop}\label{prop:GHP convergence of enriched trees}
Let $(\Dec (T_n), \td_n, \tnu_n,\trho_n)$ be as described above. As $n\rightarrow\infty$,
\[\left(\Dec (T_n), \frac{\sigma}{\chi\sqrt{n}}\td_n, \frac{\beta}{n} \tnu_n, \trho_n \right) \to (\T, d_\T, \nu_\T, \rho_\T)\]
in distribution with respect to the pointed GHP topology, where the limiting space $(\T, d_\T, \nu_\T, \rho_\T)$ is the Brownian continuum random tree, conditioned to have mass at least $1$.
\end{prop}

\begin{rmk}
The constant scaling factors appearing in \cref{prop:GHP convergence of enriched trees} can be explained as follows. Firstly, $\frac{\sigma}{2}$ is the distance scaling factor for $T_n$, see \autocite[Theorem 2]{miermont2008invariance}. Additionally, the distance between a black vertex and its grandparent will be $2$ in $T_n$, but will on average be $\chi$ in $\Dec (T_n)$, giving an extra factor of $2\chi^{-1}$. The factor $\beta$ appears in the scaling for the measure since $\beta^{-1}$ is the (average) measure of a graph inserted at each tree vertex (if we consider that entire graphs are inserted at white vertices, and that black vertices are ``empty"). Therefore, the law of large numbers roughly entails that this should change volumes by a factor of $\beta$.
\end{rmk}

\subsection{Constructing a correspondence}

The strategy for proving \cref{prop:GHP convergence of enriched trees} is roughly as follows. Under \cref{assn:inserted graph general assumption}, the inserted graphs are small and vanish in the scaling limit. Therefore, it will be sufficient to consider the scaling limit of the boundary of $\Dec(T_n)$, endowed with the metric inherited from $\Dec(T_n)$ (we formalise this later). Such a graph almost fits into the general framework for scaling limits of enriched trees studied by Stufler \autocite[Section 6.8]{StuflerTreeSurvey}, except that the result of \autocite[Section 6.8]{StuflerTreeSurvey} is written for an analogous model for one-type trees. Below, we adapt his proof to our two-type case.

It will be useful in what follows to construct a canonical correspondence between $T_n$ and $\T$. In particular, by \cref{prop:tree offspring RW} the two-type tree $T_n$ is coded by an excursion of the random walk $\Zr$ and its time reversal $Z$ using the coding mechanism illustrated in \cref{fig:contour coding looptree}, and the conditioning in the definition of $T_n$ is equivalent to conditioning this excursion to have length at least $n$. Since we assumed that $\E{\mu_{\bullet}}\E{\mu_{\circ}}=1$ and that $\mu_{\circ}$ and $\mu_{\bullet}$ both have finite exponential moments, the random walk $Z$ is centred and its conditioned excursion converges after rescaling to a Brownian excursion conditioned to have length at least $1$. This allows us to define a correspondence between $T_n$ and $\T$ using coding functions, and we can modify this to check the correspondence $\mathcal{R}_n$ between $\Dec (T_n)$ and $\T$ that we now introduce is of low distortion for large $n$. In the following definition, we assume that all the relevant objects are defined on the same probability space. Moreover, we recall from Section \ref{sctn:CRT def} that $\pi$ is the canonical projection from $[0,\zeta]$ to $\mathcal{T}$.

\begin{defn}[The correspondence $\mathcal{R}_n$]\label{def:correspondence}
A correspondence $\mathcal{R}_n\subseteq \left( \Dec (T_n), \frac{\sigma}{\chi\sqrt{n}}\td_n, \trho_n \right) \times (\mathcal{T}, d_{\T}, \rho_{\T})$ is given as follows. Represent the vertices in $T_n$ as $u_0, \ldots, u_{|T_n|-1}$ according to their labels given by the ordering of the coding function $Z$ as illustrated in \cref{fig:contour coding looptree}. Given $i \in \N_0$ with $u_i\in t_\bullet$ and $s \in [0,\infty)$ such that $\lfloor sn \rfloor = i$, we include $(u_i,\pi(s))$ in $\mathcal{R}_n$. Moreover, if $i \in \N_0$ with $u_i\in t_\circ$ and $s \in [0,\infty)$ such that $\lfloor sn \rfloor = i$, then we include $(v,\pi(s))$ in $\mathcal{R}_n$ for all $v \in G^{(u)}$. Finally, for $sn\geq |T_n|$, include $(u_0,\pi(s))$ in $\mathcal{R}_n$.
\end{defn}

Note that, since $u_0$ is necessarily black, the roots of $\Dec (T_n)$ and $\T$ are included as a pair in the correspondence $\mathcal{R}_n$. Importantly, we have the following coupling result.

\begin{prop}\label{prop:dis to 0 dec tree}
There exists a probability space $(\Omega, \F, \Pb)$ on which $\textsf{dis}(\mathcal{R}_n)\to 0$ in probability.
\end{prop}

We will prove \cref{prop:dis to 0 dec tree} by a series of intermediate lemmas regarding three other distortions. We will now introduce these, together with the basic proof strategy. We note that a very similar argument is used to establish \autocite[Theorem 6.60]{StuflerTreeSurvey}, which is an analogous result for one-type trees. See also \autocite{CurHaasKortCRTDiss,fleurat2023phase} for similar arguments. In what follows, we let $d_n^{\text{tr}}$ denote the graph metric on the two-type tree $T_n$, and consider the rescaled space
\[\left(T_n, \frac{\sigma}{2\sqrt{n}}d_n^{\text{tr}}, \rho_n\right),\]\label{fa dtr}
where $\rho_n$ is the vertex $u_0$ in $T_n$; we define two correspondences between this space and $(\T,d_{\T},\rho_{\T})$.
\begin{itemize}
\item Firstly, we represent the vertices in $T_n$ as $u_0, \ldots, u_{|T_n|-1}$ according to their labels given by the ordering of the coding function $Z$ as illustrated in \cref{fig:contour coding looptree}. Given $i \in \N_0$ with $i<|T_n|$, and $s \in [0,\infty)$ such that $\lfloor sn \rfloor = i$, we let $(u_i,\pi(s))\in \mathcal{R}^*_n$. Moreover, for $s\geq |T_n|$, let $(u_0,\pi(s))\in \mathcal{R}^*_n$.\label{fa Rnstar}
\item Secondly, we represent the vertices in $T_n$ as $v_0, \ldots, v_{|T_n|-1}$ according to their labels given by the depth-first ordering of the \textit{height} function used in \autocite[Section 1.5]{miermont2008invariance}. (This is also defined at the end of \cref{sctn:coding}.)  Given $i \in \N_0$ with $i<|T_n|$, and $s \in [0,\infty)$ such that $\lfloor sn \rfloor =i$, we let $(v_i,\pi(s))\in \mathcal{R}^{**}_n$. Moreover, for $s\geq |T_n|$, let $(v_0,\pi(s))\in \mathcal{R}^{**}_n$.
\end{itemize}
(In the remainder of this section, we will consistently use $u_i$ to denote a vertex labelled with respect to the coding function ordering as in \cref{fig:contour coding looptree}, and $v_i$ to denote a vertex labelled with respect to the height function ordering.) We also let $(T^*_n, d^*_n, \rho_n^*)$ denote the tree with vertex set equal to $t_{\bullet}$ (i.e.\ the restriction to the boundary of $\Dec(T_n)$), and edge set consisting of edges joining black vertices to their grandparents in $T_n$. The metric $d^*_n$ on this space is the metric inherited from $\Dec (T_n)$. In other words, if $u,v$ are neighbours in $T^*_n$, then $d^*_n(u,v)$, the length of the edge between $u$ and $v$ in $T^*_n$, is equal to the distance between the corresponding vertices in $\Dec (T_n)$ (i.e.\ the distance in the graph inserted at the white vertex between $u$ and $v$). The root $\rho_n^*$ is the same as that of $T_n$. Additionally, we define a correspondence $\mathcal{R}^{***}_n$ between $(T_n, \frac{\sigma}{2\sqrt{n}}d_n^{\text{tr}}, \rho_n)$ and
\[\left(T^*_n, \frac{\sigma}{\chi\sqrt{n}}d_n^*, \rho_n^*\right)\]
as follows.

\begin{itemize}
\item If $u \in T_n$ and $u \in t_{\bullet}$, then put $(u,u) \in \mathcal{R}^{***}_n$. If instead $u \in T_n$ and $u \in t_{\circ}$, put $(u,v) \in \mathcal{R}^{***}_n$, where $v$ is the parent of $u$ in $T_n$.\label{fa Rnstarrr}
\end{itemize}
Again since $\rho_n \in t_{\bullet}$ this means that the roots are in correspondence. The first key step of the proof is to show that the vertex ordering determined by our coding function gives rise to a natural correspondence with $\T$ for the underlying two-type tree.

\begin{lem}\label{lem:dis Rn1}
There exists a probability space $(\Omega, \F, \Pb)$ on which $\textsf{Dis} (\mathcal{R}^*_n) \to 0$ almost surely.
\end{lem}
\begin{proof}
It follows from \autocite[Theorem 2]{miermont2008invariance} (note that \cref{assn:offspring} ensures that the required conditions for this theorem are satisfied; in particular we assumed the finiteness of small order exponential moments) and the Skorohod representation theorem that there exists a probability space upon which, almost-surely, the height function of $T_n$ converges after rescaling space by $\frac{\sigma}{2\sqrt{n}}$ and time by $\frac{1}{n}$ towards the excursion coding $\T$, jointly with the convergence of $|T_n|/n$ to its lifetime $\zeta \geq 1$. Note that Miermont rather considers a forest of trees rather than a single tree and shows convergence to a forest of CRTs; however the result straightforwardly adapts to our setting by considering only the first tree in the forest with at least $n$ vertices. By mimicking the proof of \autocite[Lemma 2.4]{LeG2005randomtreesandapplications} using the fact that
\[\left|d_n^{\text{tr}}(v_i, v_j) - \left(H_i + H_j - 2\min_{i \leq k \leq j} H_k\right)\right| \leq 2,\]
where $H$ is the height function defined at the end of \cref{sctn:coding}, it follows that $\textsf{Dis} (\mathcal{R}^{**}_n)$ converges to $0$, almost-surely. (Here we use the height function ordering to label vertices.)

We can then exchange the height function ordering for the ordering of our coding function in two steps as follows. We claim that it is sufficient to show that
\begin{equation}\label{eqn:height osc 0}
\Delta_n := \sup_{i < |T_n|}\frac{\sigma}{\sqrt{n}} d_n^{\text{tr}}(u_i,v_i) \to 0
\end{equation}
almost surely. Indeed, the claim then follows from the conclusion of the previous paragraph since, by the triangle inequality, for any $i, j \leq |T_n|-1$ and $s,t \in [0,|\T|)$ such that $(u_i,\pi(s))$, $(u_j,\pi(t)) \in \mathcal{R}^{*}_n$, we can write
\[\left|\frac{\sigma}{2\sqrt{n}}d_n^{\text{tr}}(u_i,u_j) - d_{\T}(\pi(s), \pi(t))\right| \leq \left|\frac{\sigma}{2\sqrt{n}}d_n^{\text{tr}}(v_i,v_j) - d_{\T}(\pi(s), \pi(t))\right| + \Delta_n.\]

We first prove \eqref{eqn:height osc 0} when $u_i$ is a vertex in $t_{\bullet}$. If $u$ is a vertex in $t_{\bullet}$, note that then its labels according to the two orderings can differ by at most $\text{Height} (T_n)$; this is because the difference between the two labellings of a vertex $u$ will be the number of white ancestors of $u$. Moreover $\text{Height} (T_n)/\sqrt{n}$ is bounded by some (random) $C< \infty$ almost surely on our probability space. Hence, if $u_i$ in the coding ordering is equal to $v_j$ in the height function ordering, then $|i-j|\leq C\sqrt{n}$, and hence
\[\frac{\sigma}{\sqrt{n}}d_n^{\text{tr}}(u_i, v_i) \leq \frac{\sigma}{\sqrt{n}}\left(2+2\sup_{|i-j| \leq C\sqrt{n}}|H_i - H_j|\right).\]
Here, as required, the latter quantity converges to $0$ almost surely by the almost sure convergence of the rescaled height functions to a continuous limit.

It remains to prove the same statement when $u_i\in t_{\circ}$. We let $k$ be the index of the nearest black vertex to $u_i$ with respect to the coding ordering, i.e.\ $k$ is such that
\[|k-i| = \min\{|\ell - i|:\: u_{\ell} \in t_{\bullet}\}\]
(breaking ties arbitrarily). Note that $|k-i|$ is upper bounded by the longest run of consecutive white vertices in the coding order. Since a white vertex corresponds to a down step of the coding function $Z$, and therefore an up step of the random walk $\Zr$, it is possible to deduce from \eqref{eqn:progeny asymp} that there exist constants $c>0$ and $C<\infty$ such that the probability that there exists a run of more than $n^{1/4}$ consecutive white vertices is upper bounded by
\[\pr{|T_n| \geq n^4} + Cn^{9/2}e^{-c{n}^{1/4}} \leq Cn^{-3/2} + Cn^{9/2}e^{-c{n}^{1/4}}.\]
Hence, by Borel-Cantelli, the longest run is bounded by $n^{1/4}$ eventually almost surely. Furthermore, on the event $|k-i| \leq n^{1/4}$, it holds by construction that $d_n^{\text{tr}}(u_i,u_k) \leq 2n^{1/4}$. Moreover, since $u_k \in t_{\bullet}$, we have from \eqref{eqn:height osc 0} that $\frac{\sigma}{\sqrt{n}} d_n^{\text{tr}}(u_k,v_k) \to 0$ (uniformly over the choice of $i$ and $k$). Finally, note that
\[\frac{\sigma}{\sqrt{n}} d_n^{\text{tr}}(v_k,v_i) \leq \frac{\sigma}{\sqrt{n}}\left(2+2\sup_{|i-j| \leq n^{1/4}}|H_i - H_j|\right),\]
which also converges to zero (uniformly over the choice of $i$ and $k$). Combining these three observations gives \eqref{eqn:height osc 0}.
\end{proof}

\begin{rmk}
An alternative way to handle the vertices of $t_\circ$ in the previous proof is to note that $|T_n|-1$ minus the labelling of the white vertices is close to the height function obtained when the tree is explored in depth-first order, but with the reverse lexicographical ordering of children.
\end{rmk}

We next give a lemma that will allow us to handle the inserted graphs.

\begin{lem}\label{lem:max diam}
The probability space $(\Omega, \F, \Pb)$ of \cref{lem:dis Rn1} can be chosen so that the probability of the event
\[\sup_{v \in t_{\circ}(T_n)} \diam(G^{(v)}) < n^{2/5}\]
converges to 1 as $n \to \infty$.
\end{lem}
\begin{proof} Conditionally on $(T_n)_{n\geq 1}$, we sample the additional graphs $G^{(v)}$ for $v\in t_\circ(T_n)$ independently for each $v$ and each $n$. For convenience, we also define $G^{(v)}$ to be empty when $v \in t_{\bullet}(T_n)$, so that $G^{(v)}$ is defined for all $v\in T_n$ (and all $n$). Moreover, let $U_n$ be a uniform vertex of $T_n$. It then holds that
\begin{eqnarray*}
\lefteqn{\prcond{\sup_{v \in t_{\circ}(T_n)} \diam(G^{(v)}) \geq  n^{2/5}}{|T_n|=m}{}}\\
&\leq & \econd{\sum_{v \in T_n}\mathbf{1}_{\{\diam(G^{(v)}) \geq  n^{2/5}\}}}{|T_n|=m}{}\\
&=&m\prcond{\diam(G^{(U_n)}) \geq  n^{2/5}}{|T_n|=m}{}\\
&\leq & Cm\pr{\diam(G_{1+\Zr_1-\Zr_0}) \geq  n^{2/5}}\\
&=&Cm(1-\mu_\bullet(0))\sum_{i\geq 1}\mu_\circ(i)\pr{\diam(G_{i+1}) \geq  n^{2/5}}\\
&\leq & Cmn^{-6/5},
\end{eqnarray*}
where the second inequality follows from \eqref{eqn:Vervaat dis} (and we write $G_{1+\Zr_1-\Zr_0}$ to mean a graph that is, conditionally on the first increment of $\Zr$ being positive, sampled according to the distribution for inserted graphs of boundary length $1+\Zr_1-\Zr_0$, and which is empty if the first increment of $\Zr$ is negative), and the final inequality follows from \cref{assn:offspring}(4) and \cref{assn:inserted graph general assumption}(1) together with Markov's inequality. Hence it follows from a union bound and \eqref{eqn:progeny asymp} that
\begin{align*}
\lefteqn{\pr{\sup_{v \in t_{\circ}(T_n)} \diam(G^{(v)}) \geq  n^{2/5}}}\\
& \leq \pr{|T_n| \geq n^{11/10} } + \sum_{m=n}^{n^{11/10}}\prcond{\diam(G^{(U_n)}) \geq n^{2/5}}{|T_n|=m}{}\pr{|T_n|=m}\\
& \leq C\left(n^{-1/20} + n^{11/10} n^{-6/5}\right),
\end{align*}
which converges to 0, as desired.
\end{proof}

We now turn to the task of comparing the tree metric on $t_\bullet$ with that coming from the decorated tree metric.

\begin{lem}\label{lem:dis Rn3}
The probability space $(\Omega, \F, \Pb)$ of \cref{lem:dis Rn1} can be chosen so that it also holds that $\textsf{Dis} (\mathcal{R}^{***}_n) \to 0$ in probability.
\end{lem}
\begin{proof}
In order to show this, it is sufficient to consider only pairs of vertices of the form $(v,v)$ where $v \in t_{\bullet}$. Indeed, in the case where $u \in T_n$ is in $t_{\circ}$, it is in correspondence with a vertex $u'$ in $T_n^*$, where $u'\in t_\bullet$ is a neighbour of $u$ in $T_n$. For such pairs, $\frac{\sigma}{2\sqrt{n}}d_n^{\text{tr}}(u,u')= \frac{\sigma}{2\sqrt{n}}$, which clearly converges to $0$ as $n \to \infty$. Hence, for any $u, v \in T_n$ (using the notation $v'$ for the vertex in correspondence with $v$) we can write
\begin{align*}
\lefteqn{\left|\frac{\sigma}{2\sqrt{n}}d_n^{\text{tr}}(u,v) - \frac{\sigma}{\chi\sqrt{n}}d_n^*(u',v')\right|}\\
 &\leq \left|\frac{\sigma}{2\sqrt{n}}d_n^{\text{tr}}(u',v') - \frac{\sigma}{\chi\sqrt{n}}d_n^*(u',v')\right| + \frac{\sigma}{2\sqrt{n}}d_n^{\text{tr}}(u,u') + \frac{\sigma}{2\sqrt{n}}d_n^{\text{tr}}(v,v') \\
&= \left|\frac{\sigma}{2\sqrt{n}}d_n^{\text{tr}}(u',v') - \frac{\sigma}{\chi\sqrt{n}}d_n^*(u',v')\right| + \frac{\sigma}{\sqrt{n}}.
\end{align*}
Since $u'$ and $v'$ are necessarily in $t_{\bullet}$, this reduces to the claimed special case of $v \in t_{\bullet}$, or equivalently $v \in T_n^*$.

The claim now follows by applying a similar proof to that given by Stufler of the same result for one-type trees \autocite[Theorem 6.60]{StuflerTreeSurvey}, and appealing to the absolute continuity relation of \eqref{eqn:abs cont Kesten} at the relevant step of the proof. In particular, we claim that for arbitrary $\varepsilon>0$, the following event holds with probability converging to 1.
\begin{equation}\label{eqn:dstar good}
\text{For every pair } u, v \in t_{\bullet}, \text{ it holds that } \left|d_n^*(u,v) - \frac{\chi}{2}d_n^{\text{tr}}(u,v)\right| \leq \epsilon d_n^{\text{tr}}(u,v) \vee n^{1/3}.
\end{equation}
On the event in question, and applying the observation of the previous paragraph, we have that
\[\textsf{Dis} (\mathcal{R}^{***}_n) \leq \frac{\sigma}{\chi\sqrt{n}}\left(\epsilon\diam (T_n,d_n^\text{tr})+n^{1/3}+\chi\right).\]
Since $n^{-1/2}\diam (T_n,d_n^\text{tr})$ converges almost-surely on our probability space to a finite random variable and $\varepsilon$ is arbitrary, the result will indeed follow if we can check \eqref{eqn:dstar good} holds with high probability for each fixed $\varepsilon$.

In fact, by splitting paths between $u,v\in t_\bullet$ at the most recent common ancestor and applying \cref{lem:max diam} to deal with the contribution from the relevant inserted graph if the most recent common ancestor is white, it is sufficient to prove instead that
\begin{equation}\label{eqn:prob to 0 final dis}
\pr{\exists u,v \in t_{\bullet}(T_n): u \preceq v \text{ and the bound at \eqref{eqn:dstar good} does not hold}} \to 0,
\end{equation}
where we write $u\preceq v$ to mean that $u$ is an ancestor of $v$.

Here we will use \eqref{eqn:abs cont Kesten}, which implies that, for an unconditioned critical two-type Galton--Watson tree with root $\rho$,
\begin{align*}
&\E{\sum_{h \geq 0}\sum_{\substack{v \in T: \\ d(\rho, v)=2h}} \prcond{\exists u \in t_{\bullet}(T): u \preceq v \text{ and the bound at \eqref{eqn:dstar good} does not hold}}{T}{}}\\ &= \sum_{h \geq 0} \prstart{\exists u \in t_{\bullet}(T_\infty): u \preceq v_{2h} \text{ and the bound at \eqref{eqn:dstar good} does not hold}}{\infty},
\end{align*}
where, as is defined in \cref{kessec}, $T_{\infty}$ is Kesten's tree, with its law $\Pb_{\infty}$ extended to include the inserted graphs, and $v_h$ is the backbone vertex at distance $h$ from the root in Kesten's tree. Since $T_n$ is conditioned on having at least $n$ vertices which happens with at least probability $cn^{-1/2}$ for some $c>0$ by \eqref{eqn:progeny asymp}, we deduce that
\begin{align}\label{eqn:abs cont sum}
\begin{split}
&\pr{\exists u,v \in t_{\bullet}(T_n): u \preceq v \text{ and the bound at \eqref{eqn:dstar good} does not hold}} \\
&\leq \pr{|T_n| \geq n^2} + Cn^{1/2} \sum_{h = 0}^{n^2} \prstart{\exists u \in t_{\bullet}(T): u \preceq v_{2h} \text{ such that \eqref{eqn:dstar good} does not hold}}{\infty} \\
&\leq Cn^{-1/2} + Cn^{1/2} \sum_{h = 0}^{n^2} \sum_{i=0}^h \prstart{\left|d_n^*(v_{2i},v_{2h}) - \frac{\chi}{2}d_n^{\text{tr}}(v_{2i},v_{2h})\right| \leq \epsilon d_n^{\text{tr}}(v_{2i},v_{2h}) \vee n^{1/3}}{\infty}.
\end{split}
\end{align}
Note that, for a single pair $(v_{2i}, v_{2h})$, it holds that $d_n^{\text{tr}}(v_{2i},v_{2h})=2(h-i)$ and also
\[d_n^*(v_{2i},v_{2h}) \overset{(d)}{=} \sum_{j=1}^{h-i} Y_j,\]
where $(Y_i)_{i=1}^n$ is an i.i.d.\ sequence with law $ d_{Y+1}(U_1,U_2) $, where $Y$ is distributed according to the size-biased law $\hat{\mu}_{\circ}$, and $ d_{Y+1}(U_1,U_2) $ denotes the distance between two uniform vertices on the boundary of a copy of $G_{Y+1}$, i.e.\ an inserted graph with boundary length $Y+1$. By \cref{assn:offspring}(4) and \ref{assn:inserted graph general assumption}(1), it follows that there exists $\lambda>0$ such that $\E{e^{\lambda Y_1}} < \infty$, and moreover one can check that $\E{Y_1}=\chi$. Hence a standard large deviations estimate (see \autocite[Lemma 5.3]{StuflerTreeSurvey}, for example) and a Chernoff bound implies that there exists $c_{\epsilon}>0$ such that
\begin{align*}
\pr{\left|d_n^*(v_{2i},v_{2h}) - \frac{\chi}{2}d_n^{\text{tr}}(v_{2i},v_{2h})\right| \geq \epsilon d_n^{\text{tr}}(v_{2i},v_{2h}) \vee n^{1/3}} \leq e^{ - c_{\epsilon} n^{1/4}}.
\end{align*}
Substituting back into \eqref{eqn:abs cont sum} gives \eqref{eqn:prob to 0 final dis}, as required.
\end{proof}

With the above preparations in place, we are ready to tie up this subsection.

\begin{proof}[Proof of \cref{prop:dis to 0 dec tree}]
The overall result follows from Lemmas \ref{lem:dis Rn1}, \ref{lem:max diam} and \ref{lem:dis Rn3} since
\[\textsf{Dis} (\mathcal{R}_n) \leq \textsf{Dis} (\mathcal{R}^*_n) + \textsf{Dis} (\mathcal{R}^{***}_n) + \frac{\sigma}{\chi\sqrt{n}} \sup_{v \in T_n} \diam(G^{(v)}).\]
\end{proof}

\subsection{Convergence of measures}

Henceforth we will work on the probability space of \cref{prop:dis to 0 dec tree} and therefore assume that $\textsf{dis}(\mathcal{R}_n)\to 0$ in probability. As outlined in \cref{sctn:GHP topology}, this implies that the associated spaces converge in probability with respect to the pointed GH topology. We will extend \cref{prop:dis to 0 dec tree} to pointed GHP convergence by showing that the relevant measures converge on the canonical GH embedding associated with this correspondence. (This canonical embedding was defined in \cref{sctn:GHP topology}.)

To this end, we need to set out various notions. First, given $\epsilon > 0$, take $U \sim \textsf{Uniform}([0,\epsilon])$, independent of the other random variables introduced so far, and define a collection of intervals $(I_k^{\epsilon})_{k \geq 0}$ in $[0, \infty)$ by
\[I_k^{\epsilon} = [U+k\epsilon, U+(k+1)\epsilon]\label{fa IKu}\]
for each $k \geq 0$. We also introduce a measure $\nu^*_n$ on the real line that relates to the measure of interest on $\Dec (T_n)$. Recall from \cref{sctn:looptrees} that the integers in $[0, |T_n|-1]$ are in bijection with the vertices in $T_n$ (black and white), as illustrated in \cref{fig:contour coding looptree}. Recall also that for each $u \in t_{\circ}$, we let $G^{(u)}$ denote the copy of $(G_{\deg u}, d_{\deg u}, \nu_{\deg u}, \ell_{\deg u})$ that is inserted to replace the vertex $u$ in $\Dec (T_n)$. To lighten notation, we will write $X_u$ for $\nu_{\deg u} (G_{\deg u})$, in the specific realisation used as $G^{(u)}$. For a set $A \subset (0, \infty)$, we then define
\begin{equation}\label{eqn:def nu n star}
\nu^*_n (A) = \sum_{i \in \N: \frac{i}{n} \in A} \mathbbm{1}\{u_i \in t_{\circ}\} X_{u_i};
\end{equation}
in other words $\nu^*_n (A)$ is the sum of the volumes of graphs inserted at the white vertices with indices in the set $A$. Also let $\diam_{\T}(\pi(I_k^{\epsilon}))$ denote the diameter of the projection of $I_k^{\epsilon}$ in $\T$ (let this be equal to $0$ if the coding function of $\T$ does not intersect the interval $I_k^{\epsilon}$), and for $A \subset (0, \infty)$, write
\[\tilde{\pi}_n(A) = \left\{x \in \Dec (T_n): \exists s \in A \text{ with } (x,\pi(s)) \in \mathcal{R}_n\right\}.\]\label{fa pin tilde}
We then have the following basic estimates.

\begin{lem}\label{lem:intervals all good for enriched tree convergence}
\begin{enumerate}[(i)]
\item There exists $C<\infty$ such that, for all $\epsilon>0$,
\[\pr{\sup_{k \geq 0} \diam_{\T}(\pi(I_k^{\epsilon})) \leq \epsilon^{1/3}} \geq 1- C\epsilon^{1/2}.\]
\item Fix any $\epsilon > 0$. Then for any $\delta >0$, there exists $N< \infty$ such that for all $n \geq N$,
\[\sup_{k \geq 0} \left|\frac{1}{n} \nu^*_n(I_k^{\epsilon}) - \epsilon \beta^{-1} \right| \leq \delta\]
with probability at least $1-\epsilon$.
\end{enumerate}
\end{lem}
\begin{proof}
\begin{enumerate}[(i)]
\item First note that it follows from standard results (for example, see \autocite[Section 2]{LeGIto}) that there exists $C<\infty$ such that $\pr{\nu (\T) \geq \epsilon^{-1}} \leq C\epsilon^{1/2}$ for all sufficiently small $\epsilon > 0$. Secondly, let $(B_t)_{t \geq 0}$ be a standard one dimensional Brownian motion. Since $(B_t)_{t \geq 0}$ is continuous, it follows from the strong Markov property that there exists $c>0$ such that
\[\pr{ \sup_{t \in [0,\epsilon]} |B_t| \geq \lambda \sqrt{\epsilon}} = \pr{ \sup_{t \in [0,1]} |B_t| \geq \lambda} \leq \pr{ \sup_{t \in [0,1]} |B_t| \geq 1}^{\lfloor \lambda \rfloor} \leq e^{-c\lambda}\]
for all $\lambda \geq 1$. Now let $M = \nu (\T)$ and conditionally on $M$, choose $U \sim \textsf{Uniform}([0,M])$. We work on the event $\{M \leq \epsilon^{-1}\}$. Since $M \geq 1$ by definition, it follows from the Vervaat transform \eqref{eqn:Vervaat cts} that there exist constants $c>0, C<\infty$ such that
\begin{equation}\label{eqn:osc interval}
\pr{ \sup_{t \in [0,\epsilon]} |\B_{U+t}-\B_U| \geq \lambda \sqrt{\epsilon}} \leq Ce^{-c\lambda}
\end{equation}
for all $\lambda \geq 1$, where $\B$ is the excursion coding $\T_{\geq 1}$. Next, for each ${i=1, \ldots, \lceil \epsilon^{-2} \rceil}$, let $U_i = (U + i\epsilon)$, $\textrm{mod}(M)$. Note that the marginal of each $U_i$ is $\textsf{Uniform}([0,M])$. Moreover, since $M<\epsilon^{-1}$, we have that
\[\sup_{k \geq 0} \diam_{\T}(\pi(I_k^{\epsilon})) \leq 2 \sup_{1 \leq i \leq \lceil \epsilon^{-2} \rceil}  \sup_{t \in [0,\epsilon]} |\B_{U_i+t}-\B_{U_i}|.\]
Therefore, setting $\lambda = \epsilon^{-1/6}$ in \eqref{eqn:osc interval} and taking a union bound over $i$, we deduce that
\begin{align*}
\lefteqn{\pr{\sup_{k \geq 0} \diam_{\T}(\pi(I_k^{\epsilon}))\geq \epsilon^{1/3}} }\\
&\leq \pr{ M \geq \epsilon^{-1}} + \pr{\sup_{1 \leq i \leq \lceil \epsilon^{-2}\rceil}  \sup_{t \in [0,\epsilon]} |X_{U_i+t}-X_{U_i}| \geq \frac{1}{2}\epsilon^{1/3}} \\
&\leq C\epsilon^{1/2} + C\epsilon^{-2}e^{-c \epsilon^{-1/6}},
\end{align*}
which yields the result.

\item First sample $M_n := |T_n|-1$, then take $U^*$ uniform on $\{0, \ldots, M_n\}$, independent of everything else, and finally set $U^*_n = \frac{U^*}{n}$. By \eqref{eqn:progeny asymp}, there exists a constant $C<\infty$ such that for any $\eta > 0$, $M_n \leq n\eta^{-1}$ with probability at least $1-C\eta^{1/2}$. Given $M_n$, we apply \eqref{eqn:Vervaat dis} to consider $\nu^*_n([U^*_n, U^*_n+ \epsilon ])$ (we take all time indices modulo $M_n$, so that if $U^* + \epsilon n \geq M_n$, we replace $[U^*_n, U^*_n+ \epsilon]$ with $[U^*_n, \frac{M_n}{n}] \cup [0, U^*_n + \epsilon - \frac{M_n}{n}]$). This corresponds to an interval of length $\epsilon n$ for the random walk coding $T_n$, so we use \eqref{eqn:Vervaat dis} to compare this to the relevant sums on the interval $[0, \epsilon n]$ for an unconditioned random walk. Each jump on this interval is i.i.d.\ with law $\mu$, where $\mu$ is as in \eqref{eqn:mu def general case}, and given that the $i^{th}$ jump has size $-j$, the associated measure of the corresponding graph satisfies
    \[X_{u_i} \overset{(d)}{=} \nu_{j+1} (G_{j+1}),\]
    and also the asymptotics of \eqref{eqn:inserted graphs mean variance}. For $j \geq 1$ we define $m_j$ to be the mean of an inserted graph with boundary length $j$, i.e.\ $m_j = \nu_{j} (G_{j})$ and extend this notation to non-positive $j$ by setting $m_j=0$ for all $j \leq 0$. We then have for any $i \in [0, \lfloor \epsilon n \rfloor]$ that
\begin{align*}
\E{X_{u_i} \mathbbm{1}\{u_i \in t_{\circ}\}} &= \E{ \econd{X_{u_i} \mathbbm{1}\{u_i \in t_{\circ}\}}{Z_i-Z_{i-1}}} \\
&= \E{m_{\Zr_1-\Zr_0+1}} \\
&= \sum_{k \geq 1}\mu(-k)m_{k+1}\\
&= \beta^{-1},
\end{align*}
where $\beta$ was defined in \cref{def:scaling consts}. (Note that the final sum is finite since $\mu$ is critical and $\frac{m_k}{k}$ is bounded as $k \to \infty$ by \cref{assn:inserted graph general assumption}(2).) It therefore follows from \eqref{eqn:Vervaat dis} and the strong law of large numbers that for any $\delta>0$, we have that
\[\pr{ \left|\frac{1}{n} \nu^*_n([U^*_n, U^*_n+ \epsilon ]) - \epsilon \beta^{-1}\right|\geq  \delta} \to 0.\]
(Note that it also follows from \eqref{eqn:inserted graphs mean variance} that any error obtained from the fact that $\epsilon n$ is not an integer is negligible.) Since $(U^*+ k\epsilon n) \mod M_n$ has the same marginal distribution as $U^*$ for any $k \geq 1$, we deduce that the same bound holds for the quantity $\nu^*_n([U^*_n+ k\epsilon , U^*_n+(k+1)\epsilon ])$ for all $0 \leq k \leq \lceil \frac{M_n}{\epsilon n} \rceil$ (again taking quantities modulo $M_n$ as appropriate). Taking a union bound and conditioning on the value of $M_n$, it follows that
\begin{equation*}
\pr{ \exists k \leq \left\lceil \frac{M_n}{\epsilon n} \right\rceil: \left|\frac{1}{n} \nu^*_n([U^*_n+ k \epsilon , U^*_n+(k+1) \epsilon ]) - \epsilon \beta^{-1}\right|\geq  \delta} \to 0.
\end{equation*}
Moreover, the event in the above display is unchanged on replacing $U^*$ with $U^{**} = U^* \mod \epsilon n $. Finally, note that we can couple $U^{**}$ and $\lfloor nU\rfloor$ (recall that $U \sim \textsf{Uniform}([0, \epsilon])$ so that they are equal with probability at least $1-\frac{\epsilon}{2}$.
\end{enumerate}
\end{proof}

We are now ready to prove that the measures converge, which is the last ingredient required to establish \cref{prop:GHP convergence of enriched trees}.

\begin{proof}[Proof of \cref{prop:GHP convergence of enriched trees}]
By \cref{prop:dis to 0 dec tree}, we can work on $(\Omega, \F, \Pb)$ and assume that $r_n := \textsf{dis} (\mathcal{R}_n) \to 0$ in probability. To prove that the claimed measures also converge on this space, set $\tT_n = \frac{\sigma}{\chi\sqrt{n}} \Dec (T_n)$ and take the canonical Gromov-Hausdorff embedding $F_n = \tT_n \sqcup \T$ endowed with the metric
\[D_{F_n}(x,y) = \begin{cases}
\frac{\sigma}{\chi \sqrt{n}}\td_n, & \text{ if } x, y \in \tT_n, \\
d(x,y) ,& \text{ if } x, y \in \T ,\\
\inf_{u, v \in \mathcal{R}_n} (\frac{\sigma}{\chi  \sqrt{n}}\td_n + d(y,v) + \frac{1}{2} r_n), & \text{ if } x \in \tT_n, y \in \T,
\end{cases}\]\label{fa canonical GHP}
as defined in \cref{sctn:GHP topology}. We claim that $d^{F_n}_P(\frac{\beta}{n}\tnu_n, \nut) \rightarrow 0$ in probability as $n \rightarrow \infty$.

For this proof, given a set $A \subset F_n$ and $\epsilon>0$, the notation $A^{\epsilon}$ will refer to the $\epsilon$-fattening of $A$ with respect to the metric $D_{F_n}$. Furthermore, recall from just above \cref{lem:intervals all good for enriched tree convergence} that for $\epsilon > 0, k \in \Z$, we defined the interval $I_k^{\epsilon}$ by
\[I_k^{\epsilon} = [U+k\epsilon, U+(k+1)\epsilon].\]
Fix $\epsilon>0$. For a set $A \subset \Dec (T_n)$, we write $k \in \mathcal{R}_n(A)$ if there exists $s \in I_{k,\epsilon}$ and $x \in A$ with $(x, \pi (s)) \in \mathcal{R}_n$. We also let
\[D_n = \frac{\sigma}{\chi}\frac{1}{\sqrt{n}}\sup_{u \in t_{\circ}} \diam (G^{(u)})\label{fa diam Dn},\]
the largest rescaled diameter of an inserted graph. Note that under \cref{assn:inserted graph general assumption}, $D_n \to 0$ in probability as $n \to \infty$ by \cref{lem:max diam}. We will therefore work on the event $2D_n \leq \epsilon$ for the rest of this proof.

\noindent
\textbf{Step 1: $\frac{\beta}{n}\tnu_n(A) \leq \nu_{\T}(A^{\epsilon})+\epsilon$ for all $A \subset \Dec (T_n)$.}\\
Take a set $A_n$ of vertices in $\Dec (T_n)$, and set
\[A'_n = \bigcup_{k \in \mathcal{R}_n(A_n^{2D_n})} I_{k, \epsilon}.\]
Let $A_n'' = \pi(A'_n)$. (Recall that $\pi$ denotes the projection from $[0,\infty)$ onto $\T$.) We will show that there exists an event of probability at least $1-C\epsilon^{1/2}-o(1)$ firstly such that $A_n'' \subset A_n^{2\epsilon^{1/3}}$ for all possible choices of $A_n$, and secondly that $\frac{\beta}{n}\tnu_n (A_n) \leq \nut (A_n'') + \epsilon$, which therefore implies that $\frac{\beta}{n}\tnu_n (A_n) \leq \nut (A_n^{2\epsilon^{1/3}}) + \epsilon$ for any choice of $A_n$, provided that $n$ is large enough. For the first of these, note that it follows from the definitions above that for any $x \in A_n''$ there exists $t \in A_n'$ with $\pi(t) = x$, and therefore a $k$ such that
\begin{enumerate}[(1)]
\item $t \in I_k^{\epsilon} \subset A_n'$;
\item there additionally exists $s \in I_k^{\epsilon}$ either  such that $u_{\lfloor ns \rfloor}$ is white and $G^{(u_{\lfloor ns \rfloor })} \cap A_n^{2D_n} \neq \emptyset$, or otherwise such that $u_{\lfloor ns \rfloor}$ is black and $u_{\lfloor ns \rfloor }\cap A_n^{2D_n} \neq \emptyset$.
\end{enumerate}
By definition, if $y$ is the aforementioned representative in $G^{(u_{\lfloor ns \rfloor })} \cap A_n^{2D_n}$ or $u_{\lfloor ns \rfloor }\cap A_n^{2D_n}$ for this specific value of $s$ (as guaranteed to exist by point (2) above), this means that $(y, \pi(s)) \in \mathcal{R}_n$ and therefore $D_F(\pi(s),y) = \frac{1}{2} r_n$. Moreover, by \cref{lem:intervals all good for enriched tree convergence}, there exists $C<\infty$ such that
\begin{equation}\label{eqn:diameter bound interval reminder}
\sup_{k \geq 0} \diam_{\T}(\pi(I_k^{\epsilon})) \leq \epsilon^{1/3}
\end{equation}
with probability at least $1- C\epsilon^{1/2}$. On this event, we also have that $D_F(\pi(s),\pi(t)) \leq \epsilon^{1/3}$ and therefore $D_F(x,y) = D_F(\pi(t),y) \leq \epsilon^{1/3} + \frac{1}{2} r_n$ by the triangle inequality. In other words, $x \in A_n^{2D_n+r_n+\epsilon^{1/3}}$. Moreover, since the event in \eqref{eqn:diameter bound interval reminder} does not depend on the choice of $A_n$ or any of the vertices in it, this therefore implies that $A_n'' \subset A_n^{2D_n+r_n+\epsilon^{1/3}} \subset A_n^{2\epsilon^{1/3}}$ (for all sufficiently large $n$ and under the above assumption on $D_n$).

We now turn to checking $\frac{\beta}{n}\tnu_n (A_n) \leq \nut (A_n'') + \epsilon$. Recall that, for $A \subset (0, \infty)$, we write $\tilde{\pi}_n(A) = \{x \in \Dec (T_n): \exists s \in A \text{ with } (x,\pi(s)) \in \mathcal{R}_n\}$. Note that $A_n \subset \tilde{\pi}_n(A_n')$ by definition. Moreover, if $x \in A_n \cap t_{\bullet}$, then the structure of the coding bijection between $T_n$ and its excursion ensures that there exists $u\in t_\circ$ with $x \in G^{(u)}\subseteq A_n^{2D_n}$. By the definition of $A_n'$, it therefore follows that
\[\tnu_n (A_n) \leq \nu^*_n (A_n').\]
Moreover, by \cref{lem:intervals all good for enriched tree convergence}(ii), $\sup_{k \geq 0} \left|\frac{1}{n} \nu_n^* (I_k^{\epsilon}) - \epsilon \beta^{-1} \right| <\varepsilon^2$ with high probability for large $n$, and since $A_n'$ is a disjoint union of intervals of the form $I_k^{\epsilon}$ and $\text{Leb}(I_k^{\epsilon}) = \epsilon$ for all $k$, this implies that
\[\frac{1}{n}\tnu_n (A_n) \leq \frac{1}{n}\nu^*_n (A_n') \leq \text{Leb} (A_n')\beta^{-1} + \epsilon  = \nut (A_n'')\beta^{-1} + \epsilon,\]
whenever $\nut (\T) \leq \epsilon^{-1}$, which happens with probability at least $1 - C\epsilon^{1/2}$. In particular, since $\epsilon>0$ was fixed, this implies the result for all large enough $n$.

\noindent
\textbf{Step 2: $\nu_{\T}(B) \leq \frac{\beta}{n}\tnu_n(B^{\epsilon}) +\epsilon$ for all $B \subset \Dec (T_n)$.}\\
For the reverse direction, take a set $B \subset \T$. Set $B' = \pi^{-1} (B)$, and let
\[B'' = \bigcup_{k: I_k^{\epsilon} \cap B' \neq \emptyset} I_k^{\epsilon}.\]
On the event in \eqref{eqn:diameter bound interval reminder} and by definition of the correspondence, we have that
\[\tilde{\pi}_n(B'') \subset \pi(B'')^{r_n} \subset B^{\epsilon^{1/3} + r_n}.\]
Clearly this implies that $\tilde{\pi}_n(B'')^{D_n} \subset B^{\epsilon^{1/3} + r_n+D_n}$. Moreover, similarly to the previous step, if $x \in \tilde{\pi}_n(B'') \cap t_{\bullet}$, it follows that there exists $u\in t_\circ$ such that $x\in G^{(u)}\subseteq \tilde{\pi}_n(B'')^{D_n}$, and hence $\nu^*_n (B'') \leq \tnu_n (\tilde{\pi}_n(B'')^{D_n})$. Therefore, again because $B''$ is the disjoint union of intervals of the form $I_k^{\epsilon}$ and by definition of $\nu_n^*$, we have on the events $\nut (\T) \leq \epsilon^{-1}$ and $\sup_{k \geq 0} \left|\frac{1}{n} \nu_n^* (I_k^{\epsilon}) - \epsilon \beta^{-1} \right| <\varepsilon^2$ that
\[\nut (B) \leq \text{Leb}(B'') \leq \frac{\beta}{ n} \nu^*_n (B'') + \epsilon  \leq \frac{\beta}{ n} \tnu_n (\tilde{\pi}_n(B'')^{D_n}) + \epsilon \leq \frac{\beta }{n} \tnu_n (B^{\epsilon^{1/3} + r_n+D_n}) + \epsilon\]
with high probability as $n \to \infty$. This completes the proof.
\end{proof}

\section{Proof of \cref{thm:main scaling lim}}\label{sctn:main result proof}

In this section, we will use the results of \cref{sctn:dec tree limits} to prove \cref{thm:main scaling lim}. In order to do this we will use \cref{def:dec tree construction of cluster} to show that the cluster is close to a decorated two-type Galton-Watson tree, so that the result will essentially follow from \cref{prop:GHP convergence of enriched trees}.

We start by taking a random walk $Z$ and the cluster $\CC_Z$ as defined in \cref{def:dec tree construction of cluster} (so $Z$ corresponds to the peeling exploration process and therefore starts at $1$, has step distribution $\mu$ as in \eqref{eqn:mu def}, and is killed the first time it hits the non-positive integers). Note that, by Propositions \ref{prop:peeling probs} and \ref{prop:tree distribution} and Remark \ref{exprem}, Assumption \ref{assn:offspring} holds for the associated two-type offspring distribution appearing in \cref{prop:tree distribution}. Set $\tau = \inf\{ n \geq 0: Z_n \leq 0\}$. We will first prove the scaling limit conditionally on the event $\{\tau \geq \beta n\}$, where $\beta$ is the constant of \cref{def:scaling consts}, which appears naturally in the scaling and asymptotically determines the ratio of the excursion length $\tau$ to the total size of the cluster $\CC_Z$. A key point is that an excursion of $Z$ looks very similar to the time-reversal of the excursion of $\Zr$ as appearing in \cref{prop:tree offspring RW}. It is not \textit{exactly} equal in law because the killing conditions for $Z$ and $\Zr$ are not completely symmetric, however on conditioning their lifetimes to be large the difference becomes negligible.
\label{fa tau}

Our strategy will be to first verify that \cref{assn:inserted graph general assumption} holds for the inserted percolated $i$-gons (this is done in \cref{prop:verify assn}), so that we can immediately apply \cref{prop:GHP convergence of enriched trees} to the associated decorated two-type Galton-Watson with offspring distributions $\mu_{\bullet}$ and $\mu_{\circ}$ appearing in \cref{prop:tree distribution}. We will then verify in \cref{prop:dGHP small under coupling} that $\CC_Z$ is very close to (a similarly conditioned version of) this decorated tree under the event $\{\tau \geq \beta n\}$. Finally, in \cref{prop:switch final conditioning} we will switch this to a conditioning on the event $\{\CC_Z \geq n\}$ by verifying the conditions of \cref{lem:switch conditioning}.

\begin{prop}\label{prop:verify assn}
For each $i$, suppose $G_i$ is a $p_c$-percolated $i$-gon, selected according to $\mathbb{P}_{i,q}$, where $q=\frac12\alpha^2(1-\alpha)$, with black boundary condition. Let $d_i$ be the graph metric on $G_i$ and $\nu_i$ be the measure on $G_i$ given by Remark \ref{listrem}(4), where the labelling $\ell_i$ of boundary vertices is chosen uniformly amongst the possibilities. The family $((G_i,d_i,\nu_i,\ell_i))_{i\geq 1}$ then satisfy \cref{assn:inserted graph general assumption}.
\end{prop}
\begin{proof}
Note that the first part of point 1 of \cref{assn:inserted graph general assumption} is satisfied trivially due to the loop boundary structure. The third moment bound for the diameter will follow from that of the volume in point 2, since $\nu_i(G_i)$ is a trivial upper bound for $\diam(G_i, d_i)$ in this case. As for point 2, we have from \autocite[Proposition 3.4]{ray2014geometry} that
\[\mathrm{Var}\left(\nu_i(G_i)\right)\leq Ci,\]
which gives the first requirement. For the third moment bound, one can follow the same strategy as in \autocite[Proposition 3.4]{ray2014geometry}. Indeed, we have that
\[\mathbb{E}\left(|G_i\backslash\partial G_i|\left(|G_i\backslash\partial G_i|-1\right)\left(|G_i\backslash\partial G_i|-2\right)\right)=\frac{q^3Z_i'''(q)}{Z_i(q)},\]
where $Z_i$ is the generating function for the internal vertices of a map with $i$ boundary vertices, as given in \autocite[Definition 3.2]{ray2014geometry}. From this, the result is a consequence of the exact expression for $Z_i$ given in \autocite[Proposition 3.3]{ray2014geometry}; we leave the details to the reader.
\end{proof}

To show that the discrepancy between $\CC_{Z}$ and a decorated tree is small, as we will soon do in \cref{prop:dGHP small under coupling}, the following description of the limiting distribution of the final increment of $Z$ and length of the additional path section on the negative axis will be helpful.

\begin{lem}\label{lem:RW error small}
Conditionally on $\tau\geq n$, as $n\rightarrow\infty$, the pair $(Z_{\tau-1},-Z_\tau)$ converges in distribution to a pair of random variables $(L,K)$ taking values in $\mathbb{N}\times\N_0$. In particular, we have that
\[\mathbb{P}\left((L,K)=(l,k)\right)\propto l\overleftarrow{\mu}\left(l+k\right)\propto lp_{l+k},\qquad (l,k)\in \mathbb{N}\times\N_0,\]
where $\overleftarrow{\mu}$ is as in \eqref{eqn:mu def general case background reversed}, and $p_i$ are the peeling probabilities, as described in Proposition \ref{prop:peeling probs}.
\end{lem}
\begin{proof} We first note that, by \autocite[Proposition 2.1]{Iglehart} (see also \autocite{Bolt,DW}), for all $x\geq 0$,
\[\mathbb{P}\left(\frac{Z_n}{\sigma\sqrt{n}}\leq x\:\vline\:\tau\geq n\right)\rightarrow 1-e^{-\frac{x^2}{2}},\]
where $\sigma^2$ is the variance of the increment distribution of $Z$. In particular, it holds that, conditionally on $\tau\geq n$, $Z_n$ converges in distribution to $\infty$. Applying the Markov property at time $n$, it is a straightforward consequence of this observation that: for each $(l,k)\in \mathbb{N}\times\N_0$,
\begin{equation}\label{zlim}
\lim_{n\rightarrow\infty}\mathbb{P}\left((Z_{\tau-1},-Z_\tau)=(l,k)\:\vline\:\tau\geq n\right)=\lim_{z\rightarrow\infty}\mathbb{P}\left((Z_{\tau-1},-Z_\tau)=(l,k)\:\vline\:Z_0=z\right),
\end{equation}
if the limit on the right-hand side exists. To check that the right-hand side of \eqref{zlim} is indeed well-defined, we will apply \autocite[Theorem 4]{DK} to the random walk $(z-Z_m)_{m \geq 0}$ (started from 0), which has the same increment distribution as $\overleftarrow{Z}$. In particular, the quantities $(Z_{\tau-1}-1,-Z_\tau+1)$ under the conditioning $Z_0=z$ are distributed as the undershoot and overshoot of $z-1$ by $\overleftarrow{Z}$ under the conditioning $\Zr_0=0$, as defined in \autocite{DK}. Hence we obtain from \autocite[Theorem 4]{DK} that, for $z\geq l$,
\begin{align*}
\lefteqn{\mathbb{P}\left((Z_{\tau-1},-Z_\tau)=(l,k)\:\vline\:Z_0=z\right)}\\
&=\sum_{y=0}^{l-1}\left[\mathbb{E}_0\left[\left|\left\{n:\Zr_n = \max_{m\leq n}\overleftarrow{Z}_m=z-1-y\right\}\right|\right] \mathbb{P}_0\left(\exists n: \min_{m\leq n}\overleftarrow{Z}_m=y-l+1\right) \overleftarrow{\mu}\left(l+k\right)\right].
\end{align*}
Now, since the increments of the random walk $\overleftarrow{Z}$ have second moments, the first strict ascending ladder height for $\overleftarrow{Z}$ has finite expectation (by \autocite[(4a)]{Don}), and so we can apply the renewal theorem to the set of ascending ladder values, and the fact that the number of returns of $\overleftarrow{Z}$ to 0 before it hits $\N$ is geometric, supported on $\N_0$, with parameter $\mathbb{P}(\overleftarrow{Z}\mbox{ does not return to 0 before hitting }\mathbb{N})\in (0,1)$, to deduce that, for any $y \in \{0 ,\ldots, l-1\}$,
\begin{eqnarray*}
\lefteqn{\lim_{z\rightarrow\infty}\mathbb{E}_0\left[\left|\left\{n:\:\Zr_n = \max_{m\leq n}\overleftarrow{Z}_m=z-1-y\right\}\right|\right]}\\
&=&\lim_{z\rightarrow\infty}\mathbb{P}_0\left(\max_{m\leq n}\overleftarrow{Z}_m=z-1-y\mbox{ for some }n\right)\\
&&\qquad\times \mathbb{P}_0\left(\overleftarrow{Z}\mbox{ does not return to 0 before hitting }\mathbb{N}\right)^{-1}\\
&=&\mathbb{E}_0\left[\overleftarrow{Z}_{\inf\left\{n:\:\overleftarrow{Z}_n>0\right\}}\right]^{-1}\mathbb{P}_0\left(\overleftarrow{Z}\mbox{ does not return to 0 before hitting }\mathbb{N}\right)^{-1},
\end{eqnarray*}
with the final expression taking a value in $(0,\infty)$ and in particular not depending on $l$. Moreover, since $\overleftarrow{Z}$ is downwards-skip-free and oscillating, it holds that
\[\mathbb{P}\left(\min_{m\leq n}\overleftarrow{Z}_m=y-l+1\mbox{ for some }n\right)=1,\qquad y\leq l-1.\]
Hence we deduce that
\[\lim_{z\rightarrow\infty}\mathbb{P}\left((Z_{\tau-1},-Z_\tau)=(l,k)\:\vline\:Z_0=z\right)=Cl\overleftarrow{\mu}\left(l+k\right).\]
We obtain the alternative expression given in the statement of the lemma by noting that Proposition \ref{prop:C is centred RW} gives that $\overleftarrow{\mu}(l+k)={\mu}(-l-k)\propto p_{l+k}$.
\end{proof}

We next give the description of $\CC_Z$ that is crucial for our proof of Theorem \ref{thm:main scaling lim}.

\begin{lem}\label{lem:exc to tree}
Under the conditioning $\{\tau \geq n, Z_{\tau} = -k,  Z_{\tau - 1}=l\}$, $\CC_Z$ consists of a decorated `root loop' of size $k+l+1$, divided into $l$ consecutive black vertices and $k+1$ consecutive white vertices. Moreover, attached to each of the black vertices is a decorated two-type Galton-Watson tree with offspring distributions as in Proposition \ref{prop:tree distribution}, conditioned so that the sum of the total number of vertices of the underlying trees (including their roots) is at least $n$.
\end{lem}
\begin{proof} We will prove a version of the result for the excursion $(Z_m)_{m=0}^{\tau}$ which translates easily to the statement for $\CC_Z$. First, let $(\gamma_m)_{m=0}^{|\gamma|}$ be a possible realisation of $(Z_m)_{m=0}^{\tau}$ on the set $\{\tau\geq n,\:Z_{\tau-1}=l,\:Z_{\tau}=-k\}$ for some $n,l\geq 1$, $k\geq 0$ (and assuming $Z_0=1$), and extend $\gamma$ by setting $\gamma_{-1}=0$. Moreover, we introduce the time-reversal $\overleftarrow{\gamma}$ by setting
\[\overleftarrow{\gamma}_m=\gamma_{|\gamma|-1-m},\qquad m=0,\dots,|\gamma|.\]
It is then the case that $\overleftarrow{\gamma}$ is a possible realisation of the centred random walk of \cref{sctn:coding}, $\overleftarrow{Z}$, started from $\overleftarrow{Z}_0=l$ and run until $\sigma:=\inf\{m\geq0:\:\overleftarrow{Z}_m=0\}$. Since all the downwards jumps of $\overleftarrow{Z}$ are of unit size, we can decompose $\overleftarrow{\gamma}$ into excursions $\overleftarrow{\gamma}^{(i)}=(\overleftarrow{\gamma}^{(i)}_m)_{m=0}^{\Gamma_i}$, $i=1,\dots,l$, where
\[\overleftarrow{\gamma}^{(i)}_m=\overleftarrow{\gamma}_{\Gamma_1+\Gamma_2+\dots+\Gamma_{i-1}+m},\qquad m=0,\dots,\Gamma_i,\]
with $\Gamma_i$, $i=1,\dots,l$, being defined inductively by setting
\[\Gamma_1+\dots+\Gamma_i=\inf\{m\geq 0:\:\overleftarrow{Z}_m=l-i\}.\]
In particular, we have that $\Gamma_1+\dots +\Gamma_l=|\gamma|$. We note that, for each $i$, $\overleftarrow{\gamma}^{(i)}-(l-i)$ is a possible realisation of $\overleftarrow{Z}$, started from $\overleftarrow{Z}_0=1$ and run until $\sigma:=\inf\{m\geq0:\:\overleftarrow{Z}_m=0\}$. With this notation, for fixed $n,l\geq 1$, $k\geq 0$, we have, using the proportionality symbol $\propto$ with respect to the possible $\gamma$ (the collection of which is the same as the set of possible $\overleftarrow{\gamma}^{(i)}$, $i=1,\dots,l$, with $\Gamma_1+\dots+\Gamma_l\geq n$),
\begin{eqnarray*}
\lefteqn{\mathbb{P}\left((Z_m)_{m=0}^{\tau-1}=(\gamma_m)_{m=0}^{|\gamma|-1}\:\vline\:\tau\geq n,\:Z_{\tau-1}=l,\:Z_{\tau}=-k\right)}\\
&\propto&\prod_{m=0}^{|\gamma|-1}\mu\left(\gamma_m-\gamma_{m-1}\right)\\
&=&\prod_{i=1}^{l}\prod_{m=1}^{\Gamma_i}\overleftarrow{\mu}\left(\gamma^{(i)}_m-\gamma^{(i)}_{m-1}\right)\\
&\propto&\prod_{i=1}^{l}\mathbb{P}\left((\overleftarrow{Z}_m)_{m=0}^{\sigma}=\gamma^{(i)}-(l-i)\:\vline\:\overleftarrow{Z}_0=1\right),
\end{eqnarray*}
where we use that the increments from $\gamma_{-1}=0$ to $\gamma_0=1$ and from $\gamma_{|\gamma|-1}=l$ to $\gamma_{|\gamma|}=-k$ are fixed. In particular, this tells us that, conditionally on $\{\tau\geq n,\:Z_{\tau-1}=l,\:Z_{\tau}=-k\}$, the time-reversal of $(Z_m)_{m=0}^{\tau-1}$ is obtained by concatenating $l$ independent copies of $(\overleftarrow{Z}_m)_{m=0}^{\sigma}$ with $\overleftarrow{Z}_0=1$, conditioned so that the total length of these excursions is at least $n$, and then removing the last step. The result readily follows from this equivalence of conditional laws by considering the description of $\CC_Z$ given in Definition \ref{def:dec tree construction of cluster} and the looptree construction of Section \ref{sctn:coding}.
\end{proof}

As a consequence, we obtain the basic properties set out in the following proposition. To state the result, we will continue to write $(\Gamma_i)_{i=1}^{Z_{\tau-1}}$ for the lengths of the excursions of $\overleftarrow{Z}$, as defined in the previous proof. We will denote by $\Gamma_{\max}$ the maximum of these, and $i_{\max}$ the index that attains this (chosen arbitrarily from the possible candidates if there is a tie). Let us also introduce $(T_i)_{i=1}^{Z_{\tau-1}}$ for the collection of corresponding decorated two-type trees (i.e. the tree $T_i$ is coded by the $i^{th}$ excursion as illustrated in \cref{fig:contour coding looptree} and decorated with the family $((G_i,d_i,\nu_i,\ell_i))_{i\geq 1}$ appearing in \cref{prop:verify assn}), write $d_{T_i}$ for their graph metric, and set $|T_i|$ to be their size (i.e.\ number of vertices). Similarly, denote by $L_{\rho}$ the decorated root loop, $d_{L_{\rho}}$ its graph metric and $|L_{\rho}|$ its size.

\begin{cor}\label{cor:one big tree}
Under the conditioning $\{\tau \geq \beta n\}$, the following events occur with high probability as $n \to \infty$:
\begin{enumerate}[(i)]
\item precisely one of the Galton-Watson trees in the decomposition above will have size $\tau - o(n)$, and all of the other decorated Galton-Watson trees will have diameter $o(n^{1/2})$ and total volume $o({n})$, in particular, for any divergent sequence $(\alpha(n))_{n\geq 1}$,
\[\prcond{\tau-\Gamma_{\max}=\sum_{i\in\{1,\dots,Z_{\tau-1}\}\backslash\{i_{\max}\}}|T_i|\geq \alpha(n)}{\tau\geq \beta n}{}\rightarrow 0,\]
\[\prcond{\max_{i\in\{1,\dots,Z_{\tau-1}\}\backslash\{i_{\max}\}}\diam (T_i,d_{T_i})\geq \alpha(n)}{\tau\geq \beta n}{}\rightarrow 0;\]
\item the decorated root loop will have diameter $o(n^{1/2})$ and total volume $o({n})$, in particular, for any divergent sequence $(\alpha(n))_{n\geq 1}$,
\[\prcond{|L_{\rho}|\geq \alpha(n)}{\tau\geq \beta n}{}\rightarrow 0,\]
\[\prcond{\diam (L_{\rho},d_{L_{\rho}})\geq \alpha(n)}{\tau\geq \beta n}{}\rightarrow 0.\]
\end{enumerate}
\end{cor}
\begin{proof}
We use the same notation as in the proof of \cref{lem:exc to tree}.
\begin{enumerate}[(i)]
\item We start by showing the following weaker fact: precisely one of the Galton-Watson trees in the decomposition above will have size $\tau - o(n)$, in particular,
\begin{equation}\label{eqn:5.4i}
\prcond{\tau-\Gamma_{\max}\geq n^{5/8}}{\tau\geq \beta n}{}\rightarrow 0;
\end{equation}

Since the pair $(Z_{\tau - 1},-Z_{\tau})$ converges in distribution to a pair of finite random variables under the relevant conditioning (by Lemma \ref{lem:RW error small}),
    \[\prcond{Z_{\tau - 1}-Z_{\tau}> \log n}{\tau \geq \beta n}{}\rightarrow 0.\]
    Moreover, writing $\Gamma_{1,l}:=\Gamma_1+\dots+\Gamma_l$, for any $l\geq 1$ and $k\geq 0$, it holds that there exists $C<\infty$ such that
    \begin{eqnarray*}
    \lefteqn{\prcond{\tau-\Gamma_{\max}\geq n^{5/8}}{\tau\geq \beta n,(Z_{\tau-1},-Z_{\tau})=(l,k)}{}}\\
    &=&\frac{\prcond{\Gamma_{1,l}\geq \beta n,\: \Gamma_{1,l}-\Gamma_{\max}\geq n^{5/8}}{(Z_{\tau-1},-Z_{\tau})=(l,k)}{}}{\prcond{\Gamma_{1,l}\geq \beta n}{(Z_{\tau-1},-Z_{\tau})=(l,k)}{}}\\
    &\leq &\frac{\prcond{\max_{i\in\{1,\dots,l\}\backslash\{i_{\max}\}} \Gamma_{i}\geq n^{5/8}/l}{(Z_{\tau-1},-Z_{\tau})=(l,k)}{}}{\prcond{\Gamma_{1}\geq \beta n}{(Z_{\tau-1},-Z_{\tau})=(l,k)}{}}\\
    &\leq & C l^3n^{-1/8},
    \end{eqnarray*}
    where to deduce the final inequality we have applied a union bound (based on the fact that for the event in the numerator to be true, we require two excursions to be of length at least $n^{5/8}/l$) and \eqref{eqn:progeny asymp 2}. (Note that, conditionally only on the event $\{(Z_{\tau-1},-Z_{\tau})=(l,k)\}$, the excursions associated with lengths $(\Gamma_i)_{i=1}^l$ are independent.) Hence
    \begin{eqnarray*}
    \lefteqn{\prcond{\tau-\Gamma_{\max}\geq n^{5/8}}{\tau\geq \beta n}{}}\\
    &= &o(1)+\sup_{l,k:\:l+k\leq \log n}\prcond{\tau-\Gamma_{\max}\geq n^{5/8}}{\tau\geq \beta n,(Z_{\tau-1},-Z_{\tau})=(l,k)}{}\\
    &\leq &o(1)+C (\log n)^3n^{-1/8},
    \end{eqnarray*}
    which converges to 0 as $n\rightarrow\infty$.

We now use \eqref{eqn:5.4i} to establish (i). For any $\Lambda>0$ we have that
\begin{eqnarray*}
    \lefteqn{\prcond{\sum_{i\in\{1,\dots,Z_{\tau-1}\}\backslash\{i_{\max}\}}|T_i|\geq \alpha(n)}{\tau\geq \beta n}{}}\\
    &\leq &\prcond{Z_{\tau - 1}-Z_{\tau}> \Lambda}{\tau \geq \beta n}{}\\
    && + \sum_{l,k:\:l+k\leq \Lambda} \prcond{\max_{i\in\{1,\dots,l\}\backslash\{i_{\max}\}}|T_i|\geq \alpha(n)/l}{\tau\geq \beta n,\:(Z_{\tau-1},-Z_{\tau})=(l,k)}{}.
    \end{eqnarray*}
    Applying Lemma \ref{lem:RW error small}, the first term in the upper bound converges to 0 as $n\rightarrow\infty$ and then $\Lambda\rightarrow\infty$. Hence it will suffice to check that each term of the sum converges to 0 for each fixed $(l,k)$. Now, as in the proof of \eqref{eqn:5.4i}, we have that
\[\prcond{\max_{i\in\{1,\dots,l\}\backslash\{i_{\max}\}} \Gamma_{i}\geq n^{5/8}}{\tau\geq \beta n,(Z_{\tau-1},-Z_{\tau})=(l,k)}{}\leq Cl^3n^{-1/8},\]
    and thus it will be enough to show the decay to 0 of
    \[ \prcond{\max_{i\in\{1,\dots,l\}\backslash\{i_{\max}\}}|T_i|\geq \alpha(n)/l,\:\max_{i\in\{1,\dots,l\}\backslash\{i_{\max}\}} \Gamma_{i}\leq n^{5/8}}{\tau\geq \beta n,\:(Z_{\tau-1},-Z_{\tau})=(l,k)}{}.\]
   Decomposing into the possible values of $(\Gamma_i)$ and using the exchangeability of the excursions, this is bounded above by
   \begin{eqnarray*}
   \lefteqn{l\sum_{\lambda_1,\dots,\lambda_{l-1}=1}^{n^{5/8}}\prcond{\max_{i\in\{1,\dots,l-1\}}|T_i|\geq \alpha(n)/l,\:(\Gamma_i)_{i=1}^{l-1}=(\lambda_i)_{i=1}^{l-1}}{\tau\geq \beta n,\:(Z_{\tau-1},-Z_{\tau})=(l,k)}{}}\\
   &=&l\sum_{\lambda_1,\dots,\lambda_{l-1}=1}^{n^{5/8}}\frac{\prcond{\max_{i\in\{1,\dots,l-1\}}|T_i|\geq \alpha(n)/l,\:(\Gamma_i)_{i=1}^{l-1}=(\lambda_i)_{i=1}^{l-1},\:\Gamma_l\geq \lambda_l}{(Z_{\tau-1},-Z_{\tau})=(l,k)}{}}{\prcond{\Gamma_{1,l}\geq \beta n}{(Z_{\tau-1},-Z_{\tau})=(l,k)}{}},
   \end{eqnarray*}
   where we define $\lambda_l:=\beta n-\sum_{i=1}^{l-1}\lambda_i\geq c n$ for some $c>0$ for all large enough $n$ (depending only on $l$). Applying the conditional independence of the excursions and \eqref{eqn:progeny asymp 2}, this is in turn bounded by
   \begin{eqnarray*}
   \lefteqn{Cl\sum_{\lambda_1,\dots,\lambda_{l-1}=1}^{n^{5/8}}\prcond{\max_{i\in\{1,\dots,l-1\}}|T_i|\geq \alpha(n)/l,\:(\Gamma_i)_{i=1}^{l-1}=(\lambda_i)_{i=1}^{l-1}}{(Z_{\tau-1},-Z_{\tau})=(l,k)}{}}\\
   &\leq &Cl\prcond{\max_{i\in\{1,\dots,l-1\}}|T_i|\geq \alpha(n)/l}{(Z_{\tau-1},-Z_{\tau})=(l,k)}{}.\hspace{100pt}
   \end{eqnarray*}
    Now, the law of the finite random variable $\max_{i\in\{1,\dots,l-1\}}|T_i|$ and the conditioning do not depend on $n$. Hence the above expression converges to 0 as $n\rightarrow\infty$, which completes the proof of the first claim of (ii). The proof of the second claim also follows since $\diam (T_i,d_{T_i}) \leq |T_i|$.
\item Again using that $\prcond{Z_{\tau - 1}-Z_{\tau}> \Lambda}{\tau \geq \beta n}{}$ converges to 0 as $n\rightarrow\infty$ and then $\Lambda\rightarrow\infty$, it will suffice to check that the result when conditioning on the event $\{\tau\geq \beta n,\:(Z_{\tau-1},-Z_{\tau})=(l,k)\}$ for each fixed $(l,k)$. On the latter event, the decoration giving $L_\rho$ has a Boltzmann law $\Pb_{k+l+1,q}$. In particular, under each of these laws, it is a finite graph, and the result readily follows.
\end{enumerate}
\end{proof}

Recall that $\CC_Z$ is the cluster constructed from $Z$ as defined in \cref{def:dec tree construction of cluster}. We root it at the root vertex of the original map. In what follows we let $T$ be the largest of the two-type Galton-Watson trees in the decomposition above, and $\CC^{T}_Z$ denote its decorated (with the graph family in \cref{prop:verify assn}) version, rooted at the vertex joining this decorated tree to the root loop. We also let $\CC_{Z,n}$ and $\CC^{T}_{Z,n}$ denote their appropriate rescaled versions,
\begin{align*}
\CC_{Z,n} = \left( \CC_Z, \frac{\sigma}{\chi\sqrt{\beta n}}d_n, \frac{1}{n}\nu_n, \rho_n \right), \qquad \CC^T_{Z,n} = \left( \CC_Z^T, \frac{\sigma}{\chi\sqrt{\beta n}}d^T_n, \frac{1}{n}\nu^T_n, \rho_n^T \right),
\end{align*}\label{fa CTZ}
which we will sample under the conditioning $\{\tau \geq \beta n\}$ (recall the definitions of $\beta$, $\chi$, and $\sigma$ from \cref{def:scaling consts}). Here $d_n$ and $d_n^T$ denote the graph metric on each space, $\nu_n$ and $\nu_n^T$ denote the counting measure on vertices, and $\rho_n$ and $\rho_n^T$ denote their respective roots.

\begin{prop}\label{prop:dGHP small under coupling}
For any $\varepsilon$, it holds that
\[\lim_{n\rightarrow\infty}\mathbb{P}\left(\dGHP{\CC_{Z,n}, \CC^T_{Z,n}} >\varepsilon\:\vline\:\tau > \beta n\right)=0.\]
\end{prop}
\begin{proof}
We consider both spaces as embedded in $\CC_{Z,n}$ in the canonical way. It follows immediately from \cref{cor:one big tree}(ii),(iii) that the Hausdorff distance between the two spaces and the distance between the roots of the spaces converge to $0$ in probability as $n \to \infty$. It similarly follows that $\frac{1}{n}\nu_n(\CC_{Z,n} \setminus \CC_{Z,n}^T)$ converges to $0$ in probability. Since for any set $A \subset \CC_{Z,n}$ we clearly have that $|\nu_n(A) - \nu_n^T(A)| \leq \nu_n(\CC_{Z,n} \setminus \CC_{Z,n}^T)$, the result follows.
\end{proof}

Finally, we switch the conditioning.

\begin{prop}\label{prop:switch final conditioning}
It holds that, as $n\rightarrow\infty$,
\[\prcond{|\CC_Z| \geq n}{\tau > \beta n}{}=1+o(1)\qquad\text{and}\qquad\prcond{\tau > \beta n}{|\CC_Z| \geq n}{}=1+o(1).\]
\end{prop}
\begin{proof}
Conditionally on $Z$, we can write
\begin{equation}\label{eqn:CZ as sum}
|\CC_Z| = \sum_{i=0}^{\tau} \xi^{(i)} - \delta(Z),
\end{equation}
where the distribution of $\xi^{(i)}$ depends only on the jump of $Z$ at time $i$; precisely, the $\xi^{(i)}$ can be chosen so that they are zero if $Z$ has a positive increment at $i$, and being distributed as $m_{k+1}(G_{k+1})$ if $Z$ has a jump of size $-k$ at $i$. The term $\delta(Z)$ is a non-negative error term to account for the fact that we have a different boundary condition on the root face. We can thus write
\begin{eqnarray*}
\lefteqn{\prcond{\CC_Z \geq n}{\tau > \beta n}{}}\\
 &\geq &\prcond{\sum_{i=0}^{\tau} \xi^{(i)}\geq (1+\varepsilon) n}{\tau > \beta n}{}-\prcond{\delta(Z)\geq \varepsilon n}{\tau > \beta n}{}.
\end{eqnarray*}
Since $\delta(Z)$ is bounded above by the size of the unpercolated decorated root loop, it follows from \cref{lem:RW error small} it can be can be uniformly stochastically dominated by an order $1$ random variable, and hence the second term is negligible in the $n\rightarrow\infty$ limit. As for the first term, we have that for some $c>0$ it is bounded below by
\begin{eqnarray*}
\lefteqn{\prcond{\sum_{i=0}^{\tau} \xi^{(i)}\geq (1+\varepsilon) n}{\tau >(1+3\varepsilon) \beta n}{}\prcond{\tau > (1+3\varepsilon)\beta n}{\tau > \beta n}{}}\\
&\geq &\prcond{\sum_{i=U-\frac{1}{2}(1+2\varepsilon) \beta n}^{U+\frac{1}{2}(1+2\varepsilon) \beta n} \xi^{(i)}\geq (1+\varepsilon) n}{\tau >(1+3\varepsilon) \beta n}{}(1-c\varepsilon),
\end{eqnarray*}
where we have applied \eqref{eqn:progeny asymp 2} to deduce the second inequality, the indices are taken modulo $\tau$ and , conditionally on $\tau$, $U$ is chosen uniformly from the set $\{0, \ldots, \tau\}$. We next apply \eqref{eqn:Vervaat dis} (twice, to both the forward and reverse excursion starting from $U$) along with a union bound to deduce that there exists $C<\infty$ such that
\[\prcond{\sum_{i=U-\frac{1}{2}(1+2\varepsilon) \beta n}^{U+\frac{1}{2}(1+2\varepsilon) \beta n} \xi^{(i)}< (1+\varepsilon) n}{\tau >(1+3\varepsilon) \beta n}{}\leq
C \pr{\sum_{i=0}^{\frac{1}{2}(1+2\varepsilon) \beta n} \xi^{(i)}< \frac{1}{2}(1+\varepsilon) n},\]
where the variables on the right-hand side are defined in terms of an unconditioned version of $Z$. In particular, under this unconditioned law, the $\xi^{(i)}$ have mean $\beta^{-1}$, and so the law of large numbers implies that the upper bound above converges to 0 as $n\rightarrow\infty$. In particular, we have so far shown that $\prcond{\CC_Z \geq n}{\tau > \beta n}{}\geq 1-c\varepsilon$, and since $\epsilon>0$ was arbitrary, this gives the first claim.

For the second claim, we start by noting that it follows from the first claim and \eqref{eqn:progeny asymp 2} that there exists $c>0$ such that
\begin{equation}\label{eqn:C vol tails LB}
\pr{|\CC_Z| \geq n} \geq (1-o(1))\pr{\tau \geq \beta n} \geq cn^{-1/2}.
\end{equation}
Combining this with another application of \eqref{eqn:progeny asymp 2}, we deduce that there exists $C<\infty$ such that
\begin{align*}
\prcond{\tau \in [(1-\epsilon)\beta n, \beta n]}{|\CC_Z| \geq n}{} \leq \frac{\pr{\tau \in [(1-\epsilon)\beta n, \beta n]}}{\pr{|\CC_Z| \geq n}} \leq C\epsilon,
\end{align*}
which implies in turn that
\begin{align}\label{eqn:sigma decomp}
\prcond{\tau > \beta n}{|\CC_Z| \geq n}{} \geq \prcond{\tau > (1-\epsilon)\beta n}{|\CC_Z| \geq n}{} - C\epsilon.
\end{align}
Moreover, for the complement of the latter probability we can write
\begin{align}\label{eqn:lifetime small Bayes}
\prcond{\tau \leq (1-\epsilon)\beta n}{|\CC_Z| \geq n}{} = \frac{\pr{\tau \leq (1-\epsilon)\beta n, |\CC_Z| \geq n}}{\pr{|\CC_Z| \geq n}}.
\end{align}
Recall from \eqref{eqn:CZ as sum} that $|\CC_Z|$ is bounded above by $\sum_{i=0}^{\tau} \xi^{(i)}$. In fact the law of $\xi^{(i)}$ also implicitly depends on $\tau$ through the dependence on $Z_i-Z_{i-1}$. To make this precise, we write $\xi^{(i, \tau)}$ in place of $\xi^{(i)}$ in what follows, and reserve the notation $\xi^{(i)}$ for the variables associated with an unconditioned random walk. By the Vervaat transform estimate of \eqref{eqn:Vervaat dis}, it therefore follows that there exists $C<\infty$ such that
\begin{align*}
\lefteqn{\pr{\tau \leq (1-\epsilon)\beta n, |\CC_Z| \geq n}}\\
&=\sum_{m=1}^{(1-\varepsilon)\beta n}
\prcond{|\CC_Z| \geq n}{\tau=m}{}\pr{\tau=m}\\
&\leq \sum_{m=1}^{(1-\varepsilon)\beta n}
\prcond{\sum_{i=0}^{\tau} \xi^{(i, \tau)}\geq n}{\tau=m}{}\pr{\tau=m}\\
&\leq \sum_{m=1}^{(1-\varepsilon)\beta n}
\left(\prcond{\sum_{i=U}^{U+\frac{1}{2}m} \xi^{(i,\tau)}\geq n/2}{\tau=m}{}+\prcond{\sum_{i=U-\frac{1}{2}m}^U \xi^{(i,\tau)}\geq n/2}{\tau=m}{}\right)\pr{\tau=m}\\
&\leq C\sum_{m=1}^{(1-\varepsilon)\beta n}
\pr{\sum_{i=0}^{m/2} \xi^{(i)}\geq n/2}\pr{\tau=m}\\
&\leq C\pr{\sum_{i=0}^{(1-\epsilon)\beta n/2} \xi^{(i)}\geq n/2}.
\end{align*}
In particular, in the sums within the probabilities on the fourth line above, we consider the indices $i$ modulo $m$, and we use that the time reversal of a random walk excursion is also a random walk excursion to apply \eqref{eqn:Vervaat dis} to each of the relevant pieces. In the final line we use the fact that the $\xi^{(i)}$ are non-negative to stochastically dominate by the case $m=(1-\epsilon)\beta n/2$. Without the conditioning, the $(\xi^{(i)})_{i\geq 0}$ form an i.i.d.\ sequence, with mean given by $\beta^{-1}$ and finite variance (by the fact that Assumptions \ref{assn:offspring} and \ref{assn:inserted graph general assumption} are both satisfied in the present setting). Therefore, by Chebyshev's inequality there exists a constant $c<\infty$ such that
\begin{align*}
\pr{\sum_{i=0}^{(1-\epsilon)\beta n/2} \xi^{(i)} \geq n/2} \leq \frac{\beta n\Var{\xi^{(1)}}/2}{(\epsilon n/2)^2} = \frac{c}{n\epsilon^2}.
\end{align*}
Substituting back into \eqref{eqn:lifetime small Bayes} and recalling \eqref{eqn:C vol tails LB}, we deduce that there exists $c<\infty$ such that
\begin{align*}
\prcond{\tau \leq (1-\epsilon)\beta n}{|\CC_Z| \geq n}{} \leq \frac{c}{n^{1/2}\epsilon^2}.
\end{align*}
Hence from \eqref{eqn:sigma decomp} we obtain that
\begin{align*}
\prcond{\tau > \beta n}{|\CC_Z| \geq n}{} \geq  1- \frac{c}{n^{1/2}\epsilon^2} - C\epsilon.
\end{align*}
Again, since $\epsilon>0$ was arbitrary, this gives the second claim.
\end{proof}

We conclude with the proof of \cref{thm:main scaling lim}.

\begin{proof}[Proof of \cref{thm:main scaling lim}]
It follows directly from \cref{prop:GHP convergence of enriched trees} and the earlier results of this section (i.e.\ Proposition \ref{prop:verify assn}, Lemmas \ref{lem:RW error small} and \ref{lem:exc to tree}, and Corollary \ref{cor:one big tree}), that, when conditioned on $\{\tau\geq \beta n\}$, $\CC_{Z,n}^T$ converges to $(\T, d_\T, \nu_\T, \rho_\T)$ in distribution with respect to the GHP topology as $n \to \infty$. Thus, from \cref{prop:dGHP small under coupling}, we have the same result when $\CC_{Z,n}^T$ is replaced by $\CC_{Z,n}$. Moreover, applying \cref{prop:switch final conditioning} and \cref{lem:switch conditioning}, we can replace $\{\tau \geq \beta n\}$ by $\{|\CC_Z| \geq n\}$ in the conditioning. In particular, we have that, when conditioned on $\{|\CC_Z| \geq n\}$, $\CC_{Z,n}$ converges to $(\T, d_\T, \nu_\T, \rho_\T)$ in distribution with respect to the GHP topology as $n \to \infty$. Since $\CC_{Z,n}$ is simply a rescaled copy of $\CC_Z$, which has the same distribution as the percolation cluster of interest by \cref{prop:def36ok}, the latter statement is equivalent to \cref{thm:main scaling lim}.
\end{proof}

\section{Random walk scaling limit}\label{sctn:RW}

The aim of this section is to prove \cref{thm:main RW}. As noted in the introduction, to do this, it is useful to apply the general result on the scaling limits of stochastic processes associated with resistance forms of \autocite{DavidResForms}, cf.\ \autocite{CrHaKu}. In this direction, we first introduce the effective resistance for a simple, finite, connected, unweighted graph $G=(V,E)$. In particular, the effective resistance $R_G=(R_G(x,y))_{x,y\in V}$ on $G$ is defined by setting $R_G(x,x)=0$, $\forall x\in V$, and
\begin{equation}\label{effres}
R_G(x,y)^{-1}:=\inf\left\{\mathcal{E}_G(f):\:f:V\rightarrow \mathbb{R},\:f(x)=0,\:f(y)=1\right\},\qquad \forall x,y\in V,\:x\neq y,
\end{equation}
where
\[\mathcal{E}_G(f):=\frac12\sum_{\substack{x,y\in V:\\\{x,y\}\in E}}\left(f(x)-f(y)\right)^2\]
is the associated electrical energy. (We note that $R_G$ is a metric on $V$, see \autocite[Proposition 3.3]{Kigq}, for example; moreover, it coincides with the standard physical notion of electrical resistance on $G$ and is upper bounded by the graph distance.) Moreover, we will consider the degree measure $\tilde{\nu}_G$ on $V$, obtained by setting
\[\tilde{\nu}_G\left(\{x\}\right):=\#\left\{y:\:\{x,y\}\in E\right\}.\]
(For our decorated tree model, we recall from Remark \ref{listrem}(3) how the degree measure can be constructed from the degree measures of the constituent pieces.) The stochastic process naturally associated with the resistance metric space $(V,R_G,\tilde{\nu}_G)$ is the continuous-time Markov chain $(X_t^G)_{t\geq 0}$ with generator given by
\begin{equation}\label{gen}
\mathcal{L}_G(f)(x):=\frac{1}{\tilde{\nu}_G\left(\{x\}\right)}\sum_{\substack{y\in V:\\\{x,y\}\in E}}\left(f(y)-f(x)\right).
\end{equation}
Since this process has unit mean exponential holding times, its long-time behaviour is similar to that of the discrete-time simple random walk on $G$, and it will be sufficient to consider the scaling limit of this process in order to establish \cref{thm:main RW}.

To apply the results of \autocite{DavidResForms}, we first check that it is possible to replace the shortest-path metric $d_n$ with the effective resistance metric on $C_n$, which we will denote by $R_n$, and the counting measure $\nu_n$ with the degree measure, which we will denote by $\tilde{\nu}_n$, in the results of \cref{sctn:dec tree limits}. Specifically, we have the following.

\begin{prop}\label{resconv} In the setting of Theorem \ref{thm:main scaling lim}, there exist deterministic constants $\delta,\kappa\in (0,\infty)$ and a probability space $(\Omega, \F, \Pb)$ upon which the relevant objects are built such that
\[\sup_{x,y\in C_n}n^{-1/2}\left|d_n(x,y)-\delta^{-1} R_n(x,y)\right|\rightarrow 0\]
and
\[d_{C_n}^P\left(n^{-1}\nu_n,\kappa^{-1} n^{-1}\tilde{\nu}_n\right)\rightarrow 0\]
in probability as $n\rightarrow\infty$, where $d_{C_n}^P$ is the Prohorov distance between measures on $C_n$ and distances are measured with respect to either $n^{-1/2}d_n$ or $n^{-1/2}R_n$.
\end{prop}
\begin{proof}
Let $(G_i)_{i\geq 1}$ be as in the statement of Proposition \ref{prop:verify assn}. Let $R_i$ be the effective resistance metric on $G_i$. Since $R_i(x,y)\leq d_i(x,y)$, $\forall x,y \in G_i$, where $d_i$ is the shortest path metric on $G_i$, we have from Proposition \ref{prop:verify assn} that Assumption \ref{assn:inserted graph general assumption}(1) is satisfied with the metric under consideration being $R_i$. Moreover, let $\tilde{\nu}_i$ be the degree measure on $G_i$. Using the fact that the maps we consider here are triangulations, it is an elementary exercise to check from Euler's polyhedron formula that
\[\tilde{\nu}_i(G_i)\leq 6 \#\tilde{G}_i,\]
where we write $\tilde{G}_i$ for the unpercolated map selected from $\mathbb{P}_{i,q}$. From this and the proof of Proposition \ref{prop:verify assn}, we find that Assumption \ref{assn:inserted graph general assumption}(2) holds for the degree measures. In particular, this confirms Assumption \ref{assn:inserted graph general assumption} is met by the sequence $((G_i,R_i,\tilde{\nu}_i,\ell_i))_{i\geq 1}$.

We next note that, by using the decomposition of the cluster of interest given by Lemma \ref{lem:exc to tree} and basic estimates of Corollary \ref{cor:one big tree} (adapted to the effective resistance metric and degree measures), as well as the change of conditioning facilitated by Proposition \ref{prop:switch final conditioning}, to establish the proposition, it will suffice to show the corresponding result for the decorated tree model, where the offspring distribution is as described in Proposition \ref{prop:tree distribution}, the decorations are defined as in the previous paragraph, and the conditioning is on the length of the excursion encoding the two-type tree having length no less than $\beta n$.

Now, let $T_n$ be the relevant conditioned two-type tree and $\Dec (T_n)$ be the corresponding decorated tree. Write $d_n^*$ for the shortest path metric on $\Dec (T_n)$, and $R_n^*$ be the associated resistance on $\Dec (T_n)$. If $\chi_d$ and $\chi_R$ are the shortest path and effective resistance versions of the constant of Definition \ref{def:scaling consts}(b), then it is straightforward to see that, if $d_n^{\text{tr}}$ is as defined in \cref{prop:dis to 0 dec tree} and the correspondence between $u$ and $u'$ (and $v$ and $v'$) is as in the proof of Lemma \ref{lem:dis Rn3}, then:
\begin{eqnarray*}
\lefteqn{\sup_{x,y\in \Dec(T_n)}\left|\chi_d^{-1}d^*_n(x,y)-\chi_R^{-1} R^*_n(x,y)\right|}\\
&\leq &\sup_{u,v\in T_n}\left|2^{-1}d_n^{\text{tr}}(u,v)-\chi_d^{-1} d^*_n(u',v')\right|\\
&&\qquad+\sup_{u,v\in T_n}\left|2^{-1}d_n^{\text{tr}}(u,v)-\chi_R^{-1} R^*_n(u',v')\right|\\
&&\qquad+\sup_{u\in t_\circ(T_n)}\diam\left(G^{(u)},\chi_R^{-1} R^*_n\right)+\sup_{u\in t_\circ(T_n)}\diam\left(G^{(u)},\chi_d^{-1} d^*_n\right).
\end{eqnarray*}
Since both metrics $R_n^*$ and $d_n^*$ satisfy Assumption \ref{assn:inserted graph general assumption}, we have from Lemma \ref{lem:max diam} and the proof of Lemma \ref{lem:dis Rn3} that the above expression, when divided by $n^{1/2}$, converges to 0 in probability, which establishes the first part of the desired result.

To handle the measure, it will be convenient to work on the probability space where $\Dec (T_n)$ and $\mathcal{T}$ are coupled as in the proof of Proposition \ref{prop:dis to 0 dec tree}. In particular, we showed in the proof of Proposition \ref{prop:GHP convergence of enriched trees} that, on this probability space,
\[d_P^{F_n}\left(\beta n^{-1}\tilde{m}_n,\nu_\mathcal{T}\right)\rightarrow 0\]
in probability, where $(F_n,D_{F_n})$ was the space of the canonical Gromov-Hausdorff embedding, which was defined in \cref{sctn:GHP topology}; in fact this construction can be straightforwardly extended so that we can embed all three of the spaces $\left(\Dec (T_n), \frac{\sigma}{\chi_d\sqrt{n}}d_n^*\right)$, $\left(\Dec (T_n), \frac{\sigma}{\chi_R\sqrt{n}}R_n^*\right)$ and $(\T, d_{\T})$ in $(F_n, D_n)$ such that the Gromov-Hausdorff distance between each pair goes to $0$ in probability. Hence, if $\nu^{\Dec(T_n)}$ is the counting measure on $\Dec (T_n)$ and $\tilde{\nu}^{\Dec (T_n)}$ is the degree measure, then
\begin{align*}
d_P^{F_n}\left(\beta n^{-1}\nu^{\Dec(T_n)},\tilde{\beta} n^{-1}\tilde{\nu}^{\Dec(T_n)}\right) \leq
d_P^{F_n}\left(\beta n^{-1}\nu^{\Dec(T_n)},\nu_\mathcal{T}\right)+
d_P^{F_n}\left(\tilde{\beta} n^{-1}\tilde{\nu}^{\Dec(T_n)},{\nu}_\mathcal{T}\right) \rightarrow 0
\end{align*}
in probability as $n\rightarrow \infty$, where $\beta$ and $\tilde{\beta}$ are the versions of the constant of Definition \ref{def:scaling consts}(a) for the counting and degree measures, respectively. This completes the proof of the second claim in the case where the metric is $D_{F_n}$ and hence also in the case where it is $\frac{\sigma}{\chi_R\sqrt{n}}R_n^*$ or $\frac{\sigma}{\chi_d\sqrt{n}}d_n^*$ (since $D_{F_n}$ is lower bounded by a constant multiple of either of them).
\end{proof}

Putting \cref{thm:main scaling lim} and \cref{resconv} together, it is an almost immediate consequence of \autocite{DavidResForms} that the associated random walks converge. To state the next result, we note that the limiting space $(\mathcal{T},d_\mathcal{T},\nu_\mathcal{T})$ is almost-surely a compact resistance metric measure space, and is therefore naturally equipped with a corresponding stochastic process $(X^\mathcal{T}_t)_{t\geq 0}$; see \autocite[Section 6]{DavidCRT} for further details. Moreover, for a stochastic process $X^K$ on a random state-space $K$, equipped with a metric $d_K$, measure $\mu_K$ and distinguished point $\rho_K$, we define the associated annealed law of $X^K$ started from $\rho_K$ to be the probability measure on $\tilde{\mathbb{K}}_c$ given by
\[\mathbb{P}_K\left(\cdot\right):=\int P^K_{\rho_K}\left((K,d_K,\mu_K,\rho_K,X^K)\in \cdot\right)\mathbb{P}\left(d(K,d_K,\mu_K,\rho_K)\right),\]
where $\mathbb{P}$ is the probability measure under which $(K,d_K,\mu_K,\rho_K)$ is selected, and, for a particular realisation of $K$, $P^K_{\rho_K}$ is the law of $X^K$ started from $\rho_K$. In all the cases we need to consider here, the latter object is measurable with respect to $\mathbb{P}$ (see \cite[Proposition 6.1]{Noda}, for example) and so there is no issue in defining the integral.

\begin{cor}\label{RWcor} For a given realisation of $C_n$, let $X^{C_n,\mathrm{cont}}$ be the continuous-time Markov process on $C_n$, with generator defined as at \eqref{gen}. It is then the case that, as $n\rightarrow\infty$, the annealed laws of
\[\left(C_n,\frac{1}{\delta\gamma\sqrt{n}}R_n,\frac{1}{\kappa n}\tilde{\nu}_n,\rho,\left(X^{C_n,\mathrm{cont}}_{\theta n^{3/2}t}\right)_{t\geq 0}\right)\]
converge weakly as probability measures on $\tilde{\mathbb{K}}_c$ to the annealed law of
\[\left(\mathcal{T},d_\mathcal{T},\nu_\mathcal{T},\rho_\mathcal{T},X^\mathcal{T}\right),\]
where the constants $\delta$, $\gamma$ and $\kappa$ are as given by \cref{thm:main scaling lim} and \cref{resconv}, and $\theta:=\delta\gamma\kappa$.
\end{cor}
\begin{proof}
It readily follows from \cref{resconv} that, on the same probability space $(\Omega, \F, \Pb)$,
\[d_{GHP}\left(\left(C_n,\frac{1}{\gamma\sqrt{n}}d_n,\frac{1}{n}{\nu}_n,\rho_n\right),\left(C_n,\frac{1}{\delta\gamma\sqrt{n}}R_n,\frac{1}{\kappa n}\tilde{\nu}_n,\rho_n\right)\right)\rightarrow0\]
in probability. In conjunction with \cref{thm:main scaling lim}, this implies that
\[\left(C_n,\frac{1}{\delta\gamma\sqrt{n}}R_n,\frac{1}{\kappa n}\tilde{\nu}_n,\rho_n\right)\]
converges in distribution to $(\mathcal{T},d_\mathcal{T},\nu_\mathcal{T},\rho_\mathcal{T})$ with respect to the Gromov-Hausdorff-Prohorov topology. Moreover, since the underlying topological space is separable, we can apply Skorohod's representation theorem to suppose that versions of all the spaces in question are built on the same probability space so that the convergence holds almost-surely. Thus we can appeal to \autocite[Theorem 1.2]{DavidResForms} (and its proof) to deduce that, almost-surely, one can find a common metric space $(M,d_M)$ into which $(C_n,\frac{1}{\delta\gamma\sqrt{n}}R_n)$ and $(\mathcal{T},d_{\mathcal{T}})$ can be isometrically embedded so that
\[d_M^H\left(C_n,\mathcal{T}\right)\rightarrow 0,\qquad d_M^P\left(\frac{1}{\kappa n}\tilde{\nu}_n,\nu_\mathcal{T}\right)\rightarrow 0,\qquad d_M(\rho,\rho_\mathcal{T})\rightarrow0,\]
and also
\[P^{C_n,\mathrm{cont}}_\rho\left(\left(X^{C_n,\mathrm{cont}}_{\theta n^{3/2}t}\right)_{t\geq 0}\in\cdot\right)\rightarrow P_{\rho_\mathcal{T}}^\mathcal{T},\]
weakly as probability measures on $D(\mathbb{R}_+,M)$, where we identify the various objects and their embeddings. By the definition of the extended Gromov-Hausdorff-Prohorov topology, this yields in turn that, almost-surely with respect to the coupling measure,
\[P^{C_n,\mathrm{cont}}_\rho\left(\left(C_n,\frac{1}{\delta\gamma\sqrt{n}}R_n,\frac{1}{\kappa n}\tilde{\nu}_n,\rho,\left(X^{C_n,\mathrm{cont}}_{\theta n^{3/2}t}\right)_{t\geq 0}\right)\in\cdot\right)\rightarrow P_{\rho_\mathcal{T}}^\mathcal{T}\left(\left(\mathcal{T},d_\mathcal{T},\nu_\mathcal{T},\rho_\mathcal{T},X^\mathcal{T}\right)\in\cdot\right),\]
weakly as probability measures on $\tilde{\mathbb{K}}_c$. Taking expectations with respect to the random environments $C_n$ and $\mathcal{T}$ concludes the proof.
\end{proof}

Finally, we can state and prove a precise version of \cref{thm:main RW}.

\begin{theorem}\label{RWprecise}
As $n\rightarrow\infty$, the annealed laws of
\[\left(C_n,\frac{1}{\gamma\sqrt{n}}d_n,\frac{1}{n}{\nu}_n,\rho,\left(X^{C_n}_{\lfloor\theta n^{3/2}t\rfloor}\right)_{t\geq 0}\right)\]
converge weakly as probability measure on $\tilde{\mathbb{K}}_c$ to the annealed law of
\[\left(\mathcal{T},d_\mathcal{T},\nu_\mathcal{T},\rho_\mathcal{T},X^\mathcal{T}\right),\]
where the constant $\gamma$ is given by \cref{thm:main scaling lim}, and $\theta$ by \cref{RWcor}.
\end{theorem}
\begin{proof} Applying \cref{resconv} and \cref{RWcor}, it is possible to establish that the result is true with $X^{C_n,\mathrm{cont}}$ in place of $X^{C_n}$. Using that the holding times of $X^{C_n,\mathrm{cont}}$ are unit mean exponentials, replacing it with the discrete-time process $X^{C_n}$ is straightforward (see \cite[Section 4.2]{archerlooptrees} for an example of such an argument).
\end{proof}

\begin{rmk}\label{rmk:transition densities etc}
(a) In addition to the convergence of the stochastic processes $X^{C_n}$, one may deduce convergence of the corresponding transition densities and mixing times by applying the general results of \autocite{CHllt,CHKmix}. In particular, \cite{CHllt} explains that the additional requirement for the convergence of transition densities is an equicontinuity property for the latter objects, and shows how this is readily checked in the resistance form setting, as is applicable here. With the convergence of transition densities, the main result of \cite{CHKmix}, which concerns mixing time convergence, can then be applied. Whilst we do not present the details here, in both \autocite{CHllt} and \autocite{CHKmix}, graph trees converging to the continuum random tree are considered, and the basic arguments of those examples will transfer to the present setting.\\
(b) Another quantity associated with a random walk that is natural to study is the corresponding local time. In the case of resistance spaces, convergence of local times was established in \cite{CrHaKu} under a uniform volume doubling condition, which is too strong to apply in many examples of random graphs, including that of this work. On the other hand, a stronger result was derived in \cite{Noda}, which gives a condition for the convergence of local times based on the metric entropy of the spaces in question. As per \cite[Corollary 7.3]{Noda}, to apply the main result of \cite{Noda}, it would suffice to check that: for any fixed $\epsilon>0$, there exists $c_{\epsilon}<\infty$ such that (defining balls $B(x,r)$ with respect to resistance metric)
\begin{align*}
\liminf_{n \to \infty} \pr{\inf_{x \in C_n} \frac{1}{n}\tilde{\nu}_n(B(x,2u \sqrt{n})) > c_{\epsilon}u^3 \text{ for all } u \in \left(\frac{1}{(\log n)^4},2\right)}\geq 1-\epsilon.
\end{align*}
Due to the current length of the article, we will not give a detailed proof of this estimate, but rather sketch how one might proceed to prove it. By \cref{cor:one big tree}, the above bound is implied by the analogous result for the corresponding decorated tree model (where balls are still defined with respect to the resistance metric): for any fixed $\epsilon>0$, there exists $c_{\epsilon}<\infty$ such that
\begin{align}\label{eqn:dec tree local time prob condition}
\liminf_{n \to \infty} \pr{\inf_{x \in \Dec(T_n)} \frac{1}{n}{\nu}^{\Dec(T_n)}(B(x,u \sqrt{n})) > c_{\epsilon}u^3 \text{ for all } u \in \left(\frac{1}{(\log n)^4},2\right)} \geq 1-\epsilon,
\end{align}
where we have replaced the degree measure by the counting measure, which is clearly smaller. To establish this, we note that we obtain from the resistance version of \eqref{eqn:dstar good} that
\[\pr{\exists u, v \in t_{\bullet}(T_n) \text{ such that }  \left|\chi_R^{-1}R_n^*(u,v) - 2^{-1}d_n^{\text{tr}}(u,v)\right| > \epsilon d_n^{\text{tr}}(u,v) \vee n^{1/3}} \to 0.\]
Together with \cref{lem:max diam}, this implies that to check \eqref{eqn:dec tree local time prob condition}, it would be sufficient to verify the corresponding result for $t_\bullet(T_n)$, equipped with $d_n^{\text{tr}}$ (which is proportional to its shortest path metric) and the counting measure on this space. Now, in the unconditioned model, the black vertices form a one-type Galton-Watson tree, and for such an object conditioned to be large (i.e.\ having size equal to $n$, with $n\rightarrow\infty$), the relevant volume bound was given in \autocite[Proposition 8.10]{Noda}. So, the final step would be to justify the change in conditioning from $t_\bullet(T)=n$ to $\#T \geq n$, but we would not expect any fundamental issues to arise in doing this.
\end{rmk}

\appendix

\section{List of notation}\label{sec:notation}

\begin{center}
\begin{longtable}{p{0.22\textwidth}|p{0.5\textwidth}|p{0.12\textwidth}}
\textbf{Notation} & \textbf{Definition} & \textbf{Initial appearance}\\
\hline
$\Ha$ & law of map of the half-plane with parameter $\alpha$ & p. \pageref{fa Ha} \\
\hline
$p, p_c$ & site percolation probability, and its critical value & p. \pageref{pcdef} \\
\hline
$\left(\C_n, \frac{ 1}{\gamma\sqrt{n}}d_n, \frac{1}{n}\nu_n, \rho\right)$ & rescaled percolation cluster, conditioned to have at least $n$ vertices & p. \pageref{thm:main scaling lim}  \\
\hline
$\T$ or $(\T, d_\T, \nu_\T, \rho_\T)$ & CRT with mass at least $1$ & p. \pageref{thm:main scaling lim}  \\
\hline
$\N$ & the natural numbers $\{1,2,\ldots \}$ & p. \pageref{sctn:GW background} \\
\hline
$k_u$ & number of offspring of a vertex $u$ in a Galton-Watson tree & p. \pageref{fa ku}  \\
\hline
$t_{\bullet}$ & collection of black vertices of a tree $t$ (at even generations) & p. \pageref{fa mu t}  \\
\hline
$t_{\circ}$ & collection of white vertices of a tree $t$ (at odd generations) & p. \pageref{fa mu t}  \\
\hline
$\mu_{\bullet}, \hat{\mu}_{\bullet}$ & offspring distribution for a two-type Galton-Watson tree at even generations, and size-biased version & p. \pageref{fa mu t}  \\
\hline
$\mu_{\circ}, \hat{\mu}_{\circ}$ & offspring distribution for a two-type Galton-Watson tree at odd generations, and size-biased version & p. \pageref{def:Kesten tree} \\
\hline
$\B$ & Brownian excursion of lifetime at least $1$ & p. \pageref{fa exc defs} \\
\hline
$\zeta$ & the lifetime of a Brownian excursion or a random walk excursion & p. \pageref{fa exc defs} \\
\hline
$\pi$ & the projection from $[0, \zeta]$ to $\T$ & p. \pageref{eqn:CRT def} \\
\hline
$t$, $\Loop (t)$ & two-type plane tree and its associated looptree & p. \pageref{sctn:looptrees} \\
\hline
$Z$ & an excursion satisfying $Z_0=1, Z_i \geq 1$ for all $1 \leq i \leq n$ and $Z_n = 1$ & p. \pageref{eqn:contour equiv def} \\
\hline
$\Zr, \mur$ & reversed version of $Z$ and its offspring law & p. \pageref{eqn:mu def general case background reversed} \\
\hline
$L_Z$ & two-type looptree coded by $Z$ & p. \pageref{eqn:mu def general case background reversed} \\
\hline
$T_Z$ & two-type tree coded by $Z$ & p. \pageref{eqn:mu def general case background reversed} \\
\hline
$\mu$ & step distribution for $Z$, defined in \eqref{eqn:mu def} & p. \pageref{eqn:mu def} \\
\hline
$L_Z^*, T_Z^*$ & deterministically extended two-type looptree and tree coded by $Z$ & p. \pageref{fa extended} \\
\hline
$(G_n, d_n, m_n, \ell_n)$ & decorating graph where $G_n$ has boundary length $n$ & p. \pageref{fa dec graph} \\
\hline
$\Pb_n$ & law of $(G_n, d_n, \nu_n, \ell_n)$ & p. \pageref{fa dec graph} \\
\hline
$(\Dec (t), d, m, \rho)$ & decorated version of a tree $t$, with root $\rho$, metric $d$ and measure $\nu$ & p. \pageref{fa dec tree} \\
\hline
$G^{(u)}$ & graph in $\Dec (T)$ inserted at vertex $u \in T$ & p. \pageref{fa dec tree} \\
\hline
$\X$ & random walk excursion & p. \pageref{eqn:Vervaat dis} \\
\hline
$\dGHP{ \cdot, \cdot }, d_{\tilde{\mathbb{K}}_c}( \cdot, \cdot )$ & pointed Gromov-Hausdorff-Prohorov distance and extension with \cadlag paths & p. \pageref{eqn:GHP def} \\
\hline
$\prstart{\cdot}{m,q}$ & $q$-Boltzmann measure on triangulations of the $m$-gon & p. \pageref{fa dKc} \\
\hline
$M$ & map with law $\Ha$ & p. \pageref{fa M} \\
\hline
$p_i$ & probability that peeled vertex is on boundary at distance $i$ from root edge & p. \pageref{prop:peeling probs} \\
\hline
$(B_i)_{i \geq 0}$ & evolution of the black boundary length in a peeling exploration & p. \pageref{fa B} \\
\hline
$\hat{T}$ & termination time of the peeling exploration & p. \pageref{fa B} \\
\hline
$\mu$ & jump distribution for truncated peeling exploration $Z$ (these have laws of $Z, \mu$ above) & p. \pageref{eqn:mu def} \\
\hline
$\N_0$ & $\N \cup \{0\}$ & p. \pageref{eqn:fa No} \\
\hline
$\CC, \CC_Z$ & critical cluster discovered during a peeling exploration, and its construction via $Z$ & p. \pageref{prop:def36ok} \\
\hline
$\beta, \sigma, \chi$ & scaling constants for the volume and metric & p. \pageref{def:scaling consts} \\
\hline
$(\Dec (T_n), \td_n, \tnu_n, \rho_n)$ & decorated tree coded by an excursion of length at least $n$ & p. \pageref{prop:GHP convergence of enriched trees} \\
\hline
$\mathcal{R}_n$ & canonical correspondence between rescaled $\Dec (T_n)$ and the CRT using coding functions & p. \pageref{def:correspondence} \\
\hline
$\mathcal{R}^*_n, \mathcal{R}^{**}_n, \mathcal{R}^{***}_n$ & intermediate correspondences used in proofs & pp. \pageref{fa Rnstar}, \pageref{fa Rnstarrr} \\
\hline
$d_n^{\textsf{tr}}$ & tree distance on $T_n$ & p. \pageref{fa dtr} \\
\hline
$I_k^{\epsilon}$ & the interval $[U+k\epsilon, U+(k+1)\epsilon]$ where $U \sim \textsf{Uniform}([0,\epsilon])$ & p. \pageref{fa IKu} \\
\hline
$X_u$ & copy of $\nu_{\deg u} (G_{\deg u})$ replacing $u$ & p. \pageref{eqn:def nu n star} \\
\hline
$\nu^*_n$ & measure on real line determined by $\Dec (T_n)$, i.e.\ $\nu^*_n(A) = \sum_{i \in \N: \frac{i}{n} \in A} \mathbbm{1}\{u_i \in t_{\circ}\} \nu (X_{u_i})$ & p. \pageref{eqn:def nu n star} \\
\hline
$\tilde{\pi}_n$ & pushforward to $\Dec (T_n)$ via $\mathcal{R}_n$: $\tilde{\pi}_n(A) = \{x \in \Dec (T_n): \exists s \in A \text{ with } (x,\pi(s)) \in \mathcal{R}_n\}$ & p. \pageref{fa pin tilde} \\
\hline
$(F_n, D_{F_n})$ & space of canonical GH embedding & p. \pageref{fa canonical GHP} \\
\hline
$D_n$ & the largest rescaled diameter of an inserted graph & p. \pageref{fa diam Dn} \\
\hline
$\tau$ & $\inf\{ n \geq 0: Z_n \leq 0\}$ & p. \pageref{fa tau} \\
\hline
$\CC^{T}_Z$ & decorated version of largest GW tree attached to root loop of $\CC$ & p. \pageref{fa CTZ} \\
\hline
$R_G$ & effective resistance on a graph $G$ & p. \pageref{effres} \\
\end{longtable}
\end{center}

\printbibliography

\end{document}